\documentclass[12pt,reqno]{amsart}
\usepackage{amsmath,a4wide}
\usepackage{xfrac}
\usepackage{stmaryrd,mathrsfs,bm,amsthm,mathtools,yfonts,amssymb,color}
\usepackage{xcolor}

\begingroup
\newtheorem{theorem}{Theorem}[section]
\newtheorem{lemma}[theorem]{Lemma}
\newtheorem{proposition}[theorem]{Proposition}
\newtheorem{corollary}[theorem]{Corollary}
\endgroup

\theoremstyle{definition}
\newtheorem{definition}[theorem]{Definition}
\newtheorem{remark}[theorem]{Remark}
\newtheorem{ipotesi}[theorem]{Assumption}

\numberwithin{equation}{section}

\newcommand\supp{{\rm spt}}

\newcommand\res{\mathop{\hbox{\vrule height 7pt width .3pt depth 0pt
\vrule height .3pt width 5pt depth 0pt}}\nolimits}
\newcommand{\im}{{\rm Im}}
\newcommand{\gr}{{\rm Gr}}
\newcommand{\bT}{\mathbf{T}}
\newcommand{\bG}{\mathbf{G}}
\newcommand{\cH}{{\mathcal{H}}}

\newcommand{\p}{{\mathbf{p}}}
\newcommand{\q}{{\mathbf{q}}}

\newcommand{\sW}{{\mathscr{W}}}
\newcommand{\sC}{{\mathscr{C}}}
\newcommand\sS{{\mathscr S}}
\newcommand\sP{{\mathscr P}}

\newcommand{\bef}{\mathbf{f}}
\newcommand{\beg}{\mathbf{g}}
\newcommand{\bGam}{{\bm \Gamma}}
\newcommand\bmo{{\bm m}_0}
\newcommand{\cM}{{\mathcal{M}}}
\newcommand{\bU}{{\mathbf{U}}}
\newcommand{\bL}{{\mathbf{L}}}
\newcommand{\phii}{{\bm{\varphi}}}
\newcommand{\Phii}{{\bm{\Phi}}}

\newcommand{\cV}{{\mathcal{V}}}

\newcommand{\cJ}{{\mathcal{J}}}
\newcommand{\cL}{{\mathcal{L}}}
\newcommand{\cK}{{\mathcal{K}}}
\newcommand{\bS}{{\mathbf{S}}}
\newcommand{\bSig}{{\mathbf{\Sigma}}}

\newcommand{\bE}{{\mathbf{E}}}

\newcommand{\bOmega}{{\mathbf{\Omega}}}

\newcommand{\bB}{{\mathbf{B}}}
\newcommand{\bC}{{\mathbf{C}}}

\newcommand\Z{{\mathbb Z}}

\newcommand\N{{\mathbb N}}

\newcommand\C{{\mathbb C}}
\newcommand\R{{\mathbb R}}

\newcommand{\eps}{{\varepsilon}}
\newcommand{\bA}{\mathbf{A}}

\def\Xint#1{\mathchoice
{\XXint\displaystyle\textstyle{#1}}%
{\XXint\textstyle\scriptstyle{#1}}%
{\XXint\scriptstyle\scriptscriptstyle{#1}}%
{\XXint\scriptscriptstyle\scriptscriptstyle{#1}}%
\!\int}
\def\XXint#1#2#3{{\setbox0=\hbox{$#1{#2#3}{\int}$ }
\vcenter{\hbox{$#2#3$ }}\kern-.6\wd0}}
\def\mint{\Xint-}

\newcommand{\Lip}{{\rm {Lip}}}

\newcommand{\dist}{{\rm {dist}}}

\newcommand\weak{{\rightharpoonup}\,}

\newcommand{\cB}{{\mathcal{B}}}
\newcommand{\cG}{{\mathcal{G}}}

\newcommand{\cT}{{\mathcal{T}}}

\newcommand{\mass}{{\mathbf{M}}}

\newcommand{\Iqs}{{\mathcal{A}}_Q(\R^{n})}
\newcommand{\Iq}{{\mathcal{A}}_Q}
\def\a#1{\left\llbracket{#1}\right\rrbracket}

\newcommand{\D}{\textup{Dir}}
\newcommand{\de}{\partial}

\newcommand{\etaa}{{\bm{\eta}}}
\newcommand{\ph}{\varphi}

\newcommand{\osc}{{\textup{osc}}}

\newcommand\B{{\mathbf{B}}}
\newcommand{\bh}{\mathbf{h}}

\title[Center manifold]{Regularity of area minimizing currents II:\\ center manifold}
\author{Camillo De Lellis}
\address{Mathematik Institut der Universit\"at Z\"urich}
\email{delellis@math.uzh.ch}

\author{Emanuele Spadaro}
\address{Max-Planck-Institut f\"ur Mathematik in den Naturwissenschaften, Leipzig}
\email{spadaro@mis.mpg.de}

\subjclass[2010]{49Q15, 49N60, 49Q05}
\keywords{Integer rectifiable currents, Regularity, Area minimizing}

\begin{document}

\begin{abstract}
This is the second paper of a series of three on the regularity of higher codimension
area minimizing integral currents.
Here we perform the second main step in the analysis of the singularities,
namely the construction of a \textit{center manifold}, i.e.~an approximate
average of the sheets of an almost flat area minimizing current. Such a center manifold
is accompanied by a Lipschitz multivalued map on its normal bundle, which
approximates the current with a high degree of accuracy.
In the third and final paper these objects are used to conclude the proof
of Almgren's celebrated dimension bound on the singular set.
\end{abstract}

\maketitle

\section{Introduction}

In this second paper on the regularity of area minimizing integer rectifiable
currents (we refer to the Foreword of \cite{DS3} for the precise statement of the final
theorem and on overview of its proof) we address one of the main steps in the analysis of the singularities,
namely the construction of what Almgren calls the \textit{center manifold}.
Unlike the case of hypersurfaces, singularities in higher codimension currents
can appear as ``higher order'' perturbation of smooth minimal submanifolds.
In order to illustrate this phenomenon, we can
consider the examples of area minimizing currents induced by
complex varieties of $\C^n$, as explained in the Foreword of \cite{DS3}.
Take, for instance, the complex curve:
\[
\cV := \big\{(z,w): (z-w^2)^2=w^5\big\}\subset\C^2.
\]
The point $0\in\cV$ is clearly a singular point.
Nevertheless, in every sufficiently small neighborhood of the origin,
$\cV$ looks like a small perturbation of the smooth minimal surface $\{z=w^2\}$:
roughly speaking, $\cV = \{z=w^2\pm w^{\sfrac{5}{2}}\}$.
One of the main issues of the regularity of area minimizing currents is to understand this phenomenon of ``higher order singularities''.
Following the pioneering work of Almgren \cite{Alm}, a way to deal with it is to
approximate the minimizing current with the graph of a multiple valued
function on the normal bundle of a suitable, curved, manifold.
Such manifold must be close to the ``average of the sheets'' of the current
(from this the name \textit{center manifold}): the hope is that such a property will guarantee
a singular ``first order expansion'' of the corresponding approximating map.

A ``center manifold'' with such an approximation property
is clearly very far from being uniquely defined
and moreover the relevant estimates are fully justified only by the concluding
arguments, which will appear in \cite{DS5}. 
In this paper, building upon the works \cite{DS1, DS2, DS3}, we
provide a construction of a center manifold $\cM$ and
of an associated approximation of the corresponding area minimizing
current via a multiple valued function $F:\cM \to \Iq(\R^{m+n})$.

The corresponding construction of Almgren is given in \cite[Chapter~4]{Alm}.
Unfortunately, we do not understand this portion of Almgren's monograph deeply enough
to make a rigorous comparison
between the two constructions. Even a comparison between the statements is prohibitive,
since the main ones of Almgren (cf. \cite[4.30 \& 4.33]{Alm}) are rather involved and seem
to require a thorough understanding of most of the chapter (which by itself has the size
of a rather big monograph).
At a first sight, our approach seems to be much simpler and to deliver
better estimates. In the rest of this introduction we will explain some of the main
aspects of our construction.

\subsection{Whitney-type decomposition}
The center manifold is the graph of a classical function over an
$m$-dimensional plane with respect to which the excess
of the minimizing current is sufficiently small.
To achieve a suitable accuracy in the approximation
of the average of the sheets of the current,
it is necessary to define the function at an appropriate scale, which varies locally.
Around any given point such scale is morally
the first at which the sheets of the current cease to be close.
This leads to a Whitney-type decomposition of the reference $m$-plane,
where the refining algorithm is stopped according to three conditions.
In each cube of the decomposition the center manifold is then a smoothing
of the average of the Lipschitz multiple valued approximation constructed in \cite{DS3},
performed in a suitable orthonormal system of coordinates, which changes from
cube to cube.

\subsection{$C^{3,\kappa}$-regularity of $\cM$}
The arguments of \cite{DS5} require that the center manifold is 
at least $C^{3}$-regular.
As it is the case of Almgren's center manifold, we prove actually $C^{3,\kappa}$ estimates, 
which are a natural outcome of some Schauder estimates.
It is interesting to notice that, if the current has multiplicity one everywhere
(i.e., roughly speaking, is made of a single sheet), then the center
manifold coincides with it and, hence, we can conclude
directly a higher regularity than the one given by the usual De Giorgi-type (or Allard-type)
argument.
This is already remarked in the introduction of \cite{Alm} and it has
been proved in our paper \cite{DS-cm} with a relatively simple and short direct argument.
The interested reader might find useful to consult that reference as well,
since many of the estimates of this note appear there in a much
more elementary form.

\subsection{Approximation on $\cM$}
Having defined a center manifold, we then give a multivalued map 
$F$ on its normal bundle which
approximates the current. The relevant estimates on this map and its approximation properties
are then given locally for each cube of the Whitney decomposition used in the construction of the
center manifold. We follow a simple principle:
at each scale where the refinement
of the Whitney decomposition has stopped, the image of such function
coincides (on a large set) with the Lipschitz multiple valued
approximation constructed in \cite{DS3},
i.e.~the same map whose smoothed average has been used to construct
the center manifold.
As a result, the graph of $F$ is well centered, i.e.~the average of $F$
is very close (compared to its Dirichlet energy and its $L^2$ norm)
to be the manifold $\cM$ itself. As far as we understand Almgren is not following this principle
and it seems very difficult to separate his construction of the center manifold from the one
of the approximating map.

\subsection{Splitting before tilting}
The regularity of the center manifold $\cM$
and the centering of the approximating map $F$
are not the only properties needed to
conclude our proof in \cite{DS5}.
Another ingredient plays a crucial role. 
Assume that around a certain point, at all scales larger than a given one,
say $s$, 
the excess decays and the sheets stay very close.
If at scale $s$ the excess is not decaying anymore,
then the sheets must separate as well.
In other words, since the tilting of the current is under control up to scale 
$s$, the current must in some sense ''split before tilting''.
We borrow the terminology from a remarkable work of Rivi\`ere \cite{Ri04},
where this phenomenon was investigated independently of Almgren's monograph in the case
of $2$-dimensional area-minimizing currents. Rivi\`ere's approach relies on a clever ``lower epiperimetric
inequality'', which unfortunately seems limited to the $2$-d context.

\subsection{Acknowledgments.}
The research of Camillo De Lellis has been supported by the ERC grant
agreement RAM (Regularity for Area Minimizing currents), ERC 306247. 
The authors are warmly thankful to Bill Allard and
Luca Spolaor, for several enlightening discussions
and for carefully reading a preliminary version of the paper,
and to Francesco Maggi for many useful comments and
corrections to our previous paper \cite{DS-cm} which have been very
valuable for the preparation of this work.

\section{Construction algorithm and main existence theorem}

The goal of this section is to specify the algorithm
leading to the center manifold. The proofs of the various statements
are all deferred to later sections.

\subsection{Notation, height and excess}  For open balls in $\R^{m+n}$ we use $\B_r (p)$. 
For any linear subspace $\pi\subset \R^{m+n}$, $\pi^\perp$ is
its orthogonal complement, $\p_\pi$ the orthogonal projection onto $\pi$, $B_r (q,\pi)$ the disk $\bB_r (q) \cap (q+\pi)$
and  $\bC_r (p, \pi)$ the cylinder $\{(x+y):x\in \B_r (p), y\in \pi^\perp\}$ (in both cases $q$ is 
omitted if it is the origin and $\pi$ is omitted if it is clear from the context). We also assume that each $\pi$ is {\em oriented} by a
$k$-vector $\vec\pi :=v_1\wedge\ldots \wedge v_k$
(thereby making a distinction when the same plane is given opposite orientations) and with a slight abuse of notation we write $|\pi_2-\pi_1|$ for
$|\vec{\pi}_2 - \vec{\pi}_1|$ (where $|\cdot |$ stands for the norm associated to the usual inner product of $k$-vectors).

A primary role will be played by the $m$-dimensional plane $\mathbb R^m \times \{0\}$ with the standard orientation:
for this plane we use the symbol $\pi_0$ throughout the whole paper.

\begin{definition}[Excess and height]\label{d:excess_and_height}
Given an integer rectifiable $m$-dimensional current $T$ in $\mathbb R^{m+n}$ with finite mass and compact support and $m$-planes $\pi, \pi'$, we define the {\em excess} of $T$ in balls
and cylinders as
\begin{align}
\bE(T,\B_r (x),\pi) &:= \left(2\omega_m\,r^m\right)^{-1}\int_{\B_r (x)} |\vec T - \vec \pi|^2 \, d\|T\|,\\
\qquad \bE (T, \bC_r (x, \pi), \pi') &:= \left(2\omega_m\,r^m\right)^{-1} \int_{\bC_r (x, \pi)} |\vec T - \vec \pi'|^2 \, d\|T\|\, ,
\end{align}
and the {\it height function} in a set $A \subset \R^{m+m}$ as
\[
\bh(T,A,\pi) := \sup_{x,y\,\in\,\supp(T)\,\cap\, A} |\p_{\pi^\perp}(x)-\p_{\pi^\perp}(y)|\, .
\]
\end{definition}

In what follows all currents will have compact support and finite mass and will always be considered as currents
defined in the entire Euclidean space. As a consequence their restrictions to a set $A$ and their pushforward through
a map $\p$ are well-defined as long as $A$ is a Borel set and the map $\p$ is Lipschitz in a neighborhood of their support.

\begin{definition}[Optimal planes]\label{d:optimal_planes}
We say that an $m$-dimensional plane $\pi$ {\em optimizes the excess} of $T$ in a ball $\B_r (x)$ if
\begin{equation}\label{e:optimal_pi}
\bE(T,\B_r (x)):=\min_\tau \bE (T, \B_r (x), \tau) = \bE(T,\B_r (x),\pi).
\end{equation} 
Observe that in general the plane optimizing the excess is not unique and $\bh (T, \bB_r (x), \pi)$ might
depend on the optimizer $\pi$. Since for notational purposes it is convenient to define
a unique ``height'' $\bh (T, \B_r (x))$, we call a plane $\pi$ as in \eqref{e:optimal_pi} {\em optimal} if in addition
\begin{equation}\label{e:optimal_pi_2}
\bh(T,\B_r(x),\pi) = \min \big\{\bh(T,\B_r (x),\tau): \tau \mbox{ satisfies \eqref{e:optimal_pi}}\big\}
= : \bh(T,\B_r(x))\, ,
\end{equation}
i.e. $\pi$ optimizes the height among all planes that optimize the excess. However \eqref{e:optimal_pi_2} does not play
any further role apart from simplifying the presentation. 

In the case of cylinders, instead, $\bE (T, \bC_r (x, \pi))$ will denote $\bE(T, \bC_r (x, \pi), \pi)$ (which coincides with
the cylindrical excess used in \cite{DS3} when $(\p_\pi)_\sharp T \res \bC_r (x, \pi)=
Q \a{B_r (\p_\pi (x), \pi)}$), whereas $\bh (T, \bC_r (x, \pi))$ will be used for $\bh (T, \bC_r (x, \pi), \pi)$. 
\end{definition}

We are now ready to formulate the main assumptions of all the statements in this work.

\begin{ipotesi}\label{ipotesi}
$\eps_0\in ]0,1]$ is a fixed constant and
$\Sigma \subset \bB_{7\sqrt{m}} \subset R^{m+n}$ is a $C^{3,\eps_0}$ $(m+\bar{n})$-dimensional submanifold with no boundary in $\bB_{7\sqrt{m}}$. We moreover assume that, for each $p\in \Sigma$,  $\Sigma$ is the graph of a
$C^{3, \eps_0}$ map $\Psi_p: T_p\Sigma\cap \bB_{7\sqrt{m}} \to T_p\Sigma^\perp$. We denote by $\mathbf{c} (\Sigma)$ the
number $\sup_{p\in \Sigma} \|D\Psi_p\|_{C^{2, \eps_0}}$. 
$T^0$ is an $m$-dimensional integral current of $\mathbb R^{m+n}$ with support in $\Sigma\cap \bar\bB_{6\sqrt{m}}$ and finite mass. It is area-minimizing in $\Sigma$ (i.e.~$\mass (T) \leq \mass (T+ \partial S)$ for any current $S$ with $\supp (S)\subset \Sigma$) and moreover
\begin{gather}
\Theta (0, T^0) = Q\quad \mbox{and}\quad \partial T^0 \res \B_{6\sqrt{m}} = 0,\label{e:basic}\\
\quad \|T^0\| (\B_{6\sqrt{m} \rho}) \leq \big(\omega_m Q (6\sqrt{m})^m + \eps_2^2\big)\,\rho^m
\quad \forall \rho\leq 1,\label{e:basic2}\\
\bE\left(T^0,\B_{6\sqrt{m}}\right)=\bE\left(T^0,\B_{6\sqrt{m}},\pi_0\right),\label{e:pi0_ottimale}\\
\bmo := \max \left\{\mathbf{c} (\Sigma)^2, \bE\left(T^0,\B_{6\sqrt{m}}\right)\right\} \leq \eps_2^2 \leq 1\, .\label{e:small ex}
\end{gather}
$\eps_2$ is a positive number whose choice will be specified in each statement.
\end{ipotesi}
Constants which depend only upon $m,n,\bar{n}$ and $Q$ will be called geometric and usually denoted by $C_0$.

\begin{remark}\label{r:sigma_A}
Note that
\eqref{e:small ex} implies $\bA := \|A_\Sigma\|_{C^0 (\Sigma)}\leq C_0 \bmo^{\sfrac{1}{2}}$,
where $A_\Sigma$ denotes the second fundamental form of $\Sigma$ and $C_0$ is a geometric constant. 
Observe further that for $p\in \Sigma$ the oscillation of $\Psi_p$ is controlled in $T_p \Sigma \cap \bB_{6\sqrt{m}}$ by $C_0 \bmo^{\sfrac{1}{2}}$. 
\end{remark}

In what follows we set $l:= n - \bar{n}$. To avoid discussing domains of definitions it is convenient to extend $\Sigma$ so
that it is an entire graph over all $T_p \Sigma$. Moreover
we will often need to parametrize $\Sigma$ as the graph of a map $\Psi: \mathbb R^{m+\bar n}\to \mathbb R^l$. However
we do not assume that $\mathbb R^{m+\bar n}\times \{0\}$ is tangent to $\Sigma$ at any $p$ and thus we need the
following lemma.

\begin{lemma}\label{l:tecnico3}
There are positive constants $C_0 (m,\bar{n}, n)$ and $c_0 (m, \bar{n}, n)$ such that,
provided $\eps_2 < c_0$, the following holds. If $\Sigma$ is as in Assumption \ref{ipotesi}, then
we can (modify it outside $\bB_{6\sqrt{m}}$ and) extend it to a complete submanifold of $\mathbb R^{m+n}$ which, for every $p\in \Sigma$, is the graph of
a global $C^{3,\eps_0}$ map $\Psi_p : T_p \Sigma \to T_p \Sigma^\perp$ with $\|D \Psi_p\|_{C^{2,\eps_0}}\leq C_0 \bmo^{\sfrac{1}{2}}$.
$T^0$ is still area-minimizing in the extended manifold and in addition
we can apply a global affine isometry which leaves $\mathbb \R^m \times \{0\}$ fixed and maps $\Sigma$ onto
$\Sigma'$ so that
\begin{equation}\label{e:small_tilt}
|\R^{m+\bar{n}}\times \{0\} - T_0 \Sigma'|\leq C_0 \bmo^{\sfrac{1}{2}}\, 
\end{equation}
and $\Sigma'$ is the graph a $C^{3, \eps_0}$ map $\Psi: \mathbb{R}^{m+\bar{n}} \to \mathbb R^l$ with
$\Psi (0)=0$ and
$\|D\Psi\|_{C^{2, \eps_0}} \leq C_0 \bmo^{\sfrac{1}{2}}$.
\end{lemma}

From now on we assume, w.l.o.g. that $\Sigma' = \Sigma$.
The next lemma is a standard consequence of the theory of area-minimizing currents (we include the proofs of Lemma \ref{l:tecnico3} and
Lemma \ref{l:tecnico1} in Section \ref{ss:standard} for the reader's convenience).

\begin{lemma}\label{l:tecnico1}
There are positive constants $C_0 (m,n, \bar{n},Q)$ and $c_0 (m,n,\bar{n}, Q)$ with the following property.
If $T^0$ is as in Assumption \ref{ipotesi}, $\eps_2 < c_0$ and $T:= T^0 \res \bB_{23\sqrt{m}/4}$, then:
\begin{align}
&\partial T \res \bC_{11\sqrt{m}/2} (0, \pi_0)= 0\, ,\quad (\p_{\pi_0})_\sharp T\res \bC_{11 \sqrt{m}/2} (0, \pi_0) = Q \a{B_{11\sqrt{m}/2} (0, \pi_0)}\label{e:geo semplice 1}\\
&\quad\qquad\qquad\mbox{and}\quad\bh (T, \bC_{5\sqrt{m}} (0, \pi_0)) \leq C_0 \bmo^{\sfrac{1}{2m}}\, .\label{e:pre_height}
\end{align}
In particular for each $x\in B_{11\sqrt{m}/2} (0, \pi_0)$ there is  a point
$p\in \supp (T)$ with $\p_{\pi_0} (p)=x$.
\end{lemma}

From now we will always work with the current $T$ of Lemma \ref{l:tecnico1}.
We specify next some notation which will be recurrent in the paper when dealing with cubes of $\pi_0$.
For each $j\in \N$, $\sC^j$ denotes the family of closed cubes $L$ of $\pi_0$ of the form 
\begin{equation}\label{e:cube_def}
[a_1, a_1+2\ell] \times\ldots  \times [a_m, a_m+ 2\ell] \times \{0\}\subset \pi_0\, ,
\end{equation}
where $2\,\ell = 2^{1-j} =: 2\,\ell (L)$ is the side-length of the cube, 
$a_i\in 2^{1-j}\Z$ $\forall i$ and we require in
addition $-4 \leq a_i \leq a_i+2\ell \leq 4$. 
To avoid cumbersome notation, we will usually drop the factor $\{0\}$ in \eqref{e:cube_def} and treat each cube, its subsets and its points as subsets and elements of $\mathbb R^m$. Thus, for the {\em center $x_L$ of $L$} we will use the notation $x_L=(a_1+\ell, \ldots, a_m+\ell)$, although the precise one is $(a_1+\ell, \ldots, a_m+\ell, 0, \ldots , 0)$.
Next we set $\sC := \bigcup_{j\in \N} \sC^j$. 
If $H$ and $L$ are two cubes in $\sC$ with $H\subset L$, then we call $L$ an {\em ancestor} of $H$ and $H$ a {\em descendant} of $L$. When in addition $\ell (L) = 2\ell (H)$, $H$ is {\em a son} of $L$ and $L$ {\em the father} of $H$.

\begin{definition}\label{e:whitney} A Whitney decomposition of $[-4,4]^m\subset \pi_0$ consists of a closed set $\bGam\subset [-4,4]^m$ and a family $\mathscr{W}\subset \sC$ satisfying the following properties:
\begin{itemize}
\item[(w1)] $\bGam \cup \bigcup_{L\in \mathscr{W}} L = [-4,4]^m$ and $\bGam$ does not intersect any element of $\mathscr{W}$;
\item[(w2)] the interiors of any pair of distinct cubes $L_1, L_2\in \mathscr{W}$ are disjoint;
\item[(w3)] if $L_1, L_2\in \mathscr{W}$ have nonempty intersection, then $\frac{1}{2}\ell (L_1) \leq \ell (L_2) \leq 2\, \ell (L_1)$.
\end{itemize}
\end{definition}

Observe that (w1) - (w3) imply 
\begin{equation}\label{e:separazione}
{\rm sep}\, (\bGam, L) := \inf \{ |x-y|: x\in L, y\in \bGam\} \geq 2\ell (L)  \quad\mbox{for every $L\in \mathscr{W}$.}
\end{equation}
However, we do {\em not} require any inequality of the form 
${\rm sep}\, (\bGam, L) \leq C \ell (L)$, although this would be customary for what is commonly 
called a Whitney decomposition in the literature.

\subsection{Parameters} The algorithm for the construction of the center manifold involves several parameters which depend in
a complicated way upon several quantities and estimates. We introduce these parameters and specify some relations among them
in the following

\begin{ipotesi}\label{parametri}
$C_e,C_h,\beta_2,\delta_2, M_0$ are positive real numbers and $N_0$ a natural number for which we assume
always
\begin{gather}
\beta_2 = 4\,\delta_2 = \min \left\{\frac{1}{2m}, \frac{\gamma_1}{100}\right\}, \quad
\mbox{where $\gamma_1$ is the constant of \cite[Theorem~1.4]{DS3},}\label{e:delta+beta}\\
M_0 \geq C_0 (m,n,\bar{n},Q) \geq 4\,  \quad
\mbox{and}\quad \sqrt{m} M_0 2^{7-N_0} \leq 1\, . \label{e:N0}
\end{gather}
\end{ipotesi}

As we can see, $\beta_2$ and $\delta_2$ are fixed. The other parameters are not fixed but are subject to further restrictions in the various statements, respecting the following ``hierarchy''. As already mentioned, ``geometric constants'' are assumed to depend only upon $m, n, \bar{n}$ and $Q$. The dependence of other constants upon the various parameters $p_i$ will be highlighted using the notation $C = C (p_1, p_2, \ldots)$.

\begin{ipotesi}[Hierarchy of the parameters]\label{i:parametri}
In all the statements of the paper
\begin{itemize}
\item[(a)] $M_0$ is larger than a geometric constant (cf. \eqref{e:N0}) or larger than a costant $C (\delta_2)$, see Proposition~\ref{p:splitting};
\item[(b)] $N_0$ is larger than $C (\beta_2, \delta_2, M_0)$ (see for instance \eqref{e:N0} and Proposition \ref{p:compara});
\item[(c)] $C_e$ is larger than $C(\beta_2, \delta_2, M_0, N_0)$ (see the statements
of Proposition \ref{p:whitney}, Theorem \ref{t:cm} and Proposition \ref{p:splitting});
\item[(d)] $C_h$ is larger than $C(\beta_2, \delta_2, M_0, N_0, C_e)$ (see Propositions~\ref{p:whitney} and \ref{p:separ});
\item[(e)] $\eps_2$ is smaller than $c(\beta_2, \delta_2, M_0, N_0, C_e, C_h)$ (which will always be positive).
\end{itemize}
\end{ipotesi}

The functions $C$ and $c$ will vary in the various statements: 
the hierarchy above guarantees however that there is a choice of the parameters for which {\em all} the restrictions required in the statements of the next propositions are simultaneously satisfied. In fact it is such a choice which is then made in \cite{DS5}. To simplify our exposition, for smallness conditions on $\eps_2$ as in (e) we will use the sentence ``$\eps_2$ is sufficiently small''.

\subsection{The Whitney decomposition} 
Thanks to Lemma \ref{l:tecnico1}, for every $L\in \sC$,  we may choose $y_L\in \pi_L^\perp$ so that $p_L := (x_L, y_L)\in \supp (T)$ (recall that $x_L$ is the center of $L$). $y_L$ is in general not unique and we fix an arbitrary choice.
A more correct notation for $p_L$ would be $x_L + y_L$. This would however become rather cumbersome later, when we deal with various decompositions of the
ambient space in triples of orthogonal planes. We thus abuse the notation slightly in using $(x,y)$ instead of $x+y$ and, consistently, $\pi_0\times \pi_0^\perp$ instead of $\pi_0 + \pi_0^\perp$.

\begin{definition}[Refining procedure]\label{d:refining_procedure}
For $L\in \sC$ we set $r_L:= M_0 \sqrt{m} \,\ell (L)$ and 
$\B_L := \bB_{64 r_L} (p_L)$. We next define the families of cubes $\sS\subset\sC$ and $\sW = \sW_e \cup \sW_h \cup \sW_n \subset \sC$ with the convention that
$\sS^j = \sS\cap \sC^j, \sW^j = \sW\cap \sC^j$ and $\sW^j_{\square} = \sW_\square \cap \sC^j$ for $\square = h,n, e$. We define $\sW^i = \sS^i = \emptyset $ for $i < N_0$. We proceed with $j\geq N_0$ inductively: if { no ancestor of $L\in \sC^j$ is in $\sW$}, then 
\begin{itemize}
\item[(EX)] $L\in \sW^j_e$ if $\bE (T, \B_L) > C_e \bmo\, \ell (L)^{2-2\delta_2}$;
\item[(HT)] $L\in \sW_h^j$ if $L\not \in \mathscr{W}_e^j$ and $\bh (T, \B_L) > C_h \bmo^{\sfrac{1}{2m}} \ell (L)^{1+\beta_2}$;
\item[(NN)] $L\in \sW_n^j$ if $L\not\in \sW_e^j\cup \sW_h^j$ but it intersects an element of $\sW^{j-1}$;
\end{itemize}
if none of the above occurs, then $L\in \sS^j$.
We finally set
\begin{equation}\label{e:bGamma}
\bGam:= [-4,4]^m \setminus \bigcup_{L\in \sW} L = \bigcap_{j\geq N_0} \bigcup_{L\in \sS^j} L.
\end{equation}
\end{definition}
Observe that, if $j>N_0$ and $L\in \sS^j\cup \sW^j$, then necessarily its father belongs to $\sS^{j-1}$.

\begin{proposition}[Whitney decomposition]\label{p:whitney}
Let Assumptions \ref{ipotesi} and \ref{parametri} hold and let $\eps_2$ be sufficiently small.
Then $(\bGam, \mathscr{W})$ is a Whitney decomposition of $[-4,4]^m \subset \pi_0$.
Moreover, for any choice of $M_0$ and $N_0$, there is $C^\star := C^\star (M_0, N_0)$ such that,
if $C_e \geq C^\star$ and $C_h \geq C^\star C_e$, then 
\begin{equation}\label{e:prima_parte}
\sW^{j} = \emptyset \qquad \mbox{for all $j\leq N_0+6$.}
\end{equation}
Moreover, the following 
estimates hold with $C = C(\beta_2, \delta_2, M_0, N_0, C_e, C_h)$:
\begin{gather}
\bE (T, \B_J) \leq C_e \bmo^{}\, \ell (J)^{2-2\delta_2} \quad \text{and}\quad
\bh (T, \B_J) \leq C_h \bmo^{\sfrac{1}{2m}} \ell (J)^{1+\beta_2}
\quad \forall J\in \sS\,, \label{e:ex+ht_ancestors}\\
 \bE (T, \B_L) \leq C\, \bmo^{}\, \ell (L)^{2-2\delta_2}\quad \text{and}\quad
\bh (T, \B_L) \leq C\, \bmo^{\sfrac{1}{2m}} \ell (L)^{1+\beta_2}
\quad \forall L\in \sW\, . \label{e:ex+ht_whitney}
\end{gather}
\end{proposition}

\subsection{Construction algorithm} We fix next two important functions $\vartheta,\varrho: \R^m \to \R$.

\begin{ipotesi}\label{mollificatore}
$\varrho\in C^\infty_c (B_1)$ is radial, $\int \varrho =1$ and $\int |x|^2 \varrho (x)\, dx = 0$. For $\lambda>0$ $\varrho_\lambda$ 
denotes, as usual, $x\mapsto \lambda^{-m} \varrho (\frac{x}{\lambda})$.
$\vartheta\in C^\infty_c \big([-\frac{17}{16}, \frac{17}{16}]^m, [0,1]\big)$ is identically $1$ on $[-1,1]^m$.
\end{ipotesi}

$\varrho$ will be used as convolution kernel for smoothing maps $z$ defined on $m$-dimensional planes $\pi$ of $\mathbb R^{m+n}$. In particular, having fixed an isometry $A$ of $\pi$ onto $\mathbb R^m$, the smoothing will be given by $[(z \circ A) * \varrho] \circ A^{-1}$. Observe that since $\varrho$ is radial, our map does not depend on the choice of the isometry
and we will therefore use the shorthand notation $z*\varrho$.

\begin{definition}[$\pi$-approximations]\label{d:pi-approximations}
Let $L\in \sS\cup \sW$ and $\pi$ be an $m$-dimensional plane. If $T\res\bC_{32 r_L} (p_L, \pi)$ fulfills 
the assumptions of \cite[Theorem 1.4]{DS3} in the cylinder $\bC_{32 r_L} (p_L, \pi)$, then the resulting map $f: B_{8r_L} (p_L, \pi)  \to \Iq (\pi^\perp)$ given by \cite[Theorem 1.4]{DS3} is a {\em $\pi$-approximation of $T$ in $\bC_{8 r_L} (p_L, \pi)$}. 
The map $\hat{h}:B_{7r_L} (p_L, \pi) \to \pi^\perp$ given
by $\hat{h}:= (\etaa\circ f)* \varrho_{\ell (L)}$ will be called the {\em smoothed average of the $\pi$-approximation}, where we recall the notation $\etaa \circ f(x) := Q^{-1} \sum_{i=1}^Q f_i(x)$ for any $Q$-valued map $f = \sum_i \a{f_i}$. 
\end{definition}

\begin{definition}[Reference plane $\pi_L$] \label{d:ref_plane}
For each $L\in \sS\cup \sW$ we let $\hat\pi_L$ be an optimal plane in $\bB_L$ (cf. Definition \ref{d:optimal_planes}) and choose an $m$-plane $\pi_L\subset T_{p_L} \Sigma$ which minimizes $|\hat\pi_L-\pi_L|$.
\end{definition}

In what follows we will deal with graphs of multivalued functions $f$ in several system of coordinates. These objects can
be naturally seen as currents $\bG_f$ (see \cite{DS2}) and in this respect we will use extensively the notation and results of \cite{DS2} (therefore $\gr$ will denote the ``set-theoretic'' graph).

\begin{lemma}\label{l:tecnico2}
Let the assumptions of Proposition \ref{p:whitney} hold and assume $C_e \geq C^\star$ and $C_h \geq C^\star C_e$ (where $C^\star$ is the
constant of Proposition \ref{p:whitney}). For any choice of the other parameters,
if $\eps_2$ is sufficiently small, then $T\res \bC_{32 r_L} (p_L, \pi_L)$ satisfies the assumptions of
\cite[Theorem 1.4]{DS3} for any $L\in \sW\cup \sS$. 
Moreover, if $f_L$ is a $\pi_L$-approximation, denote by $\hat{h}_L$ its smoothed average and by $\bar{h}_L$ the map $\p_{T_{p_L}\Sigma} (\hat{h}_L)$,
which takes value in the plane $\varkappa_L := T_{p_L} \Sigma \cap \pi_L^\perp$, i.e. the orthogonal complement of $\pi_L$ in $T_{p_L} \Sigma$.
If we let $h_L$ be the map $x\mapsto h_L (x):= (\bar{h}_L (x), \Psi_{p_L} (x, \bar{h}_L (x)))\in \varkappa_L \times T_{p_L} \Sigma^\perp$, 
then there is a smooth map $g_L: B_{4r_L} (p_L, \pi_0)\to \pi_0^\perp$ such that 
$\bG_{g_L} = \bG_{h_L}\res \bC_{4r_L} (p_L, \pi_0)$.
\end{lemma}

\begin{definition}[Interpolating functions]\label{d:glued}
The maps $h_L$ and $g_L$ in Lemma \ref{l:tecnico2} will be called, respectively, the
{\em tilted $L$-interpolating function} and the {\em $L$-interpolating function}.
For each $j$ let $\sP^j := \sS^j \cup \bigcup_{i=N_0}^j \sW^i$ and
for $L\in \sP^j$ define $\vartheta_L (y):= \vartheta (\frac{y-x_L}{\ell (L)})$. Set
\begin{equation}
\hat\varphi_j := \frac{\sum_{L\in \sP^j} \vartheta_L g_L}{\sum_{L\in \sP^j} \vartheta_L} \qquad \mbox{on $]-4,4[^m$},
\end{equation}
let $\bar{\varphi}_j (y)$ be the first $\bar{n}$ components of $\hat{\varphi}_j (y)$ and define
$\varphi_j (y) = \big(\bar{\varphi}_j (y), \Psi (y, \bar{\varphi}_j (y))\big)$, where $\Psi$ is the map of Lemma \ref{l:tecnico3}.
$\varphi_j$ will be called the {\em glued interpolation} at the step $j$.
\end{definition}

\begin{theorem}[Existence of the center manifold]\label{t:cm}
Assume that the hypotheses of Lemma \ref{l:tecnico2} hold and
let $\kappa := \min \{\eps_0/2, \beta_2/4\}$. For any choice of the other parameters,
if $\eps_2$ is sufficiently small, then
\begin{itemize}
\item[(i)] $\|D\varphi_j\|_{C^{2, \kappa}} \leq C \bmo^{\sfrac{1}{2}}$ and $\|\varphi_j\|_{C^0}
\leq C \bmo^{\sfrac{1}{2m}}$, with $C = C(\beta_2, \delta_2, M_0, N_0, C_e, C_h)$.
\item[(ii)] if $L\in \sW^i$ and $H$ is a cube concentric to $L$ with $\ell (H)=\frac{9}{8} \ell (L)$, then $\varphi_j = \varphi_k$ on $H$ for any $j,k\geq i+2$. 
\item[(iii)] $\varphi_j$ converges in $C^3$ to a map $\phii$ and $\cM:= \gr (\phii|_{]-4,4[^m}
)$ is a $C^{3,\kappa}$ submanifold of $\Sigma$.
\end{itemize}
\end{theorem}

\begin{definition}[Whitney regions]\label{d:cm}
The manifold $\cM$ in Theorem \ref{t:cm} is called
{\em a center manifold of $T$ relative to $\pi_0$} and 
$(\bGam, \sW)$ the {\em Whitney decomposition associated to $\cM$}. 
Setting $\Phii(y) := (y,\phii(y))$, we call
$\Phii (\bGam)$ the {\em contact set}.
Moreover, to each $L\in \sW$ we associate a {\em Whitney region} $\cL$ on $\cM$ as follows:
\begin{itemize}
\item[(WR)] $\cL := \Phii (H\cap [-\frac{7}{2},\frac{7}{2}]^m)$, where $H$ is the cube concentric to $L$ with $\ell (H) = \frac{17}{16} \ell (L)$.
\end{itemize}
\end{definition}

\section{The $\cM$-normal approximation and related estimates}

In what follows we assume that the conclusions of Theorem \ref{t:cm} apply and denote by $\cM$ the corresponding
center manifold. For any Borel set $\cV\subset \cM$ we will denote 
by $|\cV|$ its $\cH^m$-measure and will write $\int_\cV f$ for the integral of $f$
with respect to $\cH^m$. 
$\cB_r (q)$ denotes the geodesic balls in $\cM$. Moreover, we refer to \cite{DS2}
for all the relevant notation pertaining to the differentiation of (multiple valued)
maps defined on $\cM$, induced currents, differential geometric tensors and so on.

\begin{ipotesi}\label{intorno_proiezione}
We fix the following notation and assumptions.
\begin{itemize}
\item[(U)] $\bU :=\big\{x\in \R^{m+n} : \exists !\, y = \p (x) \in \cM \mbox{ with $|x- y| <1$ and
$(x-y)\perp \cM$}\big\}$.
\item[(P)] $\p : \bU \to \cM$ is the map defined by (U).
\item[(R)] For any choice of the other parameters, we assume $\eps_2$ to be so small that
$\p$ extends to $C^{2, \kappa}(\bar\bU)$ and
$\p^{-1} (y) = y + \overline{B_1 (0, (T_y \cM)^\perp)}$ for every $y\in \cM$.
\item[(L)] We denote by $\partial_l \bU := \p^{-1} (\de \cM)$ 
the {\em lateral boundary} of $\bU$.
\end{itemize}
\end{ipotesi}

The following is then a corollary of Theorem \ref{t:cm} and the construction algorithm.

\begin{corollary}\label{c:cover}
Under the hypotheses of Theorem \ref{t:cm} and of Assumption \ref{intorno_proiezione}
we have:
\begin{itemize}
\item[(i)] $\supp (\partial (T\res \bU)) \subset \partial_l \bU$, 
$\supp (T\res [-\frac{7}{2}, \frac{7}{2}]^m \times \R^n) \subset \bU$ 
and $\p_\sharp (T\res \bU) = Q \a{\cM}$;
\item[(ii)] $\supp (\langle T, \p, \Phii (q)\rangle) \subset 
\big\{y\, : |\Phii (q)-y|\leq C \bmo^{\sfrac{1}{2m}} 
\ell (L)^{1+\beta_2}\big\}$ for every $q\in L\in \sW$, where\\
$C= C(\beta_2, \delta_2, M_0, N_0,  C_e, C_h)$;
\item[(iii)]  $\langle T, \p, p\rangle = Q \a{p}$ for every $p\in \Phii (\bGam)$.
\end{itemize}
\end{corollary}

The main goal of this paper is to couple the center manifold of Theorem \ref{t:cm} with a good approximating map defined on it.

\begin{definition}[$\cM$-normal approximation]\label{d:app}
An {\em $\cM$-normal approximation} of $T$ is given by a pair $(\cK, F)$ such that
\begin{itemize}
\item[(A1)] $F: \cM\to \Iq (\bU)$ is Lipschitz (with respect to the geodesic distance on $\cM$) and takes the special form 
$F (x) = \sum_i \a{x+N_i (x)}$, with $N_i (x)\perp T_x \cM$ and $x+N_i (x) \in \Sigma$
for every $x$ and $i$.
\item[(A2)] $\cK\subset \cM$ is closed, contains $\Phii \big(\bGam\cap [-\frac{7}{2}, \frac{7}{2}]^m\big)$ and $\bT_F \res \p^{-1} (\cK) = T \res \p^{-1} (\cK)$.
\end{itemize}
The map $N = \sum_i \a{N_i}:\cM \to \Iq (\R^{m+n})$ is {\em the normal part} of $F$.
\end{definition}

In the definition above it is not required that the map $F$ approximates efficiently the current
outside the set $\Phii \big(\bGam\cap [-\frac{7}{2}, \frac{7}{2}]^m\big)$. However, all the maps constructed
in this paper and used in the subsequent note \cite{DS5} will approximate $T$ with a high degree of accuracy
in each Whitney region: such estimates are detailed
in the next theorem. In order to simplify the notation, we will use $\|N|_{\cV}\|_{C^0}$ (or $\|N|_{\cV}\|_0$) to denote the number
$\sup_{x\in \cV} \cG (N (x), Q\a{0})$.

\begin{theorem}[Local estimates for the $\cM$-normal approximation]\label{t:approx}
Let $\gamma_2 := \frac{\gamma_1}{4}$, with $\gamma_1$ the constant
of \cite[Theorem 1.4]{DS3}.
Under the hypotheses of Theorem \ref{t:cm} and Assumption~\ref{intorno_proiezione},
if $\eps_2$ is suitably small (depending upon all other parameters), then
there is an $\cM$-normal approximation $(\cK, F)$ such that
the following estimates hold on every Whitney region $\cL$ associated to
a cube $L\in \sW$, with constants $C = C(\beta_2, \delta_2, M_0, N_0, C_e, C_h)$:
\begin{gather}
\Lip (N|
_\cL) \leq C \bmo^{\gamma_2} \ell (L)^{\gamma_2} \quad\mbox{and}\quad  \|N|
_\cL\|_{C^0}\leq C \bmo^{\sfrac{1}{2m}} \ell (L)^{1+\beta_2},\label{e:Lip_regional}\\
|\cL\setminus \cK| + \|\bT_F - T\| (\p^{-1} (\cL)) \leq C \bmo^{1+\gamma_2} \ell (L)^{m+2+\gamma_2},\label{e:err_regional}\\
\int_{\cL} |DN|^2 \leq C \bmo \,\ell (L)^{m+2-2\delta_2}\, .\label{e:Dir_regional}
\end{gather}
Moreover, for any $a>0$ and any Borel $\cV\subset \cL$, we have (for $C=C(\beta_2, \delta_2, M_0, N_0, C_e, C_h)$)
\begin{equation}\label{e:av_region}
\int_\cV |\etaa\circ N| \leq 
{ C \bmo \left(\ell (L)^{m+3+\sfrac{\beta_2}{3}} + a\,\ell (L)^{2+\sfrac{\gamma_2}{2}}|\cV|\right)}  + \frac{C}{a} 
\int_\cV \cG \big(N, Q \a{\etaa\circ N}\big)^{2+\gamma_2}\, .
\end{equation} 
\end{theorem}

From \eqref{e:Lip_regional} - \eqref{e:Dir_regional} it is not difficult to infer analogous ``global versions'' of the estimates.

\begin{corollary}[Global estimates]\label{c:globali} Let $\cM'$ be
the domain $\Phii \big([-\frac{7}{2}, \frac{7}{2}]^m\big)$ and $N$ the map of Theorem \ref{t:approx}. Then,  (again with $C = C(\beta_2, \delta_2, M_0, N_0, C_e, C_h)$)
\begin{gather}
\Lip (N|_{\cM'}) \leq C \bmo^{\gamma_2} \quad\mbox{and}\quad \|N|_{\cM'}\|_{C^0}
\leq C \bmo^{\sfrac{1}{2m}},\label{e:global_Lip}\\ 
|\cM'\setminus \cK| + \|\bT_F - T\| (\p^{-1} (\cM')) \leq C \bmo^{1+\gamma_2},\label{e:global_masserr}\\
\int_{\cM'} |DN|^2 \leq C \bmo\, .\label{e:global_Dir}
\end{gather}
\end{corollary}

\section{Additional conclusions upon $\cM$ and the $\cM$-normal approximation}

\subsection{Height bound and separation}
We now analyze more in detail the consequences of the various stopping conditions for the cubes in $\sW$. 
We first deal with $L\in \sW_h$.

\begin{proposition}[Separation]\label{p:separ}
There is a constant $C^\sharp (M_0) > 0$ with the following property.
Assume the hypotheses of Theorem \ref{t:approx} and in addition
$C_h^{2m} \geq C^\sharp C_e$. 
If $\eps_2$ is sufficiently small, then the following conclusions hold for every $L\in \sW_h$:
\begin{itemize}
\item[(S1)] $\Theta (T, p) \leq Q - \frac{1}{2}$ for every $p\in \B_{16 r_L} (p_L)$.
\item[(S2)] $L\cap H= \emptyset$ for every $H\in \sW_n$
with $\ell (H) \leq \frac{1}{2} \ell (L)$;
\item[(S3)] $\cG \big(N (x), Q \a{\etaa \circ N (x)}\big) \geq \frac{1}{4} C_h \bmo^{\sfrac{1}{2m}}
\ell (L)^{1+\beta_2}$  for every { $x\in \Phii (B_{2 \sqrt{m} \ell (L)} (x_L, \pi_0))$}.
\end{itemize}
\end{proposition}

A simple corollary of the previous proposition is the following.

\begin{corollary}\label{c:domains}
Given any $H\in \sW_n$ there is a chain $L =L_0, L_1, \ldots, L_j = H$ such that:
\begin{itemize}
\item[(a)] $L_0\in \sW_e$ and $L_i\in \sW_n$ for all $i>0$; 
\item[(b)] $L_i\cap L_{i-1}\neq\emptyset$ and $\ell (L_i) = \frac{1}{2} \ell (L_{i-1})$ for all $i>0$.
\end{itemize}
In particular,  $H\subset B_{3\sqrt{m}\ell (L)} (x_L, \pi_0)$.
\end{corollary}

We use this last corollary to partition $\sW_n$.

\begin{definition}[Domains of influence]\label{d:domains}
We first fix an ordering of the cubes in $\sW_e$ as $\{J_i\}_{i\in \mathbb N}$ so that their sidelengths do not increase. Then $H\in \sW_n$
belongs to $\sW_n (J_0)$ (the domain of influence of $J_0$) if there is a chain as in Corollary \ref{c:domains} with $L_0 = J_0$.
Inductively, $\sW_n (J_r)$ is the set of cubes $H\in \sW_n \setminus \cup_{i<r} \sW_n (J_i)$ for which there is
a chain as in Corollary \ref{c:domains} with $L_0 = J_r$.
\end{definition}

\subsection{Splitting before tilting I} The following proposition contains a ``typical'' splitting-before-tilting phenomenon: the key assumption of the 
theorem (i.e. $L\in \sW_e$) is that the excess does not decay at some given scale (``tilting'') and the main conclusion \eqref{e:split_2} implies a certain amount of separation between the sheets of the current (``splitting'').

\begin{proposition}(Splitting I)\label{p:splitting}
There are functions $C_1 (\delta_2), C_2 (M_0, \delta_2)$ such that, if $M_0 \geq C_1 ( \delta_2)$, $C_e \geq C_2 (M_0, \delta_2)$, if
the hypotheses of Theorem~\ref{t:approx} hold and if $\eps_2$ is chosen sufficiently small,
then the following holds. If $L\in \sW_e$, $q\in \pi_0$ with $\dist (L, q) \leq 4\sqrt{m} \,\ell (L)$ and $\Omega = \Phii (B_{\ell (L)/4} (q, \pi_0))$, then (with $C, C_3 = C(\beta_2, \delta_2, M_0, N_0, C_e, C_h)$):
\begin{align}
&C_e \bmo \ell(L)^{m+2-2\delta_2} \leq \ell (L)^m \bE (T, \B_L) \leq C \int_\Omega |DN|^2\, ,\label{e:split_1}\\
&\int_{\cL} |DN|^2 \leq C \ell (L)^m \bE (T, \B_L) \leq C_3 \ell (L)^{-2} \int_\Omega |N|^2\, . \label{e:split_2}
\end{align}
\end{proposition}

\subsection{Persistence of $Q$ points} We next state two important properties triggered by the existence of $p\in \supp (T)$ with $\Theta (p, T)=Q$,
both related to the splitting before tilting.

\begin{proposition}(Splitting II)\label{p:splitting_II}
Let the hypotheses of Theorem~\ref{t:cm} hold and assume $\eps_2$ is sufficiently small. For any
$\alpha, \bar{\alpha}, \hat\alpha >0$, there is $\eps_3 = \eps_3 (\alpha, \bar{\alpha}, \hat\alpha, \beta_2, \delta_2, M_0, N_0, C_e, C_h) >0$ as follows. If, for some $s\leq 1$
\begin{equation}\label{e:supL}
\sup \big\{\ell (L): L\in \sW, L\cap B_{3s} (0, \pi_0) \neq \emptyset\big\} \leq s\, ,
\end{equation}
\begin{equation}\label{e:many_Q_points}
\cH^{m-2+\alpha}_\infty \big(\{\Theta (T, \cdot) = Q\}\cap \B_{s}\big) \geq \bar{\alpha} s^{m-2+\alpha},
\end{equation}
and $\min \big\{s, \bmo\big\}\leq\eps_3$, then,
\[
\sup \big\{ \ell (L): L\in \sW_e \mbox{ and } L\cap B_{19 s/16} (0, \pi_0)\neq \emptyset\big\} 
\leq \hat{\alpha} s\, .
\]
\end{proposition}

\begin{proposition}(Persistence of $Q$-points)\label{p:persistence}
Assume the hypotheses of Proposition \ref{p:splitting} hold.
For every $\eta_2>0$ there are $\bar{s}, \bar{\ell} > 0$, depending upon $\eta_2, \beta_2, \delta_2, M_0, N_0, C_e$ and $C_h$, such that,
if $\eps_2$ is sufficiently small, then the following property holds. 
If $L\in \sW_e, \ell (L)\leq \bar\ell$, $\Theta (T, p) = Q$ and
$\dist (\p_{\pi_0} (\p (p)), L) \leq 4 \sqrt{m} \,\ell (L)$, then
\begin{equation}\label{e:persistence}
\mint_{\cB_{\bar{s} \ell (L)} (\p (p))} \cG \big(N, Q \a{\etaa\circ N}\big)^2 \leq \frac{\eta_2}{\ell(L)^{m-2}}
\int_{\cB_{\ell (L)} (\p (p))} |DN|^2\, .
\end{equation}
\end{proposition}

\subsection{Comparison between different center manifolds} 
We list here a final key consequence of the splitting before tilting phenomenon. $\iota_{0,r}$ denotes the map $z\mapsto \frac{z}{r}$.

\begin{proposition}[Comparing center manifolds]\label{p:compara}
There is a geometric constant $C_0$ and a function $\bar{c}_s (\beta_2, \delta_2, M_0, N_0, C_e, C_h) >0$ with
the following property. Assume the hypotheses of Proposition \ref{p:splitting}, $N_0 \geq C_0$, $c_s := \frac{1}{64\sqrt{m}}$
and $\eps_2$ is sufficiently small. If for some $r\in ]0,1[$:
\begin{itemize}
\item[(a)] $\ell (L) \leq  c_s \rho$ for every $\rho> r$ and every
$L\in \sW$ with $L\cap B_\rho (0, \pi_0)\neq \emptyset$;
\item[(b)] $\bE (T, \B_{6\sqrt{m} \rho}) < \eps_2$ for every $\rho>r$;
\item[(c)] there is $L\in \sW$ such that $\ell (L) \geq c_s r $ and $L\cap \bar B_r (0, \pi_0)\neq\emptyset$;
\end{itemize}
then
\begin{itemize}
\item[(i)] the current $T':= (\iota_{0,r})_\sharp T \res \B_{6\sqrt{m}}$ and the submanifold 
$\Sigma':= \iota_{0,r} (\Sigma)\cap \bB_{7\sqrt{m}}$ satisfy the assumptions
of Theorem \ref{t:approx} for
some plane $\pi$ in place of $\pi_0$;
\item[(ii)] for the center manifold $\cM'$ of $T'$ relative to $\pi$ and the 
$\cM'$-normal approximation $N'$ as in Theorem \ref{t:approx}, we have
\begin{equation}\label{e:restart}
\int_{\cM'\cap \bB_2} |N'|^2 \geq \bar{c}_s \max
\big\{\bE (T', \bB_{6\sqrt{m}}), \mathbf{c} (\Sigma')^2\big\}\, . 
\end{equation}
\end{itemize}
\end{proposition}

\section{Center manifold's construction}\label{s:parte_tecnica_1}

In this section we lay down the technical preliminaries to prove Theorem \ref{t:cm}, state
the related fundamental estimates and show how the theorem follows from them. 

\subsection{Technical preliminaries and proof of \eqref{e:prima_parte}}\label{ss:standard}

\begin{proof}[Proof of Lemma \ref{l:tecnico1}]
Recalling that $T:= T^0 \res \bB_{23\sqrt{m}/4}$,
we want to show that \eqref{e:geo semplice 1} hold.
To this regard, we can argue by contradiction. If for instance the second statement in \eqref{e:geo semplice 1}
were false, then we would have a sequence of currents 
$T^0_k$ in $\bB_{6\sqrt{m}}$ and of submanifolds $\Sigma_k$ satisfying
Assumption~\ref{ipotesi} with $\eps_2 (k)\downarrow 0$
and $(\p_{\pi_0})_\sharp T^0_k \res (\bC_{11 \sqrt{m}/2}\cap \bB_{23 \sqrt{m}/4}) \neq Q
\a{B_{11\sqrt{m}/2}}$. On the other hand, from \eqref{e:basic}, \eqref{e:pi0_ottimale}, \eqref{e:small ex} and the standard monotonicity formula
\[
T^0_k \weak T_\infty := Q\a{B_{6\sqrt{m}}}\, .
\]
Also, by standard regularity theory for area minimizing currents, we conclude that
$\supp (T^0_k)\cap \bB_r$ converges to $\supp (T_\infty )\cap \bB_r$ in the Hausdorff distance for every $r<6\sqrt{m}$.
Since $\partial T^0_k$ vanishes in $\bB_{6\sqrt{m}}$, 
$T^0_k\res (\bC_{11 \sqrt{m}/2}\cap \bB_{23 \sqrt{m}/4})$ has no boundary in $\bC_{11\sqrt{m}/2}$ for $k$ large enough, 
thereby implying that 
$(\p_{\pi_0})_\sharp T^0_k\res (\bC_{11 \sqrt{m}/2}\cap  \bB_{23 \sqrt{m}/4})= Q_k \a{B_{11\sqrt{m}/2}}$ for some integer $Q_k$. 
Since $T^0_k \weak T_\infty$, we deduce that $Q_k=Q$ for $k$ large enough, giving
the desired contradiction. Note that the argument actually shows also the first statement in \eqref{e:geo semplice 1}. The height bound \eqref{e:pre_height} now follows
from Theorem~\ref{t:height_bound} because 
$(\p_{\pi_0})_\sharp T^0\res (\bC_{11 \sqrt{m}/2}\cap \bB_{23 \sqrt{m}/4}) = Q \a{B_{11\sqrt{m}/2}}$
and $\Theta (T^0, 0) = Q$: in particular, the latter assumption and Theorem \ref{t:height_bound}(iii) imply
that there is one single open set $\bS_1$ as in Theorem \ref{t:height_bound}(i), which in turn must contain the origin.

By the slicing theory of currents (see \cite[Section 28]{Sim} or \cite[4.3.8]{Fed}) and by
\eqref{e:geo semplice 1},
there is a set $A\subset B_{5\sqrt{m}}$ of 
full measure such that
\[
\langle T, \p_{\pi_0}, x\rangle =
\sum_{i=1}^{N(x)} k_i (x) \delta_{(x, y_i (x))} \qquad \forall x\in A\, ,
\]
where $N(x)\in \mathbb N$, 
 $k_i(x)\in\Z$ with $\sum_{i} {k_i} = Q$, and $(x,y_i (x))\in \supp (T)$ with
 $|y_i (x)|\leq C_0 \bmo^{\sfrac{1}{2m}}$.
By the density of $A$ in $B_{5\sqrt{m}}$,
we conclude 
that $\supp (T) \cap (x+\pi_0^\perp) \neq \emptyset$ for every $x\in \overline{B}_{5\sqrt{m}}$. 
This completes the proof of Lemma \ref{l:tecnico1}. Observe also that as a consequence, if $L\in \sC$, then 
\begin{equation}\label{e:bound_p_L}
|y_L| \leq C \bmo^{\sfrac{1}{2m}}\qquad\mbox{and}\qquad |p_L| \leq 4\sqrt{m} + C_0 \bmo^{\sfrac{1}{2m}}\, 
\end{equation}
(recall that $p_L = (x_L, y_L)\in \pi_0\times \pi_0^\perp \cap \supp (T)$ is the center of $\bB_L$, cf. Definition \ref{d:refining_procedure}).
\end{proof}

\begin{proof}[Proof of Lemma \ref{l:tecnico3}]
The first part of the statement, i.e. the extension of the manifold $\Sigma$, is a fairly standard fact: it suffices to
make the correct extension of the map $\Psi_0$ to $T_0 \Sigma$ and then use the smallness of the norm to show that 
$\Sigma$ is globally graphical over every $T_p \Sigma$. The fact that $T^0$ remains area-minimizing is also fairly simple:
any area minimizing current $T'$ in the extended manifold $\Sigma$ with $T' - T^0 = \partial S$ must be supported in $\bB_{C_0}$ for some geometric constant $C_0$, by the monotonicity formula. On the other hand, for a sufficiently small $\eps_2$, $\bB_r \cap \Sigma$ is geodesically
convex in $\Sigma$ for every $r\in ]0, C_0]$ and thus there is a projection $\p: \bB_{C_0}\cap \Sigma \to \bar\bB_{6\sqrt{m}}\cap \Sigma$ which is $1$-Lipschitz with respect to the Riemannian metric on $\Sigma$. Since $\pi_\sharp T'$ cannot have mass smaller than $T'$, $T'$ must be supported in $\bar \bB_{6\sqrt{m}}$. But then $T'$ is area-minimizing even in the original (i.e. not extended) $\Sigma$ and must have the same mass as $T^0$.

By Assumption \ref{ipotesi} and Remark \ref{r:sigma_A}, $\bA \leq C_0 \bmo^{\sfrac{1}{2}} \leq C_0$. Then,
by the monotonicity formula, $\|T^0\| (\bB_1) \geq c_0>0$ and so there is $p\in \supp (T)\cap \bB_1$ such that
\[
|\vec{T} (p) - \pi_0|^2 = |\vec{T^0} (p) - \pi_0|^2 \leq C_0 \frac{\bE (T^0, \bB_1, \pi_0)}{\|T\| (\bB_1)} \leq C_0 \bmo\, .
\]
We conclude that, if $\eps_2$ is smaller than a geometric constant, $\p_{T_p \Sigma} (\pi_0)$ is an $m$-dimensional plane with
$|\p_{T_p \Sigma} (\pi_0) - \pi_0| \leq C_0 \bmo^{\sfrac{1}{2}}$. On the other hand $|\p_{T_0 \Sigma} - \p_{T_p \Sigma}| \leq C_0 |T_p \Sigma - T_0 \Sigma| \leq C_0 \bA \leq C_0 \bmo^{\sfrac{1}{2}}$ and we conclude $|\p_{T_0 \Sigma} (\pi_0) - \pi_0|\leq C_0 \bmo^{\sfrac{1}{2}}$. Therefore there is an $n$-dimensional plane $\varkappa_0$ orthogonal to  $\pi_0$ such that $|\pi_0\times \varkappa_0 - T_0 \Sigma| \leq C_0 \bmo^{\sfrac{1}{2}}$. We then find a rotation which fixes $\pi_0$ and maps $\varkappa_0$ onto $\{0\}\times \R^n \times \{0\}$. The remaining statements follows easily from Lemma \ref{l:rotazioni_semplici}.
\end{proof}

\begin{proof}[Proof of \eqref{e:prima_parte}]
Fix $L\in \sW^j$ with $N_0 \leq j\leq N_0+6$.
Since $r_L\leq 2^{-7}$ (cf. Assumption \ref{parametri}), \eqref{e:bound_p_L} guarantees
$\B_L\subset \B_{5\sqrt{m}}$ if $\eps_2$ is small enough. Moreover
\[
\bE (T, \bB_L, \pi_0) \leq \frac{6^m}{(64M_0 2^{-N_0-6})^m} \bE (T^0, \bB_{6\sqrt{m}}, \pi_0) \leq 
\frac{6^m}{(64 M_0)^m 2^{-(N_0 +6) m}}\,\bmo\, .
\]
For a suitable $C^\star (M_0, N_0)$ the inequality $C_e \geq C^\star$ implies 
\[
\bE (T, \B_L) \leq 
\bE (T, \B_L, \pi_0)\leq C_e \bmo \,\ell(L)^{2-2\delta_2}\, .
\]
Let now $\hat\pi_L$ be an optimal plane in $\bB_L$: 
since the center $p_L$ belongs to 
$\supp (T)$, by the monotonicity formula $\|T\| (\bB_L) \geq c_0 r_L^m$ 
(cf.~\cite[Section 17]{Sim} or \cite[Appendix A]{DS3}). Thus
\begin{equation}\label{e:primo tilt}
|\hat\pi_L-\pi_0|^2 \leq C_0 \big(\bE (T, \B_L, \pi_0) + \bE (T, \bB_L, \hat{\pi}_L)\big)  \leq C_0 C_e \bmo \, \ell (L)^{2-2\delta_2}\, ,
\end{equation}
where $C_0$ is a geometric constant. This in turn implies that
\begin{align*}
\bh (T, \bB_L) &\leq C_0  M_0 \, |\hat{\pi}_L-\pi_0|\,\ell(L) + \bh (T, \bB_L, \pi_0) \leq 
C_0\, M_0 C_e^{\sfrac{1}{2}} \bmo^{\sfrac{1}{2}} \ell (L)^{2-\delta_2} +
\bh (T, \bC_{5\sqrt{m}})\\
& \stackrel{\mathclap{\eqref{e:pre_height}}}{\leq} C(M_0, N_0) (C_e^{\sfrac{1}{2}} + 1) 
\bmo^{\sfrac{1}{2m}} \ell (L)^{1+\beta_2}\, .
\end{align*}
Thus, if $C^\star (M_0, N_0)$ is chosen sufficiently large 
and $C_h \geq C^\star C_e \geq (C^\star)^2$, neither the condition (EX) nor (HT) apply to $L$.
Therefore,  $\sW^j=\emptyset$ for every $j\leq N_0+6$.
\end{proof}

\subsection{Tilting of planes and proof of Proposition \ref{p:whitney}}
Next we compare optimal planes
and height functions across different cubes of $\sW\cup\sS$.

\begin{proposition}[Tilting of optimal planes]\label{p:tilting opt}
Assume that the hypotheses of Assumptions \ref{ipotesi} and \ref{parametri} hold, that
$C_e \geq C^\star$ and $C_h \geq C^\star C_e$, where $C^\star (M_0, N_0)$ is the constant of the previous section. 
If $\eps_2$ is sufficiently small, then 
\begin{itemize}
\item[(i)] $\bB_H\subset\bB_L \subset \bB_{5\sqrt{m}}$ for all $H, L\in \sW\cup \sS$ with $H\subset L$.
\end{itemize}
Moreover, if $H, L \in \sW\cup\sS$ and either $H\subset L$ or $H\cap L \neq \emptyset$ and $\frac{\ell (L)}{2} \leq \ell (H) \leq \ell (L)$, then the following holds, for $\bar{C} = \bar{C} (\beta_2, \delta_2, M_0, N_0, C_e)$ and $C = C(\beta_2, \delta_2, M_0, N_0, C_e, C_h)$:
\begin{itemize}
\item[(ii)]$|\hat\pi_H - \pi_H| \leq \bar{C} \bmo^{\sfrac{1}{2}} \ell (H)^{1-\delta_2}$;
\item[(iii)] $|\pi_H-\pi_L| \leq \bar{C} \bmo^{\sfrac{1}{2}} \ell (L)^{1-\delta_2}$;
\item[(iv)] $|\pi_H-\pi_0| \leq  \bar C \bmo^{\sfrac{1}{2}}$;
\item[(v)] $\bh (T, \bC_{36 r_H} (p_H, \pi_0)) \leq C \bmo^{\sfrac{1}{2m}} \ell (H)$ and $\supp (T) \cap \bC_{36 r_H} (p_H, \pi_0) \subset \bB_H$; 
\item[(vi)]  For $\pi= \pi_H, \hat{\pi}_H$, $\bh (T, \bC_{36r_L} (p_L, \pi))\leq C \bmo^{\sfrac{1}{2m}} \ell (L)^{1+\beta_2}$
and $\supp (T) \cap \bC_{36 r_L} (p_L, \pi)\subset \bB_L$.
\end{itemize}
In particular, the conclusions of Proposition~\ref{p:whitney} hold.
\end{proposition}

\begin{proof} In this proof we will use the following convention: geometric constants will be denoted by $C_0$ or $c_0$, constants depending upon
$\beta_2, \delta_2, M_0, N_0, C_e$ will be denoted by $\bar{C}$ or $\bar{c}$ and constants depending upon $\beta_2, \delta_2, M_0, N_0, C_e$ and $C_h$ will be denoted by $C$ or $c$.

\medskip

{\bf Proof of (i)--(vi) when $H\subset L$.} The proof is by induction over $i = - \log_2 (\ell (H))$, where we start with
$i = N_0$. For the starting step $i=N_0$ we need to check (i), (ii), (iv), (v) and (vi), all in the special case $H=L$.  Observe first that (i) is a consequence of
\eqref{e:bound_p_L} and the estimate $64 r_L \leq M_0 \sqrt{m} 2^{-N_0} \leq \sqrt{m}/2$.
Since $\sW^{N_0}= \emptyset$, for $i=N_0$ we have $H\in \sS^{N_0}$, which means that $H$ satisfies neither condition (EX) nor condition (HT).
Since by the monotonicity formula
$\|T\| (\bB_{H}) \geq c_0 \, r_{H}^m$, 
there exists at least a point $p\in \supp (T) \cap \B_{{H}}$ such that 
\begin{equation}\label{e:tilt tangent}
|\vec{T} (p) - \hat{\pi}_{{H}}|^2\leq \bE (T, \bB_{{H}})
\frac{C_0\, r_{H}^m}{\|T\| (\bB_{{H}})} 
\leq \bar{C} \bmo \,\ell ({H})^{2-2\delta_2}\, .
\end{equation}
Since $\vec{T} (p)$ is an
$m$-vector of $T_p \Sigma$, this implies that
$|\p_{T_p \Sigma} (\hat\pi_{{H}}) - \hat\pi_{{H}}| \leq \bar{C} \bmo^{\sfrac{1}{2}} \ell ({H})^{1-\delta_2}$.
Recalling that
$|\p_{T_{p_{{H}}} \Sigma} - \p_{T_p \Sigma}| \leq 
C_0 r_{{H}} \bA \leq \bar{C} \bmo^{\sfrac{1}{2}} \ell ({H})$, we conclude (ii).
(iv) follows simply from \eqref{e:primo tilt} and (ii). As for (v), observe that the radius of $\bC_{36 r_{H}}(p_{H}, \pi_{0})$ is smaller than $\sqrt{m}/2$ and its center $p_H = (x_H, y_H)$ satisfies $|x_H|\leq 4\sqrt{m}$. Thus $\bC_{36r_H} (p_H, \pi_0)\subset \bC_{5\sqrt{m}} (0, \pi_0)=: \bC$ and the first conclusion of (v) is a consequence of \eqref{e:pre_height}. The second conclusion follows from the first provided $\eps_2 < c$.
Finally, with regard to (vi), recall that $H=L$. There are two cases: $\pi=\pi_H$ and $\pi = \hat{\pi}_H$. Since the arguments are entirely analogous, we just give it in the first case. The base point $p_H$ of the cylinder $\bC' := \bC_{36r_H} (p_H, \pi_H)$ satisfies, by \eqref{e:pre_height}
$|p_H|\leq 4\sqrt{m} + C_0 \bmo^{\sfrac{1}{2m}}$ and its radius is smaller than $\sqrt{m}/2$. By a simple geometric consideration,  $\bC'\cap \bB_{6\sqrt{m}}\subset \bC$ holds provided $|\pi_H - \pi_0|$ and $|p_H| - 4\sqrt{m}$ are smaller than a geometric constant: this requires $\eps_2 \leq \bar{c}$. Under this assumption $\supp (T) \cap \bC'\subset \bC$ and from \eqref{e:pre_height} and (iv) we conclude $\bh (T, \bC', \pi_H) \leq C_0 |\pi_H - \pi_0| + \bh (T, \bC_{5\sqrt{m}}, \pi_0) \leq \bar{C} \bmo^{\sfrac{1}{2m}}$.
it follows then that $\supp(T) \cap \bC' \subset \bB_{H}$, provided $\eps_2$ is sufficiently small.
Since $H \not\in \sW$, from (ii) we then conclude
\begin{align*}
\bh (T, \bC', \pi_H) &\leq 
\bh (T, \bB_H)+ C_0 M_0 \ell (H) |\pi_H - \hat{\pi}_{H}| &\leq C \bmo^{\sfrac{1}{2m}}
\ell (H)^{1+\beta_2} + \bar{C} \bmo^{\sfrac{1}{2}} \ell (H)^{2-\delta_2}\, .
\end{align*}

Now we pass to the inductive step. Thus fix some $H_{i+1}\in \sS^{i+1} \cup \sW^{i+1}$ and consider a chain $H_{i+1}\subset H_i \subset \ldots \subset H_{N_0}$ with $H_l\in \sS^l$ for $l\leq i$.
We wish to prove all the conclusions (i)--(vi) when $H = H_{i+1}$ and $L = H_j$ for some $j \leq i+1$, recalling that, by inductive assumption, all the statements hold when $H =H_k$ and $L= H_l$ for $l\leq k \leq i$. 
With regard to (i), it is enough to prove that $\bB_{H_{i+1}} \subset \bB_{H_{i}}$. By inductive assumption we know (v) holds with $ H=H_i$, whereas $|x_{H_i}-x_{H_{i+1}}| \leq \sqrt{m}\,\ell(H_i)$: so
$|p_{H_{i+1}}-p_{H_i} | \leq (\sqrt{m} + C \bmo^{\sfrac{1}{2m}}) 2 \ell(H)_{i+1}$.
In particular, for $\eps_2$ small enough, we conclude $|p_{H_{i+1}}-p_{H_i}|\leq 3 \sqrt{m} \ell (H_{i+1})$. 
Assuming that the geometric constant in the first inequality of \eqref{e:N0} is large enough, we infer $\bB_{H_{i+1}} \subset \bB_{H_i}$.
We show now (ii). By (i),
\begin{equation}\label{e:prop whitney1}
\bE(T, \bB_{H_{i+1}}) \leq 2^m \,\bE(T, \bB_{H_{i}}) \leq 2^{m+2-2\delta_2} C_e\, \bmo\,\ell(H_{i+1})^{2- 2\delta_2}\, .
\end{equation}
Therefore, we can argue as above in the case $i = N_0$ to achieve (ii).
We next come to (iii) and (iv). Fix any $l \in \{N_0+1, \ldots, i+1\}$.
By the inclusion in (i), we can argue similarly to infer
\begin{equation}\label{e:tilt opt vicini}
|\hat\pi_{H_{l-1}} - \hat\pi_{H_l}|^2 \leq 
\big(\bE (T, \bB_{H_{l-1}}) + \bE (T, \bB_{H_l})\big)
\frac{C_0\, r_{H_{l-1}}^m}{\|T\| (\bB_{H_l})}\leq \bar{C}\, \bmo \,\ell (H_l)^{2-2\delta_2}\, .
\end{equation}
Using the estimate
$\sum_{l = j}^{\infty} \ell(H_l)^{1-\delta_2} \leq C_0\,\ell(H_j)^{1- \delta_2}$
and (ii), we conclude (iii) for
$H = H_{i+1}$ and $L = H_j$ .
As for (iv) it follows from (iii) and the case $|\pi_{H_{N_0}} - \pi_0| \leq \bar{C} \bmo^{\sfrac{1}{2}}$.
We next come to (v).
(v) holds  for $H_i$ and so we conclude 
$\supp (T) \cap \bC_{36 r_{H_{i}}} (p_{H_{i}}, \pi_{0}) \subset \bB_{H_{i}}$. Since
$|p_{H_{i+1}} - p_{H_i}| \leq 3 \sqrt{m}\, \ell(H_{i+1})$
and $r_{H_{i+1}} = \frac{1}{2} r_{H_i}$, we have 
$\bC_{36 r_{H_{i+1}}} (p_{H_{i+1}}, \pi_{0})\subset \bC_{36 r_{H_{i}}} (p_{H_{i}}, \pi_{0})$ provided the geometric constant in the first inequality of \eqref{e:N0} is large enough. Thus:
\begin{align*}
\bh (T, \bC_{36 r_{H_{i+1}}} (p_{H_{i+1}}, \pi_{0})) & \leq
\bh (T, \bB_{H_{i}}) + C_0\,r_{H_i} |\hat \pi_{H_{i}} - \pi_{0}| \\
& \stackrel{\mathclap{\textup{(iv)}}}
\leq \,
C_h\,\bmo^{\sfrac{1}{2m}} \ell(H_i)^{1+\beta_2} + \bar{C}\,\bmo^{\sfrac{1}{2}}
\ell(H_i) \leq C\, \bmo^{\sfrac{1}{2m}} \ell(H_i),
\end{align*}
where we used $H_i \in \sS^i$. Thus (v) follows easily for $H=H_{i+1}$. The inclusion $\supp (T) \cap \bC_{36 r_{H_{i+1}}} (p_{H_{i+1}}, \pi_0) \subset \bB_{H_{i+1}}$
is an obvious corollary of the bound and of the fact that the center of the ball $\bB_{H_{i+1}}$ (i.e. the
point $p_{H_{i+1}}$) belongs to $\supp (T) \cap \bC_{36 r_{H_{i+1}}} (p_{H_{i+1}}, \pi_0)$: we again need to ensure that $\eps_2$ is chosen small enough.

Next we show (vi) for $H= H_{i+1}$ and $L = H_j$ with $j\leq i+1$
(including the case $L=H_{i+1}$). The argument is the same in both cases $\pi_H$ and $\hat{\pi}_H$ and we show it in the first case. We first prove the second claim of (vi) inductively on $j$. Observe that for $j =N_0$ we can argue as for the inclusion $\bC_{36 r_{H_{N_0}}} (p_{H_{N_0}}, \pi_{H_{N_0}}) \cap \bB_{6\sqrt{m}} \subset \bC_{5\sqrt{m}} (0, \pi_0)$ to infer also 
$\bC_{36 r_{H_{N_0}}} (p_{H_{N_0}}, \pi_H) \cap \bB_{5\sqrt{m}}
\subset \bC_{5\sqrt{m}} (0, \pi_0)$: since $|\pi_{H_{N_0}} - \pi_{H}|\leq \bar{C} \bmo^{\sfrac{1}{2}} \ell (H_{N_0})^{1-\delta_2}$, such inclusion simply requires a smaller choice of $\eps_2$. We can then use \eqref{e:pre_height} to infer
\[
\bh (T, \bC_{36 r_{H_{N_0}}} (p_{H_{N_0}}, \pi_H)) \leq \bh (T, \bC_{5\sqrt{m}} (0, \pi_0), \pi_0) + C_0 r_{H_{N_0}} |\pi_0 - \pi_H| \leq C \bmo^{\sfrac{1}{2m}} \ell (H_{N_0})^{1+\beta_2}\, .
\]
Again the inclusion $\supp (T) \cap \bC_{36 r_{H_{N_0}}} (p_{H_{N_0}}, \pi_H)\subset \bB_{H_{N_0}}$ follows from assuming $\eps_2$ sufficiently small. Next, assume that the second claim of (vi) holds for $H$ and $L= H_l$. Observe that 
\[
\bC_{36 r_{H_{l+1}}} (p_{H_{l+1}}, \pi_H) \subset \bC_{36 r_{H_l}} (p_{H_l}, \pi_H):
\] 
in fact, arguing as above, we have $|p_{H_{l+1}} - p_H| \leq 3 \sqrt{m} \ell (H_{l+1})$ and thus such inclusion requires only a sufficiently large geometric constant in the first inequality of \eqref{e:N0}. But then, we know
$\bC_{36 r_{H_{l+1}}} (p_{H_{l+1}}, \pi_H) \subset \bB_{H_l}$ and we can therefore conclude
\[
\bh (T, \bC_{36 r_{H_{l+1}}} (p_{H_{l+1}}, \pi_H), \pi_H) \leq \bh (T, \bB_{H_l}) + C_0 r_{H_{l+1}} |\pi_H - \hat{\pi}_{H_l}|\, .
\]
From this we then conclude the second claim of (vi), i.e. $\supp (T) \cap \bC_{36 r_{H_{l+1}}} (p_{H_{l+1}}, \pi_H) \subset \bB_{H_{l+1}}$. 
Next, the first claim of (vi) is an obvious consequence of the second claim when $L = H_j$ for $j\leq i$ because $L\in \sS$: in this case we have, as computed above,
\begin{align*}
\bh (T, \bC_{36 r_{L}} (p_{L}, \pi_H)) & \leq
\bh (T, \bB_{L}) + C_0\,r_L |\hat \pi_L - \pi_H| \\
& \stackrel{\mathclap{\textup{(iii)\&(ii)}}}
\leq \,
C_h\,\bmo^{\sfrac{1}{2m}} \ell(L)^{1+\beta_2} + \bar{C}\,\bmo^{\sfrac{1}{2}}
\ell(L)^{2-\delta_2} \leq C\, \bmo^{\sfrac{1}{2m}} \ell(L)^{1+\beta_2}\, .
\end{align*}
Finally, since $\bC_{36 r_H} (p_H, \pi_H) \subset \bC_{36 r_{H_i}} (p_{H_i}, \pi_H) \subset \bB_{H_i}$ and the sidelengths $\ell (H)$ and $\ell (H_i)$ differ by a factor $2$, we conclude as well that the first claim of (vi) holds for $H=L$.

\medskip

{\bf Proof of Proposition \ref{p:whitney}.} Observe that \eqref{e:prima_parte} has already been shown
in the previous subsection and that \eqref{e:ex+ht_ancestors} is an obvious consequence of the definition of
$\sS$: it only remains to
show \eqref{e:ex+ht_whitney}. Fix then $L\in \sW$ and recall that its father $J$ belongs to $\sS$. However, having proved (i)--(vi) for pairs of cubes in which one is the ancestor of the other, we know that $\bB_L \subset \bB_J$ and thus we achieve
\begin{align}
\bE (T, \bB_L) \leq 2^m \bE (T, \bB_J) \leq 2^m C_e \bmo \ell (J)^{2-2\delta_2} \leq 2^{m+2-2\delta_2} C_e \bmo \ell (L)^{2-2\delta_2}\\
\bh (T, \bB_L) \leq \bh (T, \bB_J) + C_0 r_L |\hat{\pi}_J - \hat{\pi}_L| \stackrel{(ii)\&(iii)}{\leq} C \bmo^{\sfrac{1}{2m}} \ell (L)^{1+\beta_2}\, .
\end{align}

\medskip

{\bf Proof of (i)--(vi) for neighboring $H$ and $L$.} Observe that in this case we only have to show (iii) and (vi). 
The argument for (iii) is entirely analogous to the case $H\subset L$. Assume first that $L\not \in \sS^{N_0}$. Then $L$ has a father $J$. As already
seen we have $|p_L - p_J| \leq 3 \sqrt{m} \ell (J)$. On the other hand it is also easy to see that, with the same argument,
we conclude $|p_H- p_L|\leq 3 \sqrt{m} \ell (L)$ and thus $|p_H- p_J| \leq 5 \sqrt{m} \ell (J)$. We therefore easily conclude
$\bB_H \cup \bB_L \subset \bB_J$, provided the geometric constant in the first inequality of \eqref{e:N0} is large enough.
Therefore, we can estimate
\[
|\hat{\pi}_L - \hat{\pi}_J| \leq C_0 (\bE (T, \bB_J) + \bE (T, \bB_L))^{\sfrac{1}{2}}
\]
and use (ii) to conclude. In case $L\in \sS^{N_0}$, we can simply replace $\bB_J$ with $\bB_{5\sqrt{m}}$. 

We pass finally to (vi). We can in fact use the very same argument already explained when $H\subset L$: we claim indeed that (vi) holds not only for $L$ but also for all its ancestors $J$ and prove this claim by induction exactly as done above.
\end{proof}

\subsection{Existence of several approximating maps}
Next, we prove that the building blocks for the construction of the
center manifold are well-defined.

\begin{proposition}[Existence of interpolating functions]\label{p:gira_e_rigira}
Assume the conclusions of Proposition \ref{p:tilting opt} apply.
The following facts are true provided $\eps_2$ is sufficiently small. 
Let $H, L\in \sW\cup \sS$ be such that either $H\subset L$ or $H\cap L \neq \emptyset$ and
$\frac{\ell (L)}{2} \leq \ell (H) \leq \ell (L)$. Then,
\begin{itemize}
\item[(i)] for $\pi= \pi_H, \hat{\pi}_H$, $(\p_{\pi})_\sharp T\res \bC_{32r_L} (p_L, \pi) = Q \a{B_{32r_L} (p_L, \pi))}$ and $T$ satisfies the assumptions of \cite[Theorem 1.4]{DS3} in the cylinder $\bC_{32 r_L} (p_L, \pi)$;
\item[(ii)] Let $f_{HL}$ be the $\pi_H$-approximation of $T$ in $\bC_{8 r_L} (p_L, \pi_H)$ and $h := (\etaa\circ f_{HL})*\varrho_{\ell (L)}$ be its smoothed average. Set $\varkappa_H := \pi_H^\perp \cap T_{p_H} \Sigma$ and consider the maps 
\begin{equation*}
\begin{array}{lll}
x\quad \mapsto\quad \bar{h} (x)  &:=  \p_{T_{p_H}\Sigma} (h)&\in \varkappa_H\\ 
x\quad \mapsto\quad h_{HL} (x) &:= (\bar{h} (x), \Psi_{p_H} (x, \bar{h} (x)))&\in \varkappa_H \times (T_{p_H} (\Sigma))^\perp\, .
\end{array}
\end{equation*} 
Then there is a smooth
$g_{HL} :  B_{4r_L} (p_L, \pi_0)\to \pi_0^\perp$ s.t. $\bG_{g_{HL}} = \bG_{h_{HL}}\res \bC_{4r_L} (p_L, \pi_0)$.
\end{itemize}
\end{proposition}

\begin{definition}\label{d:mappe_h_HL}
$h_{HL}$ and $g_{HL}$ will be called, respectively, {\em tilted $(H,L)$-interpolating function} and {\em $(H,L)$-interpolating function}.
\end{definition}

Observe that the tilted $(L,L)$-interpolating function and the $(L,L)$-interpolating function correspond to the tilted $L$-interpolating function and to the $L$-interpolating function 
of Definition~\ref{d:glued}. Obviously, Lemma \ref{l:tecnico2} is just a particular case of Proposition \ref{p:gira_e_rigira}.

\begin{proof} We use the convention that $C_0$ and $c_0$ denote geometric constants, $\bar{C}$ and $\bar{c}$ denote dependence upon $\beta_2, \delta_2, M_0, N_0$ and $C_e$, whereas $C$ and $c$ dependence upon
$\beta_2, \delta_2, M_0, N_0, C_e$ and $C_h$. There are two cases: (i) $\pi = \pi_H$ and (ii) $\pi = \hat{\pi}_H$; since the argument for case (ii) is entirely analogous to that for case (i) we only give it for case (i). First recall that, by Proposition~\ref{p:tilting opt},
\begin{equation}\label{e:cilindro_dentro_palla}
\supp (T\res \bC_{32r_L} (p_L, \pi_H)) \subset \B_L\subset \B_{5\sqrt{m}}.
\end{equation}
We thus have $\partial T \res \bC_{32r_L} (p_L, \pi_H)=0$ and thus, setting $\p:=\p_{\pi_H}$, we conclude
\begin{equation}\label{e:proiezione_intera}
\p_\sharp T \res \bC_{32r_L} (p_L, \pi_H) = k \a{B_{32r_L} (\p (p_L), \pi_H)}
\end{equation}
for some integer $k$. We will show now that $Q=k$. If $J$ is the father of $L$, we then have proved in the previous section that
$|p_L-p_J|\leq 3\sqrt{m} \ell (L)$. We thus conclude $\bC_{32 r_L} (p_L, \pi_H) \subset \bC_{32 r_J} (p_J, \pi_H)$, provided $M_0$ is larger than a geometric constant. Consider the
chain of ancestors $J\subset \ldots \subset M$ of $L$, till $M\in \sS^{N_0}$. 
We then have $\bC_{32 r_L} (p_L, \pi_H) \subset \bC_{32 r_M} (p_M, \pi_H)$ and it suffices to show
\begin{equation}\label{e:di_nuovo_proietta}
\p_\sharp T \res \bC_{32r_M} (p_M, \pi_H) = Q \a{B_{32r_M} (\p (p_M), \pi_H)}
\end{equation}
Observe also that $|\pi_0 - \pi_H|\leq \bar{C} \bmo^{\sfrac{1}{2}}$, by Proposition~\ref{p:tilting opt}.
Join $\pi_{H} =: \pi (1)$ and $\pi_{0} =: \pi (0)$ 
with a continuous one-parameter family of planes $\pi (t)$ with the property
that 
\begin{equation}\label{e:uniforme}
|\pi (t) - \pi_{0}| \leq C_0 |\pi_{H} - \pi_{0}|\leq \bar{C}\bmo^{\sfrac{1}{2}}\, ,
\end{equation}
where $C_0 >0$ is some geometric constant. Since $\bar C = \bar C(\beta_2, \delta_2, M_0, N_0, C_e)$, it is then clear from \eqref{e:uniforme} that, 
if $\eps_2$ is suitably small, then
we have $\bB_{6\sqrt{m}} \cap \bC_{32r_{M}} (p_{M}, \pi_{t})
\subset \bC_{5\sqrt{m}} (0, \pi_{0})$ for every $t\in [0,1]$ (as already argued in the proof of Proposition~\ref{p:tilting opt}). 
We consider then the currents
$S(t):=(\p_{\pi(t)})_\sharp T \res \bC_{32r_{M}} (p_{L}, \pi (t))$ and get
$S(t) = Q (t) \a{ B_{34r_{M}} (\p_{\pi (t)} (p_{M}),
\pi(t))}$, where $Q(t)$ is an integer for every $t$
by the Constancy Theorem. 
On the other hand $t\mapsto S(t)$ is weakly continuous in the space of currents
and thus $Q(t)$ must be constant. Since $Q(0)=Q$ by \eqref{e:geo semplice 1},
this proves the desired claim.

\medskip

Observe next that, again from Proposition \ref{p:tilting opt}, 
\begin{equation*}
\bE (T, \bC_{32 r_L} (p_L, \pi_H)) \leq \bar{C} \bE (T, \bB_L, \pi_H) \leq \bar{C} \bE (T, \bB_L) + \bar{C} |\pi_H -\hat\pi_L|^2 \leq \bar{C} \bmo \,\ell(L)^{2-2\delta_2}\, .
\end{equation*}
If $\eps_2$ is sufficiently small, then
$\bE (T, \bC_{32 r_L} (p_L, \pi_H)) < \eps_{1}$,
where $\eps_1$ is the constant of \cite[Theorem~1.4]{DS3}. 
Therefore, 
the current $T\res \bC_{32 r_L} (p_L, \pi_H)$ and the submanifold $\Sigma$
satisfy all the assumptions of \cite[Theorem~1.4]{DS3} 
in the cylinder $\bC_{32 r_L} (p_L, \pi_H)$: we apply it to construct the $\pi_H$-approximation $f_{HL}$.
By \cite[Theorem~1.4]{DS3} and the properties of $\Psi_{p_H}$, we have 
\[
\Lip (h_{HL}) \leq C_0 \Lip (\etaa \circ f_{HL}) \leq \bar{C} \left(\bE (T, \bC_{32 r_L} (p_L, \pi_H))\right)^{\gamma_1} \leq \bar{C} \bmo^{\gamma_1}\ell(L)^{\gamma_1},
\]
and
\begin{align*}
\|h_{HL} - \p_{\pi_H^\perp} (p_L)\|_{C^0} &\leq C_0 \|\etaa\circ f_{HL} - \p_{\pi_H^\perp} (p_L)\|_{C^0}
\leq C_0 \|\cG \big(f_{HL}, Q \,\llbracket p_{\pi_H^\perp} (p_L) \rrbracket\big)\|_{C^0}\nonumber\\ 
&\leq\; C_0 \bh (T, \bC_{32r_L} (p_L, \pi_H)) + \big(\bE (T, \bC_{32r_L} (p_L, \pi_H))^{\sfrac{1}{2}} + \bA\, r_L\big) r_L\notag\\
&\leq C \bmo^{\sfrac{1}{2m}} \ell (L)^{1+\beta_2}\, .
\end{align*}
Since $C$ does not depend on $\eps_2$, if the latter is smaller than a suitable positive constant $c(\beta_2, \delta_2, M_0, N_0, C_e, C_h)$, we can apply Lemma \ref{l:rotazioni_semplici} 
to conclude that the interpolating function $g_{HL}$ is well-defined. 
\end{proof}

\subsection{Key estimates and proof of Theorem \ref{t:cm}}
We are now ready to state the key construction estimates and show
how Theorem \ref{t:cm} follows easily from them.

\begin{proposition}[Construction estimates]\label{p:stime_chiave}
 Assume the conclusions of Propositions \ref{p:tilting opt} and \ref{p:gira_e_rigira} apply 
and set $\kappa = \min \{\beta_2/4, \eps_0/2\}$. Then,
the following holds for any pair of cubes $H, L\in \sP^j$ (cf. Definition \ref{d:glued}), where
$C = C (\beta_2, \delta_2, M_0, N_0, C_e, C_h)$:
\begin{itemize}
\item[(i)] $\|g_H\|_{C^0 (B)}\leq C\, \bmo^{\sfrac{1}{2m}}$ and
$\|Dg_H\|_{C^{2, \kappa} (B)} \leq C \bmo^{\sfrac{1}{2}}$, for $B = B_{4r_H} (x_H, \pi_0)$;
\item[(ii)] if $H\cap L\neq \emptyset$,
then $\|g_H-g_L\|_{C^i (B_{r_L} (x_L))} \leq C \bmo^{\sfrac{1}{2}} \ell (H)^{3+\kappa-i}$ 
for every $i\in \{0, \ldots, 3\}$;
\item[(iii)] $|D^3 g_H (x_H) - D^3 g_L (x_L)| \leq C \bmo^{\sfrac{1}{2}} |x_H-x_L|^\kappa$;
\item[(iv)] $\|g_H-y_H\|_{C^0} \leq C \bmo^{\sfrac{1}{2m}} \ell (H)$ and 
$|\pi_H - T_{(x, g_H (x))} \bG_{g_H}| \leq C \bmo^{\sfrac{1}{2}} \ell (H)^{1-\delta_2}$
$\forall x\in H$;
\item[(v)] if $L'$ is the cube concentric to $L\in \sW^j$ with $\ell (L')=\frac{9}{8} \ell (L)$, 
then
\[
\|\varphi_i - g_L\|_{L^1 (L')} \leq C\, \bmo\, \ell (L)^{m+3+\beta_2/3} \quad \text{for all }\; i\geq j.
\]
\end{itemize}
\end{proposition}


\begin{proof}[Proof of Theorem \ref{t:cm}] As in all the proofs so far, we will use $C_0$ for geometric constants and $C$ for constants which depend upon $\beta_2, \delta_2, M_0, N_0, C_e$ and $C_h$. Define $\chi_H := \vartheta_H/ (\sum_{L\in \sP^j} \vartheta_L)$ for each $H\in \sP^j$
and observe that 
\begin{align}
\sum_{H\in\sP^j} \chi_H = 1 \;\; \mbox{on $[-4,4]^m$}\qquad \text{and}\qquad
\|\chi_H\|_{C^i} &\leq C_0 \,\ell (H)^{-i} \quad \forall i\in \{0,1,2,3,4\}\label{e:p_unita}\, .
\end{align}
Set $\sP^j (H):=\{L\in \sP^j : L\cap H\neq\emptyset\}\setminus \{H\}$ for each $H\in \sP^j$. 
By construction $\frac{1}{2} \ell (L) \leq \ell (H) \leq 2\, \ell (L)$ for every $L\in \sP^j (H)$ and the cardinality of $\sP^j(H)$
is bounded by a geometric constant $C_0$. 
The estimate $|\hat\varphi_j| \leq C \bmo^{\sfrac{1}{2m}}$ follows then immediately
from Proposition~\ref{p:stime_chiave}(i). 
For $x\in H$ we write 
\begin{align}
\hat\varphi_j (x) &=\Big(g_H \chi_H  + \sum_{L\in \sP^j (H)} 
g_L \chi_L\Big) (x) = g_H (x) + \sum_{L\in \sP^j (H)} (g_L - g_H) \chi_L\,  (x)\, ,
\end{align}
because $H$ does not meet the support of $\vartheta_L$ for any $L\in \sP^j$ which does not meet $H$.
Using the Leibniz rule, \eqref{e:p_unita} and the estimates of 
Proposition~ \ref{p:stime_chiave}(i) - (ii), for $i\in \{1,2,3\}$ we get
\[
\|D^i \hat\varphi_j\|_{C^0 (H)} \leq \|D^ig_H\|_{C^0} + C_0 \sum_{0\leq l \leq i} 
\sum_{L\in \sP^j (H)} \|g_L-g_H\|_{C^l (H)} \ell (L)^{l-i} \leq C \bmo^{\frac{1}{2}} 
\big(1+\ell (H)^{3+\kappa-i}\big),
\]
(assuming $M_0$ is larger than the geometric constant $2\sqrt{m}$, we have $H \subset B_{r_L} (x_L)$ and the
estimate of Proposition \ref{p:stime_chiave}(ii) can be applied).
Next, using also $\|D^3 g_H - D^3 g_L\|_{C^\kappa (B_{r_L} (x_L)} \leq C \bmo^{\sfrac{1}{2}}$, we obtain 
\begin{align*}
[D^3 \hat \varphi_j]_{\kappa, H} \leq& C_0 \sum_{0\leq l \leq 3} \sum_{L\in \sP^j (H)} \ell (H)^{l-3} \big(
\ell (H)^{-\kappa} \|D^l (g_L-g_H)\|_{C^0 (H)} + [D^l (g_L - g_H)]_{\kappa, H}\big)\\
& + [D^3 g_H]_{\kappa, H}  \leq C \bmo^{\sfrac{1}{2}}\, ,
\end{align*}
where $[a]_{\kappa, D}$ is the usual H\"older seminorm $\sup \{|x-y|^{-\varkappa} |a (x)-a(y)|: x\neq y, x,y\in D\}$.
Fix now $x, y\in [-4,4]^m$, let $H, L\in \sP^j$ be such that $x\in H$ and $y\in L$.
If $H\cap L\neq \emptyset$, then
\begin{equation}\label{e:primo caso}
|D^3 \hat\varphi_j (x) - D^3 \hat\varphi_j (y)| \leq C \big([D^3 \hat\varphi_j]_{\kappa, H}
+ [D^3 \hat\varphi_j]_{\kappa, L}\big) |x-y|^{\kappa}.
\end{equation}
If $H\cap L= \emptyset$, we assume w.l.o.g. $\ell (H) \leq \ell (L)$ and observe that 
\[
\max \big\{|x-x_H|, |y-x_L|\big\} \leq \sqrt{m} \ell(L) \leq 2 \sqrt{m} |x-y|\, .
\] 
Moreover, by construction $\hat{\varphi}_j$ is identically equal to $g_H$ in a neighborhood of its center $x_H$. Thus, we can estimate 
\begin{align}
|D^3 \hat\varphi_j (x) - &D^3 \hat\varphi_j (y)|
\nonumber\\ 
\leq \;& |D^3 \hat\varphi_j (x) - D^3 \hat\varphi_j (x_H)|  
+ |D^3 g_H (x_H) - D^3 g_L (x_L)|
+ |D^3 \hat\varphi_j (x_L) - D^3 \hat\varphi_j (y)|\nonumber\\
\leq \;&C \bmo^{\sfrac{1}{2}} \left(|x-x_H|^\kappa + |x_H-x_L|^\kappa+ |y-x_L|^\kappa \right) \leq C \bmo^{\sfrac{1}{2}} |x-y|^\kappa\, ,
\end{align}
where we used
\eqref{e:primo caso} and Proposition~\ref{p:stime_chiave}(iii).
We have thus shown $\|D \hat\varphi_j\|_{C^{2,\kappa}} \leq C \bmo^{\sfrac{1}{2}}$. 
Since $\varphi_j (x) = (\bar\varphi_j (x), \Psi (x, \bar\varphi_j (x)))$, where $\bar\varphi_j (x)$ 
denote the first $\bar{n}$ components of $\hat\varphi_j (x)$, 
Theorem \ref{t:cm}(i) follows easily from the chain rule.

Let $L\in \sW^i$ and fix $j\geq i+2$. 
Observe that, by the inductive procedure defining $\sS^j\cup \sW^j$, 
we have $\sP^j (L) = \sP^{i+2} (L)\subset \sW$. Let $H$ be the cube concentric to $L$ with $\ell (H) = \frac{9}{8} \ell (L)$. Then, by Assumption \ref{mollificatore}, 
$\supp (\vartheta_M)\cap H=\emptyset$ $\forall M\not\in \sP^j (L)$. Thus, 
Theorem~\ref{t:cm}(ii) follows.

We now show below that $\|\varphi_j - \varphi_{j+1}\|_{C^0 (]-4,4[^m)} \leq C 2^{-j}$. This immediately implies the existence of a continuous $\varphi$ to which $\varphi_j$ converges uniformly. The bounds of Theorem \ref{t:cm}(i) immediately implies Theorem \ref{t:cm}(iii). Fix therefore $x\in [-4,4]^m$ and assume that $x \in L\cap H$ with $L \in \sP^j$
and $H\in \sP^{j+1}$.
Without loss of generality, we can make the choice of $H$
and $L$ in such a way that either $H=L$ or $H$ is a son of $L$.
Now, if $\ell(L) \geq 2^{-j+2}$, then by (ii) we have $\ph_j(x) = \ph_{j+1}(x)$.
Otherwise, from (i) and Proposition \ref{p:stime_chiave}(iv), we can conclude that:
\begin{align}
 |\hat\varphi_j (x) - \hat\varphi_{j+1} (x)| &\leq |\hat\varphi_j (x) - \hat\varphi_j (x_H)| + |g_H (x_H) - g_L (x_L)| + |\hat\varphi_{j+1} (x_L) - \hat\varphi_{j+1} (x)|\nonumber\\
&\leq C\big( \|\hat\varphi_j\|_{C^1} + \|\hat\varphi_{j+1}\|_{C^1}\big) 2^{-j} + \|g_H-y_H\|_{C^0} + \|g_L - y_L\|_{C^0} + |y_H-y_L|\nonumber\\
&\leq C \bmo^{\sfrac{1}{2m}} 2^{-j} + |p_H-p_L|\, .
\end{align}
Since $\B_H\subset \B_L$, we conclude 
$|\hat\varphi_j (x) - \hat\varphi_{j+1} (x)|\leq C 2^{-j}$.
Given that $\Psi$ is Lipschitz, we get 
$\|\ph_j - \ph_{j+1}\|_{C^0} \leq C\, 2^{-j}$ and conclude.
\end{proof}

\section{Proof of the three key construction estimates}\label{s:stima_C3alpha}

\subsection{Elliptic PDE for the average}
This section contains the most important computation, namely
the derivation via a first variation argument of a suitable elliptic system for
the average of the $\pi$-approximations. 
In order to simplify the notation we introduce the following definition.

\begin{definition}[Tangential parts]\label{d:tangential}
Having fixed $H\in \sP^j$ and $\pi:= \pi_H\subset T_{p_H} \Sigma$,
we let $\varkappa$ be the
orthogonal complement of $\pi$ in $T_{p_H} \Sigma$. 
For any given point $q\in \R^{m+n}$, any set $\Omega\subset \pi$ and any map 
$\xi: q+\Omega \to \pi^\perp$, 
 the map $\p_{\varkappa} \circ \xi$ will be called the {\em tangential part of $\xi$} and usually denoted by $\bar{\xi}$.
Analogous notation and terminology will be used for multiple-valued maps.
\end{definition}

\begin{proposition}[Elliptic system]\label{p:pde}
Assume the conclusions of Proposition \ref{p:tilting opt} and \ref{p:gira_e_rigira}.
Let $H\in \sW^j\cup \sS^j$ and $L$ be either an ancestor or a cube of
$\sW^j\cap \sS^j$ with $H\cap L\neq \emptyset$
(possibly also $H$ itself). Let $f_{HL}: B_{8r_L} (p_L, \pi_H) \to \Iq (\pi_H^\perp)$ be the 
$\pi_H$-approximation of $T$ in $\bC_{8r_L} (p_L, \pi_H)$, $h_{HL}$ the tilted $(H,L)$-interpolating functions
and $\bar{f}_{HL}$ and $\bar{h}_{HL}$ their tangential parts, according to Definition \ref{d:tangential}.
Then, there is a matrix $\bL$, which depends on $\Sigma$ and $H$ but not on $L$, such that
$|\bL|\leq C_0 \bA^2 \leq C_0 \bmo$ for a geometric constant $C_0$ and
(for $C= C(\beta_2, \delta_2, M_0, N_0, C_e, C_h)$) 
\begin{equation}\label{e:pde}
\left| \int \big( D (\etaa \circ \bar f_{HL}) : D \zeta + (\p_\pi (x-p_H))^t \cdot \bL \cdot\zeta\big)\right|
\leq C \bmo \, r_L^{m+1+\beta_2} \big(r_L\,\|\zeta\|_{C^1} + \|\zeta\|_{C^0}\big)\, 
\end{equation}
for every $\zeta\in C^\infty_c (B_{8r_L} (p_L, \pi_L), \varkappa)$.
Moreover (for $C= C(\beta_2, \delta_2, M_0, N_0, C_e, C_h)$)
\begin{equation}\label{e:L1_est}
\|\bar h_{HL} - \etaa \circ \bar f_{HL}\|_{L^1 (B_{7r_L} (p_L, \pi_L))} \leq C \bmo \, r_L^{m+3+\beta_2}\, .
\end{equation}
\end{proposition}

Before coming to the proof we introduce the oscillation of a multivalued function $f$, which will play an important role
also later:
\begin{equation}\label{e:oscillation}
{\rm osc}\, (f) := \sup \{|P-P'| : P\in \supp (f (x)), P' \in\supp f(y))\}\, .
\end{equation}
Observe that the oscillation is comparable to $\sup_{x,y} \cG (f(x), f(y))$.

\begin{proof} We use the convention that geometric constants are denoted by $C_0$, whereas $C$ denotes constants
depending upon the parameters $\beta_2, \delta_2, M_0, N_0, C_e$ and  $C_h$. Set $\pi = \pi_H$.
We fix  a system of coordinates $(x,y,z)\in \pi\times \varkappa \times (T_{p_H} \Sigma)^\perp$
so that $p_H = (0,0,0)$. Also, in order to simplify the notation, although the domains 
of the various maps are subsets $\Omega$
of $p_L + \pi$, we will from now on consider them as functions of $x$ 
(i.e.~we shift their domains to $\p_{\pi} (\Omega)$). 
We also use $\Psi_H$ for the map $\Psi_{p_H}$ 
of Assumption \ref{ipotesi}.
Recall that $\Psi_H (0,0)=0$, $D\Psi_H (0,0)=0$ and
$\|D \Psi_H\|_{C^{2, \eps_0}}\leq \bmo^{\sfrac{1}{2}}$. Finally, to simplify the notation we also drop the
subscripts $HL$ from the functions $f_{HL}$, $\bar{f}_{HL}$ and $\bar{h}_{HL}$ (this notation might
generate some confusion since $h$ is used in Proposition \ref{p:gira_e_rigira} for the smoothed average of
$f_{HL}$; observe however that the tangential part of such smoothed average does coincide with the tangential
part of the tilted $(H,L)$-approximation).

Given a test function $\zeta$ and any point $q = (x,y,z)\in \Sigma$,
we consider the vector field
$\chi (q) = (0, \zeta (x), D_y \Psi (x,y) \cdot  \zeta(x))$.
$\chi$ is tangent to $\Sigma$ and therefore $\delta T (\chi) =0$. Thus, 
\begin{align}
|\delta \bG_f (\chi)| = |\delta \bG_f (\chi) - \delta T (\chi)| 
\leq C_0 \int_{\bC_{8r_L} (p_L, \pi)} |D\chi|\, d \|\bG_f - T\|\, .\label{e:var_prima}
\end{align}
Let $r = r_L$ and $B= B_{8r_L} (p_L, \pi)$.
Since $\|D\Psi_H\|_0 \leq \bmo^{\sfrac{1}{2}}$, for $\eps_2$ sufficiently small we achieve
$|\chi| \leq 2 |\zeta|$ and $|D\chi| \leq 2 |\zeta| + 2 |D\zeta|$.
Set now $E:= \bE \big(T, \bC_{32 r} (p_L, \pi)\big)$ and recall \cite[Theorem~1.4]{DS3} to derive
\begin{gather}
|Df| \leq C_0 E^{\gamma_1} + C_0 r \bA
\leq C \bmo^{\gamma_1} r^{\gamma_1},\label{e:lip f}\\
|f| \leq C_0 \bh (T, \bC_{32 r} (p_L, \pi)) + C_0 (E^{\sfrac{1}{2}} + r\bA) r \leq C \bmo^{\sfrac{1}{2m}} r^{1+\beta_2},\label{e:Linf}\\
\int_B |Df|^2 \leq C_0\, r^m E \leq C \bmo \, r^{m+2-2\delta_2}\, ,\label{e:Dir f}
\end{gather}
and 
\begin{align}
|B\setminus K| &\leq C_0 E^{\gamma_1} (E + r^2 \bA^2) \leq C \bmo^{1+\gamma_1} r^{m + 2 - 2\delta_2 + \gamma_1}\, ,\label{e:aggiuntiva_1}\\
\left| \|T\| (\bC_{8r} (p_L, \pi)) - |B| - \frac{1}{2} \int_B |Df|^2\right| &\leq C_0 E^{\gamma_1} (E + r^2 \bA^2) \leq  C \bmo^{1+\gamma_1} r^{m + 2 - 2\delta_2 + \gamma_1}\, , \label{e:aggiuntiva_2}
\end{align}
where $K\subset B$ is the set
\begin{equation}\label{e:recall_K}
B\setminus K = \p_\pi \left(\left(\supp (T)\Delta \gr (f)\right)\cap \bC_{8r_L} (p_L, \pi)\right)\, .
\end{equation}
Concerning \eqref{e:Linf} observe that the statement of \cite[Theorem~1.4]{DS3} bounds
indeed ${\rm osc}\, (f)$. However, in our case we have $p_H = (0,0,0)\in \supp (T)$ and 
$\supp (T)\cap \gr (f)\neq\emptyset$. Thus we conclude $|f|\leq C_0 {\rm osc}\, (f) + C_0
\bh (T, \bC_{32 r} (p_L, \pi))$.

Writing  $f= \sum_i \a{f_i}$ and $\bar{f} = \sum_i \a{\bar{f}_i}$,
$\gr (f) \subset \Sigma$ implies
$f = \sum_i \a{(\bar{f}_i, \Psi_H (x, \bar{f}_i))}$.
From \cite[Theorem~4.1]{DS2} we can infer that
\begin{align}
\delta \bG_f (\chi) &= 
\int_B \sum_i \Big( \underbrace{D_{xy}\Psi_H (x, \bar{f}_i)\cdot \zeta}_{(A)} + 
\underbrace{(D_{yy} \Psi_H (x, \bar{f}_i)\cdot D_x \bar{f}_i) \cdot \zeta}_{(B)} + \underbrace{D_y \Psi_H (x, \bar{f}_i) \cdot D_x \zeta}_{(C)}\Big)\nonumber\\
&\qquad\quad
: \Big(\underbrace{D_x \Psi_H (x, \bar{f}_i)}_{(D)} + \underbrace{D_y \Psi_H (x, \bar{f}_i)\cdot D_x \bar{f}_i}_{(E)}\Big)
+ \int_B \sum_i D_x \zeta : D_x \bar f_i + {\rm Err}\, .\label{e:tayloraccio}
\end{align}
To avoid cumbersome notation we use $\|\cdot\|_0$ for $\|\cdot\|_{C^0}$ and
$\|\cdot\|_1$ for $\|\cdot\|_{C^1}$. Recalling \cite[Theorem 4.1]{DS2},
the error term ${\rm Err}$ in 
\eqref{e:tayloraccio} satisfies the inequality
\begin{equation}\label{e:tayloraccio2}
|{\rm Err}|\leq C \int |D\chi| |Df|^3 \leq \|\zeta\|_1 \int |Df|^3 \leq C \|\zeta\|_1 \bmo^{1+\gamma_1} r^{m+2-2\delta_2+\gamma_1}\, .
\end{equation}
The second integral in \eqref{e:tayloraccio} is obviously 
$Q \int_B D\zeta : D (\etaa \circ \bar f)$.
We therefore expand the product in the first integral and estimate
all terms separately. 
We will greatly profit from the Taylor expansion $D \Psi_H (x,y) = D_x D \Psi_H (0,0) \cdot x + D_y D\Psi_H (0,0) \cdot y + O \big(\bmo^{\sfrac{1}{2}} (|x|^2 + |y|^2)\big)$. In particular we gather the following estimates:
\begin{align}
&|D \Psi_H (x, \bar{f}_i)| \leq C \bmo^{\sfrac{1}{2}} r
\quad\text{and}\quad
D\Psi_H (x, \bar{f}_i) = D_x D\Psi_H (0,0)\cdot x + O \big(\bmo^{\sfrac{1}{2}} r^{1+\beta_2}\big),\notag\\\
&|D^2 \Psi_H (x, \bar{f}_i)| \leq \bmo^{\sfrac{1}{2}}\quad\mbox{and}\quad D^2 \Psi_H (x, \bar{f}_i) = D^2 \Psi_H (0,0) + O \big(\bmo^{\sfrac{1}{2}} r\big)\, .\nonumber
\end{align}
We are now ready to compute
\begin{align}
\int \sum_i (A):(D) &= \int \sum_i (D_{xy} \Psi_H (0,0) \cdot \zeta) : D_x \Psi_H (x, \bar{f}_i) + O\Big( \bmo\, r^2 \int |\zeta|\Big) \nonumber\\
&= \int \sum_i (D_{xy} \Psi_H (0,0)\cdot \zeta : D_{xx} \Psi_H (0,0)\cdot x + O \Big( \bmo\, r^{1+\beta_2} \int |\zeta|\Big)\, .\label{e:(AD)}
\end{align}
Obviously the first integral in \eqref{e:(AD)} has the form $\int x^t \cdot \bL_{AD} \cdot \zeta$, where the matrix $\bL_{AD}$ is a quadratic function of $D^2 \Psi_H (0,0)$. Next, we estimate
\begin{gather}
\int \sum_i (A):(E)  = O \Big(\bmo^{1+\gamma_1} r^{1+\gamma_1} \int |\zeta|\Big),\label{e:(AE)}\\
\int \sum_i (B):\left((D)+(E)\right) = O\Big(\bmo^{1+\gamma_1}
 r^{1+\gamma_1} \int |\zeta|\Big),\label{e:(B(D+E))}\\
\int \sum_i (C):(E) = O \Big(\bmo^{1+\gamma_1} r^{2+\gamma_1} \int |D\zeta|\Big).\label{e:(CE)}
\end{gather}
Finally we compute 
\begin{align*}
& \int \sum_i (C):(D) = \int \sum_i ((D_{xy}\Psi_H (0,0) \cdot x) \cdot D_x \zeta) : D_x \Psi_H (x, \bar{f}_i) 
+ O \Big(\bmo\,r^{2+\beta_2} \int |D\zeta|\Big)\\
&= \int \sum_i (D_{xy}\Psi_H (0,0) \cdot x) \cdot D_x \zeta) : (D_{xx} \Psi_H (0,0) \cdot x) 
+ O \Big(\bmo \, r^{2+\beta_2} \int |D\zeta|\Big)\, .
\end{align*}
Integrating by parts the first integral in the last line we reach
\begin{equation}\label{e:(CD)}
\int \sum_i (C):(D) = \int x^t \cdot \bL_{CD} \cdot \zeta + O \Big(\bmo\, r^{2+\beta_2} \int |D\zeta|\Big)\, ,
\end{equation}
where the matrix $\bL_{CD}$ is a quadratic function of $D^2 \Psi_H (0,0)$. Set $\bL :=\bL_{AD} + \bL_{CD}$. 
Since $D \Psi_H (0,0) =0$, $\bL$ is in fact a 
quadratic function of the tensor $A_\Sigma$ at the point $p_H$. 
In order to summarize all our estimates we introduce some simpler notation.
We define $\bef = \etaa \circ \bar{f}$,
$\ell := \ell (L)$ and (recalling  the set $K$ of \eqref{e:recall_K}) the measure $\mu$ on $B$ as
\[
\mu (E):= |E\setminus K| + \|T\| ((E\setminus K)\times \R^n) \qquad \qquad \mbox{for every Borel $E\subset B$}.
\]
{ Since $\|T - \bG_f\| (E\times \R^n) \leq C_0 \mu (E)$ for every Borel $E\subset B$}, we can summarize \eqref{e:var_prima} and \eqref{e:tayloraccio} - \eqref{e:(CD)} into the following estimate:
\begin{align}
\Big|
\int \left( D\bef : D \zeta + x^t \cdot \bL \cdot\zeta\right)\Big| \leq &C \bmo \,r^{1+\beta_2}
\int \big(r |D\zeta (x)| + |\zeta (x)|\big)\, dx\, \nonumber\\
&+ C\int \big(r |D\zeta (x)| + |\zeta (x)|\big) \big(|Df (x)|^3 dx + d\mu (x)\big)\, .\label{e:tutto_insieme}
\end{align}
From \eqref{e:lip f} and \eqref{e:Dir f} we infer that
\begin{equation}\label{e:Dir+ f}
\int |Df|^3 \leq C r^m \Lip (f) E \leq
C\, \bmo^{1+ \gamma_1} r^{m+2-2\delta_2+\gamma_1} .
\end{equation}
Next, observe that 
\begin{align*}
\mu (B) &= |B\setminus K| + \|T\| ((B\setminus K)\times \pi^\perp)\\
 &\leq |B\setminus K| + \left|\|T\| (\bC_{32 r_L} (p_L, \pi)) - \mass (\bG_f)\right| + \|\bG_f\| ((B\setminus K)\times \pi^\perp)\\
&\leq C_0 |B\setminus K| (1+\Lip (f)) + \left|  \|T\| (\bC_{32 r_L} (p_L, \pi)) - |B| - \frac{1}{2} \int_B |Df|^2\right|
+ C_0 \int_B |Df|^3\, ,
\end{align*}
where in the last line we have used the Taylor expansion of the mass of $\bG_f$, cf. \cite[Corollary 3.3]{DS2}.
Using next \eqref{e:aggiuntiva_1}, \eqref{e:aggiuntiva_2} and \eqref{e:Dir+ f} we conclude
\begin{equation}\label{e:mu}
\mu (B) \leq C \, \bmo \,r^{m+2-2\delta_2 +\gamma_1}\, .
\end{equation}
Therefore \eqref{e:pde} follows from \eqref{e:tutto_insieme} 
and our choice of the parameters in Assumption \ref{parametri} (recall, in particular, $\gamma_1 - 2\delta_2 > \beta_2$).

We next come to \eqref{e:L1_est}. Fix a smooth radial test function $\varsigma \in C_c(B_\ell)$ with $\ell = \ell(L)$,
and set
$\zeta(\cdot) := \varsigma (z - \cdot) e_i$, where $e_{m+1}, \ldots, e_{m+\bar{n}}$ 
is on orthonormal base of $\varkappa$. 
Observe that, if in addition we assume $\int \varsigma =0$, then
$\int x_i \varsigma (z-x) dx = 0$. Under these assumptions, $\int x^t \cdot \bL \cdot \varsigma (z-x) dx =0$ and from
\eqref{e:tutto_insieme} we get for $z \in B_{7r_L}(p_L, \pi_L)$
\begin{align}
&\left|\int_{B_\ell(z)} \langle D\bef^i (x), D \varsigma (z-x)\rangle\, dx\right|\leq
C \int_{B_\ell(z)} |Df|^3 (x) (|D \varsigma| + |\varsigma|) (z-x)\, dx\nonumber\\
&\qquad\qquad + C \int_{B_\ell(z)} (r |D \varsigma| + |\varsigma|) (z-x)\, d\mu (x) + C \bmo r^{1+\beta_2} \int_{B_\ell} (r|D\varsigma| + |\varsigma|).\label{e:convoluzione}
\end{align}
Recall the standard estimate on convolutions $\|a*\mu\|_{L^1} \leq \|a\|_{L^1} \mu (B)$,
and integrate \eqref{e:convoluzione} in $z\in B_{7r_L}(p_L, \pi_L)$: by \eqref{e:Dir+ f} and \eqref{e:mu} (and recalling that $\gamma_1 - 2 \delta_2\geq \beta_2$) we reach
\begin{align}
\|D\bef^i*D\varsigma\|_{L^1(B_{7r_L}(p_L, \pi_L))} \leq & C \bmo \,r^{m+1+\beta_2} \int_{B_\ell} (r|D\varsigma| + |\varsigma|)
\quad \forall \varsigma\in C^\infty_c (B_\ell) \; \;\mbox{with}\;\int_{B_\ell} \varsigma =0\, .\label{e:convolve}
\end{align}
By a simple density argument, \eqref{e:convolve} holds also when $\varsigma \in W^{1,1}$ 
is supported in $B_\ell$ and $\int \varsigma =0$.
Observe next
\begin{align}
&\bar h (x) - \bef (x) = \int \varrho_\ell (y) (\bef (x-y) - \bef (x))\, dy
= \int \varrho_\ell (y) \int_0^1 D \bef (x-\sigma y) \cdot (-y)\, d\sigma\,dy\nonumber\\
=\;& \int \int_0^1 \varrho_\ell \left(\textstyle{\frac{w}{\sigma}}\right) D \bef (x-w) \cdot 
\frac{-w}{\sigma^{m+1}}\, dw
= \int D \bef (x-w) \cdot 
\underbrace{(-w) \int_0^1 \varrho_\ell \left(\textstyle{\frac{w}{\sigma}}\right) \sigma^{-m-1}\, d\sigma}_{=: \Upsilon (w)}\,
dw\, .\nonumber
\end{align}
Note that $\Upsilon$ is smooth on $\R^m\setminus \{0\}$ and unbounded in a neighborhood of $0$.
However, 
\begin{equation}
 \|\Upsilon\|_{L^1} = \int \int_0^1 |w| \left|\varrho \left(\textstyle{\frac{w}{\ell\sigma}}\right)\right| 
\ell^{-m} \sigma^{-m-1}\, d\sigma\, dw = \ell \int \int_0^1 |u| |\varrho (u)|\, d\sigma\, du \leq C r\, .
\end{equation}
Observe also that $\Upsilon (w)= w \,\psi (|w|)$.
Therefore $\Upsilon$ is a gradient. Since $\Upsilon (w)$ vanishes
outside a compact set, integrating along rays from $\infty$, we can compute a potential for it:
\begin{align}
\varsigma (w) &= \int_{|w|}^\infty \tau \int_0^1 \varrho_\ell \left(\textstyle{\frac{w \tau}{|w|\sigma}}\right) \sigma^{-m-1}\, d\sigma\, 
d\tau
= |w|^2 \int_1^\infty t \int_0^1 \varrho_\ell \left(\textstyle{\frac{w t}{\sigma}}\right) \sigma^{-m-1}\, d\sigma\, 
dt\, .
\end{align}
Then, $\varsigma$ is a $W^{1,1}$ function, supported in $B_\ell (0)$,
$\int \varsigma =0$ by Assumption~\ref{mollificatore}.
Summarizing, $\bar h^i - \bef^i = (D \bef^i) * D \varsigma$ for a convolution kernel for
which \eqref{e:convolve}  holds. Since
\begin{align}
\|\varsigma\|_{L^1} &\leq \int \int_1^\infty \int_0^1 t |w|^2 \left|\varrho \left(\textstyle{\frac{wt}{\ell \sigma}}\right)\right|
\ell^{-m}\sigma^{-m-1}d\sigma\, dt\, dw\nonumber\\
&= \ell^2 \int_1^\infty \int_0^1 \int |u|^2 |\rho (u)|du\, \sigma d\sigma\, t^{-m-1} dt \leq C r^2\, ,\label{e:L1_ancora}
\end{align}
we then conclude from \eqref{e:convolve} that
\[
\int_{B_{7r_L}(p_L, \pi_L)} |\bar h - \bef| \leq C \bmo\, r^{m+ 1+\beta_2} \int_{B_\ell} (r |D\varsigma| + |\varsigma|) \leq C \bmo\,
r^{m+ 3+\beta_2}.\qedhere
\]
\end{proof}

\subsection{$C^k$ estimates for $h_{HL}$ and $g_{HL}$}\label{ss:iterata}
Recall the tilted $(H,L)$-interpolating
function $h_{HL}$ and the interpolating function $g_{HL}$ of Definition~\ref{d:mappe_h_HL}.

\begin{lemma}\label{l:stime_iterative}
Assume that $H$ and $L$ are as in Proposition \ref{p:pde} and the hypotheses
of Proposition \ref{p:stime_chiave} hold. Set $B':= B_{5r_H} (p_H, \pi_H)$ and $B:= B_{4r_H} (p_H, \pi_0)$. Then,
for $C = C( \beta_2, \delta_2, M_0, N_0, C_e, C_h)$,
\begin{gather}
\|h_{HL} -h_H\|_{C^j (B')} + \|g_{HL} - g_H\|_{C^j (B)} \leq C \bmo^{\sfrac{1}{2}}\, \ell (L)^{3+2\kappa-j} \qquad
\forall j\in \{0, \ldots, 3\}\, , \label{e:stime_ricorsive_1}\\
\|h_{HL} -h_H\|_{C^{3,\kappa} (B')} + \|g_{HL} - g_H\|_{C^{3,\kappa} (B)} \leq C \bmo^{\sfrac{1}{2}} \,\ell (L)^{\kappa}\, .\label{e:stime_ricorsive_2}
\end{gather}
As a consequence Proposition \ref{p:stime_chiave}(i) and (iv) hold.
\end{lemma}

\begin{proof} All the constants $C$ will depend only upon the parameters $\beta_2, \delta_2, M_0, N_0, C_e$ and $C_h$, unless otherwise specified. 

Consider a triple of cubes $H$, $J$ and $L$ where $H\in \sS^j\cup \sW^j$ and
\begin{itemize}
\item[(a)] either $L$ is an ancestor of $H$ (possibly $H$ itself) and $J$ is  father of $L$;
\item[(b)] or $J$ is the father of $H$, and $L\in \sS^j\cup \sW^j$ is adjacent to $H$.
\end{itemize} 
In order to simplify the notation let $\pi:=\pi_H$ and $r:=r_L$.
By Proposition~\ref{p:tilting opt}(i), up
to taking the geometric constant in the first inequality of \eqref{e:N0} sufficiently large, we can assume that
$B^\flat:= B_{6r} (p_L, \pi) \subset
B^\sharp=B_{13r/2} (p_L, \pi) \subset \bar B:= B_{7r_J} (p_J, \pi)$. Consider the 
$\pi$-approximations $f_{HL}$ and $f_{HJ}$, respectively
in $\bC_{8r} (p_L, \pi)$ and $\bC_{8r_J} (p_J, \pi)$, and introduce the corresponding maps 
\begin{gather*}
\bar \bef_L := \p_\varkappa (\etaa\circ f_{HL})\quad \text{and}\quad \bar \bef_J:= \p_\varkappa (\etaa \circ f_{HJ}),\\
\bar h_{HL} := \bar \bef_L * \varrho_{\ell (L)} \quad\text{and}\quad
\bar h_{HJ} = \bar \bef_J * \varrho_{\ell (J)}\, ,
\end{gather*}
which are the tangential parts of the corresponding maps according to Definition \ref{d:tangential}.

If $\mathbf{l}$ is an affine function on $\R^m$ and $\varsigma$ a radial convolution kernel,
then $\varsigma * \mathbf{l} = \left(\int \varsigma\right) \mathbf{l}$ because 
$\mathbf{l}$ is an harmonic function. This means that  
$\int \langle (\zeta * \varrho), \mathbf{l}\rangle = \int \langle\zeta, \mathbf{l}\rangle$ for any test function $\zeta$ and any
radial convolution kernel $\varrho$ with integral $1$. Similarly 
$\int \langle (\zeta * \partial^I \varrho), \mathbf{l}\rangle = \int \langle\zeta, \partial^I \mathbf{l}\rangle$
for any partial derivative $\partial^I$ of any order.
Consider now a ball $\hat{B}$ concentric to 
$B^\flat$ and contained in $B^\sharp$ in such a way that, if $\zeta\in C^\infty_c (\hat{B})$, then $\zeta* \varrho_{\ell (L)}$
and $\zeta*\varrho_{\ell (J)}$ are both supported in $B^\sharp$. Set $\xi:= \bar{h}_{HL}- \bar{h}_{HJ}$ and 
(assuming $\p_{\pi} (x_H)$ is the origin of our system of coordinates)
compute:
\begin{align}
&\int \langle\zeta, \Delta \xi\rangle = - \int D (\bar{h}_{HL} - \bar{h}_{HJ}) : D \zeta = 
\int D \bar\bef_J : D(\zeta * \varrho_{\ell (J)}) - \int D \bar\bef_L : D(\zeta*\varrho_{\ell (L)})\nonumber\\
=& \int \left(D \bar\bef_J : D(\zeta * \varrho_{\ell (J)}) + x^t \cdot \bL \cdot (\zeta*\varrho_{\ell (J)})\right)
- \int \left(D \bar\bef_L : D(\zeta * \varrho_{\ell (L)}) + x^t \cdot \bL \cdot (\zeta*\varrho_{\ell (L)})\right)\, ,\nonumber
\end{align}
where the last line holds for any matrix $\bL$ (with constant coefficients) because $x\mapsto x^t\cdot \bL$ is a linear
function. In particular, we can use the matrix of Proposition \ref{p:pde} to achieve
\begin{align}
\int \langle\zeta, \Delta \xi \rangle &\leq C \bmo\, r^{m+1+\beta_2} \Big(r \|\zeta * \varrho_{\ell(L)}\|_1 + 
r\|\zeta*\varrho_{\ell (J)}\|_1 + \|\zeta*\varrho_{\ell (J)}\|_0 + \|\zeta*\varrho_{\ell (L)}\|_0\Big)\, ,\nonumber
\end{align}
where $\|\cdot\|_0$ and $\|\cdot\|_1$ denote the $C^0$ and $C^1$ norms respectively.
Recalling the inequality $\|\psi*\zeta\|_0 \leq \|\psi\|_{\infty}\|\zeta\|_{L^1}$ and taking into account that $\ell (L)$ and
$\ell (J)$ are both comparable to $r$ (up to a constant depending only on $M_0$ and $m$), we achieve
$\int \langle\zeta , \Delta \xi\rangle \leq C \bmo\, r^{1+\beta_2} \|\zeta\|_{L^1}$.
Taking the supremum over all possible test functions with $\|\zeta\|_{L^1}\leq 1$, we obviously conclude 
$\|\Delta \xi\|_{L^\infty (\hat{B})} \leq C\bmo \,r^{1+\beta_2}$. Observe that a similar estimate could be achieved
for any partial derivative $D^k \xi$ simply using the identity 
\[
\int D(D^k (a*\varsigma)) : Db = - \int Da : (Db * D^k\varsigma)\, .
\]
Summarizing we conclude
\begin{equation}\label{e:laplaciano_derivato}
\|\Delta D^k (\bar{h}_{HL} - \bar{h}_{HJ})\|_{C^0 (\hat{B})} \leq \|\Delta D^k \xi\|_\infty \leq C \bmo r^{1+\beta_2-k}\, ,
\end{equation} 
where the constant $C$
depends upon all the parameters and on $k\in \N$, but not on $\eps_2$, $\bmo$, $H$, $J$ or $L$.
By \cite[Theorem~1.4]{DS3} (cf.~also the proof of Proposition~\ref{p:gira_e_rigira}),
we have $\osc(f_{HL}) + \osc(f_{HJ}) \leq C\, \bmo^{\sfrac{1}{2m}} r$ and, setting $\bE := \bE(T, \bC_{32r_L}(p_L,\pi_H))$
and $\bE' := \bE (T, \bC_{32r_J} (p_J, \pi_H))$, 
\[
\cH^m(\{f_{HL} \neq f_{HJ}\}\cap \hat{B}) \leq C\, [(\bE + \bA^2 r^2) \bE^{\gamma_1} + (\bE' + \bA^2 r^2) \bE'^{\gamma_1}]\, r^m
\leq C\, \bmo^{1+\gamma_1} r^{m+2+\sfrac{\gamma_1}{2}} .
\]
Therefore, taking into account \eqref{e:L1_est}, we conclude $\|\bar h_{HL} - \bar h_{HJ}\|_{L^1 (\hat{B})} \leq C\, \bmo\,r^{m+3+\beta_2}$.
Thus, we appeal to Lemma \ref{l.interpolation} and use the latter estimate together with \eqref{e:laplaciano_derivato} (in the case $k=0$) to get $\|\bar{h}_{HL} - \bar{h}_{HJ}\|_{C^k (B')}\leq 
C \bmo r^{3+\beta_2-k}$ for $k=\{0,1\}$ and for every concentric smaller ball $B'\subset \hat{B}$ 
(where the constant depends also on the ratio between the corresponding radii). 
This implies $\|D (\bar{h}_{HL} - \bar{h}_{HJ})\|_{L^1 (B')} \leq C \bmo r^{m+2+\beta_2}$ and hence
we can use again Lemma \ref{l.interpolation} (based on the case $k=1$ of \eqref{e:laplaciano_derivato})
to conclude $\|\bar{h}_{HL} - \bar{h}_{HJ}\|_{C^2 (B'')}\leq C\bmo r^{1+\beta_2}$. Iterating another two times
we can then conclude $\|\bar{h}_{HL} - \bar{h}_{HJ}\|_{C^k (B^\sharp)} \leq C \bmo r^{3+\beta_2-k}$
for $k\in \{0,1,2,3,4\}$.
By interpolation, since $\kappa \leq \beta_2/4$,  
$\|\bar{h}_{HL} - \bar{h}_{HJ}\|_{C^{3, 2\kappa} (B^\sharp)}\leq C \bmo\, \ell (L)^{2\kappa}$.

Observe now that, since $h_{HL} = (\bar{h}_{HL}, \Psi (x, \bar{h}_{HL}))$ and $h_{HJ} = (\bar{h}_{HJ}, \Psi_H (x, \bar{h}_{HJ}))$, we
deduce the corresponding estimates for $h_{HL}$ and $h_{HJ}$ from the chain rule, namely:
\begin{align}
&\|h_{HL} -h_{HJ}\|_{C^j (B^\sharp)}\leq C \bmo \ell (L)^{3+2\kappa -j} \quad \forall j\in \{0, \ldots , 3\}\\
&\|h_{HL} -h_{HJ}\|_{C^{3,2\kappa} (B^\sharp)}\leq C \bmo\, \ell (L)^{2\kappa}\, .\nonumber
\end{align}
We next want to prove the first estimate of \eqref{e:stime_ricorsive_1} and the first estimate of \eqref{e:stime_ricorsive_2}.
We distinguish two cases. In the first $L$ is adjacent to $H$ and has the same side-length. Let then $J$ be the father of $H$. From the argument above
we then know how to bound $h_H - h_{HJ} = h_{HH} - h_{HJ}$ and $h_{HJ} - h_{HL}$. Both estimates follow then
from the triangle inequality. In the second case $L$ is an ancestor of $H$.
Let then $H=:L_j\subset L_{j-1} \subset \ldots
\subset L = L_i$. We then know how to bound $h_{HL_l} - h_{HL_{l-1}}$ on the ball $B^l := B_{13/2 r_{L_l}} (p_{L_l}, \pi)$.
On the other hand, if the constant in the first inequality of \eqref{e:N0} is large enough (independently of $l$), then $B' \subset B^l$. Summing the corresponding estimates, we get
\begin{align}
\|h_H - h_{HL}\|_{C^{3, 2\kappa} (B')} &\leq C \sum_{l=i}^{j-1}
\|h_{HL_{l}}- h_{HL_{l+1}}\|_{C^{3, 2\kappa} (B^l)}\nonumber\\ 
&\leq C \bmo \ell (L)^{2\kappa} \sum_{l=0}^{j-i-1} 2^{-2\kappa l} \leq C \bmo \ell^{2\kappa}\, , \label{e:geometrica_10}
\end{align}
with a constant $C$ independent of $i$ and $j$. Obviously a similar estimate holds for $\|h_H - h_{HL}\|_{C^j (B')}$.

We still need to prove the second estimate of \eqref{e:stime_ricorsive_1} and the second estimate of \eqref{e:stime_ricorsive_2}. If $H$ is a fixed cube in the Whitney decomposition and $L_{N_0}\in \sS^{N_0}$ its biggest ancestor, we then have $\|h_H - h_{HL_{N_0}}\|_{C^{3,2\kappa} (B')} \leq C \bmo$. On the other hand
\[
\|D \bef_{HL_{N_0}}\|^2_{L^2 (B^{N_0})} 
\leq \D (f_{HL_{N_0}}) \leq C \bE (T, \bC_{32 r_{L_{N_0}}} (p_{L_{N_0}}, \pi_H)) \leq C \bmo + C |\pi_H-\pi_0|^2 \leq C \bmo\, .
\]
Thus, by standard convolution estimates, $\|D \bar{h}_{H L_{N_0}}\|_{C^k (B^{N_0})} \leq C \bmo^{\sfrac{1}{2}}$ (where
the constant $C$ depends on $k\in \N$ and on he various parameters). Using 
\eqref{e:geometrica_10} we then get $\|D\bar{h}_H\|_{C^{2, 2\kappa} (B')} \leq \|D \bar h_H - D \bar h_{HL_{N_0}}\|_{C^{2, 2\kappa} (B')} + \|D\bar h_{HL_{N_0}}\|_{C^{2, 2\kappa} (B^{N_0})}\leq C \bmo^{\sfrac{1}{2}}$. By the chain rule and the regularity of $\Psi$ we then conclude the general bound $\|Dh_H\|_{C^{3, 2\kappa} (B')} \leq C \bmo^{\sfrac{1}{2}}$. This implies the existence of a constant $\xi$ such that $\|h_H - \xi\|_{C^{3, 2\kappa} (B')}\leq C \bmo^{\sfrac{1}{2}}$. Applying Lemma \ref{l:rotazioni_semplici} we achieve the bound $\|g_H - \zeta\|_{C^{3, 2\kappa} (B)} \leq C \bmo^{\sfrac{1}{2}}$ for some other constant $\zeta$. With a similar argument using the bound $\|\bar h_{H{L_{N_0}}}\|_{C^0 (B^{N_0})}\leq C \bmo^{\sfrac{1}{2m}}$, we achieve
$\|\bar h_H\|_{C^0 (B')} \leq  C \bmo^{\sfrac{1}{2m}}$. Hence again by Lemma \ref{l:rotazioni_semplici} $\|g_H\|_{C^0 (B)} \leq C \bmo^{\sfrac{1}{2m}}$. This shows obviously Proposition \ref{p:stime_chiave}(i). 

Next, observe that we have, by the very same arguments, $\|g_{HL} - \zeta\|_{C^{3, 2\kappa} (B)} \leq C \bmo^{\sfrac{1}{2}}$, thus concluding
$\|g_{HL}-g_H\|_{C^{3,2\kappa} (B)} \leq C \bmo^{\sfrac{1}{2}}$. On the other hand, it also follows from the same arguments above that $\|h_{HL}- h_H\|_{L^1 (B')} \leq C \bmo \ell^{m+3+\beta_2} \leq C \bmo \ell^{m+3+4\kappa}$. Applying Lemma \ref{l:rotazioni_semplici}(b) we then conclude $\|g_{HL}- g_H\|_{L^1 (B)} \leq C \bmo \ell^{m+3+4\kappa}$. We can now apply Lemma \ref{l:interpolation_bis} to conclude that $\|D^i (g_{HL} - g_H)\|_{C^0 (B)} \leq C \bmo^{\sfrac{1}{2}} \ell^{3-i + 4\kappa}$ for every $i\in \{0,1,2,3\}$, reaching the second estimate of \eqref{e:stime_ricorsive_1}. Interpolating between the latter estimates and $\|g_{HL}-g_H\|_{C^{3,2\kappa} (B)} \leq C \bmo^{\sfrac{1}{2}}$ we reach as well the second conclusion of \eqref{e:stime_ricorsive_2}.

Coming to (iv) in Proposition \ref{p:stime_chiave}, the estimate on
$g_H-y_H$ is a straightforward consequence
of the height bound, {\cite[Theorem~1.4]{DS3}} 
and Lemma~\ref{l:rotazioni_semplici}
(applied to $h_H$).
Next, observe that 
\[
\|D h_H\|^2_{L^2 (B')} \leq C (1+ \Lip (\Psi_H)) \|D \bar{h}_H\|_{L^2 (B')}^2 + C\|D_x \Psi_H (x, \bar{h})\|_{L^2 (B')}^2\, 
\]
and
\[
\|D \bar{h}_H\|_{L^2 (B')}^2 \leq C\|D (\etaa\circ f_H)\|_{L^2 (B')}^2 \leq C \D (f_H, B')) \leq C \bmo \ell (H)^{m + 2-2\delta_2}\, .
\]
On the other hand recall that the plane $\pi_H$ is contained in the plane $T_{p_H} \Sigma$ and thus $D_x \Psi_H (p_H, 0) = 0$. Since $\|D^2 \Psi_H\|_0 \leq \bmo^{\sfrac{1}{2}}$ we obviously conclude that $\|D_x \Psi_H (x, \bar{h})\|_{L^2}^2 \leq C \bmo \ell^{m+2}$. Therefore $\|D h_H\|_{L^2 (B')}^2 \leq C \bmo \ell^{m+2-2\delta_2}$.

Thus, there is at least one point $q\in\gr (h_H|_{B'})$ such that $|T_q \bG_{h_H} - \pi_H| \leq C \bmo^{\sfrac{1}{2}}
\ell (H)^{1-\delta_2}$. 
Since $\|D^2 h_H\|_0 \leq C \bmo^{\sfrac{1}{2}}$, we then conclude
that $|T_{q'} \bG_{h_H} - \pi_H|\leq C \bmo^{\sfrac{1}{2}} \ell (H)^{1-\delta_2}$ holds indeed for any point $q'\in\gr (h_H|_{B'})$.
Since $\gr (g_H|_B)$ is a subset of $\gr (h_H|_{B'})$ (with the same orientation!), the second inequality of Proposition
\ref{p:stime_chiave}(iv) follows.
\end{proof}

\subsection{Tilted $L^1$ estimate}
In order to achieve Proposition~\ref{p:stime_chiave}(ii) and (iii), we need to
compare tilted interpolating functions coming from different coordinates. To this aim, we set the following terminology.

\begin{definition}[Distant relation]\label{d:condizione_(Q)}
Four cubes $H, J, L, M$ make a distant relation between $H$ and $L$ if
$J, M\in \sS^j\cup \sW^j$ have nonempty intersection, $H$ is a descendant of $J$ (or $J$
itself) and $L$ a descendant of $M$ (or $M$ itself).
%
\end{definition}

\begin{lemma}[Tilted $L^1$ estimate]\label{l:tilted_L1}
Assume the hypotheses of Proposition \ref{p:stime_chiave} hold and $\eps_2$ is sufficiently small.
Let $H, J, L$ and $M$ be a distant relation between $H$ and $L$,
and let $h_{HJ}$, $h_{LM}$ be the maps given in Definition \ref{d:mappe_h_HL}.
Then there is a map $\hat{h}_{LM}: B_{4r_J} (p_J, \pi_H)\to \pi_H^\perp$ such that
$\bG_{\hat{h}_{LM}} = \bG_{h_{LM}} \res \bC_{4r_J} (p_J, \pi_H)$ and, for $C= C(\beta_2, \delta_2, M_0, N_0, C_e, C_h)$,
\begin{equation}\label{e:tilted_L1}
\|h_{HJ} - \hat{h}_{LM}\|_{L^1 (B_{2r_J} (p_J, \pi_H))} \leq C \bmo\, \ell (J)^{m+3+\sfrac{\beta_2}{2}}\, .
\end{equation}
\end{lemma}

\begin{proof} As in the previous proofs we follow the convention that $C_0$ denotes geometric constants whereas
$C$ denotes constants which depend upon $\beta_2, \delta_2, M_0, N_0, C_e$ and $C_h$. First observe that Lemma \ref{l:rotazioni_semplici} can be applied because, by Proposition \ref{p:tilting opt},
\[
|\pi_H-\pi_L|\leq |\pi_H- \pi_J| + |\pi_J-\pi_M|+ |\pi_M-\pi_L| \leq C \bmo^{\sfrac{1}{2}} \ell (J)^{1-\delta_2}. 
\]
Set $\pi:= \pi_H$ and let $\varkappa$ be its orthogonal complement
in $T_{p_H} \Sigma$, and similarly
$\bar{\pi}= \pi_L$ and $\bar{\varkappa}$ its orthogonal in $T_{p_L} \Sigma$.
After a translation we also assume $p_J=0$, and write $r=r_J=r_M$, 
$\ell=\ell (J)=\ell (M)$ and 
$E:= \bE (T, \bC_{32r} (0, \pi))$, $\bar{E}:= \bE (T, \bC_{32r} (p_M, \bar\pi))$.
Recall that $\max\{E, \bar{E}\}\leq C \bmo \,\ell^{2-2\delta_2}$. 
We fix also the maps $\Psi_H : T_{p_H} \Sigma \to T_{p_H} \Sigma^\perp$ and 
$\Psi_L: T_{p_L} \Sigma \to T_{p_L} \Sigma^\perp$ whose
graphs coincide with the submanifold $\Sigma$. Observe that 
$|\pi-\bar\pi| + |\varkappa- \bar\varkappa| \leq C \bmo^{\sfrac{1}{2}} \ell^{1-\delta_2}$,
$\|D \Psi_H\|_{C^{2,\eps_0}}
+\|D \Psi_L\|_{C^{2,\eps_0}} \leq C \bmo^{\sfrac{1}{2}}$ and
\[
\ell^{-1} \left(\|\Psi_H\|_{C^0 (B_{8r})} + \|\Psi_L\|_{C^0 (B_{8r})}\right) + \|D\Psi_H\|_{C^{0}(B_{8r})}
+\|D\Psi_L\|_{C^{0}(B_{8r})} \leq C \bmo^{\sfrac{1}{2}}\ell.
\]

Consider the map $\hat{f}_{LM}: B_{4r} (0, \pi) \to \Iq (\pi^\perp)$ such that $\bG_{\hat{f}_{LM}} =
\bG_{f_{LM}} \res \bC_{4r} (0, \pi)$, which exists by \cite[Proposition~5.2]{DS2}. Recalling the estimates therein and those of
\cite[Theorem~1.4]{DS3},
\begin{gather}
\Lip (f_{HJ})+ \Lip (\hat{f}_{LM}) \leq C \bmo^{\gamma_1} \ell^{\gamma_1}
\quad \text{and} \quad |f_{HJ}| + |\hat{f}_{LM}| \leq C \bmo^{\sfrac{1}{2m}} \ell^{1+\beta_2},\label{e:Lip+C0_10}\\
\D (f_{HJ})+ \D (\hat{f}_{LM}) \leq C \, \bmo \, \ell^{m+2-2\delta_2}\, .\label{e:Dir_10}
\end{gather}
Consider next the projections $A$ and $\hat{A}$ onto $\pi$
of the Borel sets $\gr (f_{HJ})\setminus \supp (T)$ and
$\gr (\hat{f}_{LM})\setminus \supp (T)$. We know from \cite[Theorem~1.4]{DS3} that 
\begin{equation}\label{e:differs}
|A\cup \hat{A}| \leq C \left[\|\bG_{f_{HJ}} -T\| (\bC_{32} (0, \pi)) + \|\bG_{\hat{f}_{LM}} -T\|(\bC_{32} (p_M, \bar{\pi}))\right]
\leq C \bmo^{1+\gamma_1}\, \ell^{m+2+\gamma_1}\, .
\end{equation}
Recall that
\begin{align*}
h_{HJ} &=(\p_{\varkappa} ((\etaa\circ f_{HJ}) * \varrho_\ell), \Psi_H (x, \p_{\varkappa} ((\etaa\circ f_{HJ}) * \varrho_\ell)))\\
h_{LM} &= ( \p_{\bar\varkappa} ((\etaa\circ f_{LM}) * \varrho_\ell), \Psi_L (x, \p_{\bar\varkappa} ((\etaa\circ f_{LM}) * \varrho_\ell)))\, 
\end{align*}
and define in addition the maps
\begin{align*}
&\bef_{HJ} = (\p_{\varkappa} (\etaa\circ f_{HJ}) , \Psi_H (x, \p_{\varkappa} (\etaa\circ f_{HJ})))\, \\
&\bef_{LM} = (\p_{\bar\varkappa} (\etaa\circ f_{LM}), \Psi_L (x, \p_{\bar\varkappa} (\etaa\circ f_{LM})))\, .
\end{align*} 
Recall that $\hat{h}_{LM}: B_{4r} (0, \pi)\to
\pi^\perp$ satisfies $\bG_{\hat{h}_{LM}}  = \bG_{h_{LM}} \res \bC_{4r} (0, \pi)$ and let
$\hat{\bef}_{LM}$ be such that 
$\bG_{\hat{\bef}_{LM}} = \bG_{\bef_{LM}} \res \bC_{4r} (0, \pi)$. We use Proposition \ref{p:pde},
the Lipschitz regularity of $\Psi_H$ and
Lemma \ref{l:rotazioni_semplici} to conclude
\[
\|\hat{h}_{LM} - \hat{\bef}_{LM}\|_{L^1} \leq C \|h_{LM} - \bef_{LM}\|_{L^1}\leq C \bmo\, r^{m+3+\beta_2}.
\]
Likewise $\|h_{HJ}-\bef_{HJ}\|_{L^1} \leq C \bmo r^{m+3+\beta_2}$.
We therefore need to estimate $\|\bef_{HJ} - \hat{\bef}_{LM}\|_{L^1}$. Define next the map 
$\beg_{LM} = (\p_{\varkappa} (\etaa \circ \hat{f}_{LM}), \Psi_H (x, \p_{\varkappa} (\etaa \circ \hat{f}_{LM})))$ and observe that
$\|\beg_{LM}-\bef_{HJ}\|_{L^1} \leq C \|\etaa \circ \hat{f}_{LM} - \etaa \circ f_{HJ}\|_{L^1}$. On the other hand,
since the two maps $\hat{f}_{LM}$ and $f_{HJ}$ differ only on $A\cup \bar{A}$, we can estimate 
\[
\|\etaa \circ \hat{f}_{LM} - \etaa \circ f_{HJ}\|_{L^1} \leq C|A\cup \bar{A}| (\|f_{LM}\|_\infty+\|\hat{f}_{HJ}\|_\infty) \leq C \bmo^{1+\sfrac{1}{2m}} \ell^{3+m +\gamma_1 + \beta_2}\, .
\]
It thus
suffices to estimate $\|\beg_{LM} - \hat\bef_{LM}\|_{L^1}$. This estimate is independent of the rest 
and it is an easy consequence of \eqref{e:che_fatica} in Lemma \ref{l:cambio_tre_piani} below.
\end{proof}

\begin{lemma}\label{l:cambio_tre_piani}
Fix $m,n,l$ and $Q$. There are geometric constants $c_0, C_0$ with the following property.
Consider two triples of planes $(\pi, \varkappa, \varpi)$ and $(\bar\pi, \bar\varkappa, \bar\varpi)$, where
\begin{itemize}
\item $\pi$ and $\bar\pi$ are $m$-dimensional;
\item $\varkappa$ and $\bar\varkappa$ are
$\bar{n}$-dimensional and orthogonal, respectively, 
to $\pi$ and $\bar\pi$;
\item $\varpi$ and $\bar\varpi$ $l$-dimensional and orthogonal, respectively, to $\pi\times \varkappa$ and $\bar\pi\times
\bar\varkappa$.
\end{itemize}
Assume ${\rm An} := |\pi-\bar\pi| + |\varkappa-\bar\varkappa|\leq c_0$ and
let $\Psi: \pi\times \varkappa \to \varpi$,
$\bar\Psi: \bar\pi\times \bar\varkappa \to \bar\varpi$ be two maps
whose graphs coincide and such that $|\bar \Psi (0)| \leq c_0 r$ and $\|D\bar \Psi\|_{C^{0}} \leq c_0$.
Let $u: B_{8r} (0, \bar{\pi}) \to \Iq (\bar\varkappa)$ be a map with
$\Lip (u) \leq c_0$ and $\|u\|_{C^0} \leq c_0 r$
and set $f (x)= \sum_i \llbracket(u_i (x), \bar\Psi (x, u_i (x)))\rrbracket$ and $\bef (x) = 
(\etaa \circ u (x), \bar\Psi (x, \etaa \circ u (x)))$. Then there are
\begin{itemize}
\item a map $\hat{u}: B_{4r} (0, \pi) \to \Iq (\varkappa)$ such
that the map $\hat{f} (x) := \sum_i \a{(\hat{u}_i (x), \Psi (x, \hat{u}_i (x)))}$ satisfies $\bG_{\hat{f}}
= \bG_{f} \res \bC_{4r} (0, \pi)$
\item and a map $\hat\bef: B_{4r} (0, \pi) \to \varkappa \times \varpi$ such that
$\bG_{\hat\bef} = \bG_{\bef} \res \bC_{4r} (0, \pi)$. 
\end{itemize} 
Finally, if $\beg (x) := (\etaa \circ \hat{u} (x),
\Psi (x, \etaa \circ \hat{u} (x)))$, then
\begin{align}
\|\hat\bef-\beg\|_{L^1} &\leq C_0 \left(\| f \|_{C^0}+ r {\rm An}\right) \big(\D (f) + r^m \big(\|D\bar \Psi\|^2_{C^0} + {\rm An}^2\big)\big)\, .\label{e:che_fatica}
\end{align}
\end{lemma}

The proof of the lemma is quite long and we defer it to Appendix \ref{a:cambio_tre_piani}.

\subsection{Proof of Proposition \ref{p:stime_chiave}}
We are finally ready to complete the proof of
Proposition~\ref{p:stime_chiave}. Recall that (i) and (iv) have already been shown in Lemma \ref{l:stime_iterative}.
In order to show (ii) fix two cubes $H, L\in \sP^j$ with nonempty intersection. 
If $\ell(H) = \ell(L)$, then we can apply Lemma \ref{l:tilted_L1} to conclude 
\begin{equation}\label{e:L1_tilt_10}
\|h_{HH}-\hat{h}_{LL}\|_{L^1 (B_{2r_H} (p_H, \pi_H))} \leq C \bmo\, \ell (H)^{m+3+\sfrac{\beta_2}{2}}
\leq C \bmo\, \ell (H)^{m+3+2\kappa}\, .
\end{equation} 
If $\ell (H) = \frac{1}{2} \ell (L)$, then let $J$ be the father of $H$.
Obviously,  $J \cap L\neq \emptyset$.
We can therefore apply Lemma \ref{l:tilted_L1} above to infer
$\|h_{HJ} - \hat{h}_{LL}\|_{L^1 (B_{2r_J} (p_J, \pi_H)}\leq C \bmo \,\ell (J)^{m+3+\sfrac{\beta_2}{2}}$. On the other hand, 
by Lemma \ref{l:stime_iterative}, 
$\|h_{HH}-h_{HJ}\|_{L^1 (B_{2r_H} (p_H, \pi_H))} \leq C r^m \|h_{H}-h_{HJ}\|_0 
\leq C \bmo\, \ell (J)^{m+3+2\kappa}$. Thus we conclude \eqref{e:L1_tilt_10} as well. Note
that $\bG_{g_L} \res \bC_{r_H} (x_H, \pi_0) = \bG_{\hat{h}_{LL}} \res \bC_{r_H} (x_H, \pi_0)$ and that the same
property holds with $g_H$ and $h_{HH}$. We can thus appeal to Lemma \ref{l:rotazioni_semplici} to conclude
\begin{equation}\label{e:real_L1}
\|g_H-g_L\|_{L^1 (B_{r_H} (p_H, \pi_0))} \leq C \bmo\, \ell (H)^{m+3+2\kappa}\, .
\end{equation} 
However, recall also that $\|D^3 (g_H-g_L)\|_{C^\kappa (B_{r_H} (p_H, \pi_0))} \leq C \bmo^{\sfrac{1}{2}}$. We can then
apply Lemma \ref{l:interpolation_bis} to conclude (ii). 

Now, if $L\in \sW^j$ and $i\geq j$,
consider the subset $\sP^i (L)$ of all cubes in $\sP^i$ which intersect $L$. If $L'$ is the cube concentric to $L$
with $\ell (L')=\frac{9}{8} \ell (L)$,
we then have by definition of $\varphi_j$:
\begin{equation}\label{e:L1_importante}
\|\varphi_i - g_L\|_{L^1 (L')} \leq C \sum_{H\in \sP^i (L)} \|g_H - g_L\|_{L^1(B_{r_L} (p_L, \pi_0))}
\leq C \bmo\, \ell (H)^{m+3+2\kappa}\, ,
\end{equation}
which is the claim of (v).

As for (iii), observe first that the argument above applies also when $L$ is the father of $H$. Iterating
then the corresponding estimates, it is easy to see that
\begin{equation}
|D^3 g_H (x_H) - D^3 g_J (x_J)|\leq C \bmo^{\sfrac{1}{2}} \ell (J)^\kappa\quad \mbox{for any ancestor $J$ of $H$}\, .  
\end{equation}
Fix now any pair $H, L\in \sP^j$. Let $H_i$, $L_i$ be the ``first ancestors'' of $H$ and $L$ which are adjacent,
i.e. among all pairs $H', L'$ of ancestors of $H$ and $L$ with same side-length and nonempty intersection, 
we assume that the side-length $\ell$ of $H_i, L_i$ is the smallest possible. We can therefore use the estimates obtained so far to conclude
\begin{align*}
|D^3 g_H (x_H) - D^3 g_L (x_L)| \leq &\; |D^3 g_H (x_H) - D^3 g_{H_i} (x_{H_i})| +
|D^3 g_{H_i} (x_{H_i}) - D^3 g_{L_i} (x_{L_i})|\\
& + |D^3 g_{L_i} (x_{L_i}) - D^3 g_L (x_L)| \leq C \bmo^{\sfrac{1}{2}} \ell^\kappa.
\end{align*}
A simple geometric consideration shows that $|x_L - x_H|\geq c_0 \ell$,
where $c_0$ is a dimensional constant, thus completing the proof.

\section{Existence and estimates for the $\cM$-normal approximation}

In this section we continue using the convention that $C$ denotes constants which depend
upon $\beta_2, \delta_2, M_0, N_0, C_e$ and $C_h$, whereas $C_0$ denotes geometric constants.

\subsection{Proof of Corollary \ref{c:cover}} The first two statements of (i) follow immediately from
Theorem \ref{t:cm}(i) and Proposition~\ref{p:tilting opt}(v).
Coming to the third claim of (i), we extend the function
$\phii$ to the entire plane $\pi_0$ by increasing its $C^{3, \kappa}$ norm
by a constant geometric factor. Let $\phii_t (x):= t \phii (x)$ for $t\in [0,1]$,
$\cM_t := \gr (\phii_t|_{]-4,4[^m})$ and set
\[
\bU_t := \{x+y : x\in \cM_t, y\perp T_x \cM_t, |y|<1\}\, .
\]
For $\eps_2$ sufficiently small the orthogonal projection $\p_t: \bU_t \to \cM_t$ is
a well-defined $C^{2, \kappa}$ map for every $t\in [0,1]$, which depends smoothly on $t$. It is also easy
to see that $\partial T \res \bU_t = 0$.
Thus, $(\p_t)_\sharp (T\res \bU_t) = Q(t) \a{\cM_t}$
for some integer $Q(t)$. On the other hand these currents depend continuously on $t$ and therefore
$Q(t)$ must be a constant. Since $\cM_0 = ]-4,4[^m \times \{0\} \subset \pi_0$ and $\p_0 = \p_{\pi_0}$,
we conclude $Q(0) =Q$.

With regard to (ii), consider 
$q\in L\in \sW$, set $p:= \Phii (q)$ and $\pi := T_p \cM$, whereas $\pi_L$
is as in Definition \ref{d:glued}. Let $J$ be the cube concentric to $L$ and
with side-length $\frac{17}{16} \ell (L)$. By the definition of $\phii$, Theorem \ref{t:cm}(ii) and Proposition~\ref{p:stime_chiave}, we have that, denoting by $\bar\phii$ and $\bar{g}_L$ the first $\bar{n}$ components
of the corresponding maps,
\[
\|\bar\phii - \bar{g}_L\|_{C^0 (J)} \leq \sum_{H\in \sW, H\cap L\neq\emptyset}
\|g_L-g_H\|_{C^0 (J)} \leq C \bmo^{\sfrac{1}{2}} \ell (L)^{3+\kappa}\, .
\]
So, since $\phii = (\bar\phii, \Psi (x, \bar\phii))$ and $g_H = (\bar g_H, \Psi (x, \bar g_H))$,
we conclude $\|g_L-\phii\|_{C^0 (J)} \leq C \bmo^{\sfrac{1}{2}} \ell (L)^{3+\kappa}$. 
On the other the graph of $g_L$ coincides with the graph of the tilted interpolating function $h_L$. 
Consider in $\bC:= \bC_{8 r_L} (p_L, \pi_L)$ the $\pi_L$-approximation $f_L$ used in the construction algorithm
and recall that, by \cite[Theorem 1.4]{DS3}.
\begin{align*}
{\rm osc}\, (f_L) &\leq C_0 \left(\bh (T, \bC_{32 r_L} (p_L, \pi_L), \pi_L) +
((\bE (T, \bC_{32r_L} (p_L, \pi_L))^{\sfrac{1}{2}} + r_L \bA) r_L\right)\\
&\leq C \bmo^{\sfrac{1}{2m}} \ell (L)^{1+\beta_2}\, .
\end{align*}
Recall that $p_L = (z_L, w_L) \in \pi_L\times \pi_L^\perp$ belongs to $\supp (T)$, so we easily conclude that $\|\etaa\circ f_L - w_L\|_{C^0}\leq
C \bmo^{\sfrac{1}{2m}} \ell (L)^{1+\beta_2}$. This implies $\|h_L - w_L\|_{C^0} \leq C \bmo^{\sfrac{1}{2m}} \ell (L)^{1+\beta_2}$.
Putting all these estimates together, we easily conclude that, for any point $p$ in $\supp (T)\cap \bC_{7r_L} (p_L, \pi_L)$
the distance to the graph of $h_L$ is at most $C \bmo^{\sfrac{1}{2m}} \ell (L)^{1+\beta_2}$.
{ This shows the claim if we can prove that $\supp (\langle T, \p, p)\subset \bB_{r_L} (p) \subset \bC_{7r_L} (p_L, \pi_L)$, for which we argue by contradiction. Assuming the opposite, there is a $p'\in \supp (\langle T, \p, p)$ and an ancestor $J$ with largest sidelength among those for which $|p'-p|\geq r_J$. Let $\pi$ be the tangent to $\cM$ at $p$ and observe that we have the estimates $|\pi-\pi_J| \leq C \bmo^{\sfrac{1}{2}}$ and $|\pi-\pi_0|\leq C \bmo^{\sfrac{1}{2}}$. If $J$ were an element of $\sS^{N_0}$, the height bound \eqref{e:pre_height} would imply $|p'-p|\leq C\bmo^{\sfrac{1}{2m}}$. If $J\not\in \sS^{N_0}$ and we let $H$ be the father of $J$, we then conclude that $q\in \bB_H$ and thus we have 
$|p'-p|\leq C \bh (T, \bB_H)\leq C \bmo^{\sfrac{1}{2m}} \ell (H)^{1+\beta_2}$. In both cases this would be incompatible with $|p'-p|\geq r_J$, provided $\eps_2 \leq c (\beta_2, \delta_2, M_0, N_0, C_e, C_h)$}

Finally, we show (iii). Fix a point $p\in \bGam$. By construction, there is
an infinite chain $L_{N_0} \supset L_{N_0+1}\supset \ldots \supset L_j \supset \ldots$ of
cubes $L_j\in\sS^j$ such that $\{p\}= \bigcap_j L_j$.
Set $\pi_j:=\pi_{L_j}$. From Proposition~\ref{p:tilting opt} we infer that
the planes $\pi_j$ converge to a plane $\pi$ with a rate $|\pi_j-\pi|\leq C \bmo^{\sfrac{1}{2}} 2^{-j (1-\delta_2)}$.
Moreover, the rescaled currents $(\iota_{p_{L_j}, 2^{-j}})_\sharp T$ (where the map $\iota_{q,r}$ 
is given by $\iota_{q,r} (z) = \frac{z-q}{r}$)
converge to $Q\a{\pi}$.
Since $|\Phii (p)-p_{L_j}| \leq C \sqrt{m} \,2^{-j}$ for some constant $C$ independent of $j$, we
easily conclude that $\Theta (T, \Phii (p)) = Q$ and $Q\a{\pi}$ is the unique tangent cone to
$T$ at $\Phii (p)$.
We next show that $\p^{-1} (\Phii (p)) \cap \supp (T) = \{\Phii (p)\}$. 
Indeed, assume there were $q\neq \Phii(p)$ which belongs to $\supp (T)$ and such that $\p(q) = \Phii(p)$.
Let $j$ be such that $2^{-j-1}\leq 
|\Phii (p)-q|\leq 2^{-j}$. Provided $\eps_2$ is sufficiently small, Proposition~\ref{p:tilting opt}(v) guarantees
that $j\geq N_0$. Consider the cube $L_j$ in the chain above 
and recall that $\bh (T, \bC_{32r_{L_j}} (p_{L_j}, \pi_j))
\leq C \bmo^{\sfrac{1}{2m}} 2^{-j (1+\beta_2)}$. Hence,
\begin{align*}
2^{-j-1} &\leq |q-\Phii (p)| = |\p_{\pi^\perp} (q-\Phii (p))| \leq C_0 |q-\Phii (p)| |\pi-\pi_j| + 
\bh (T, \bC_{32r_{L_j}} (p_{L_j}, \pi_j))\nonumber\\
&\leq C \bmo^{\sfrac{1}{2}} 2^{-j (1-\delta_2)} 2^{-j} + C \bmo^{\sfrac{1}{2m}} 2^{-j (1+\beta_2)}
\leq C \eps_2^{\sfrac{1}{2m}} 2^{-j}\, ,
\end{align*}
which, for an appropriate choice of $\eps_2$ (depending only on the various other parameters $\beta_2, \delta_2, C_e, C_h, M_0, N_0$) is a contradiction.

\subsection{Construction of the $\cM$-normal approximation and first estimates} 
We set $F(p) = Q\a{p}$ for $p\in \Phii (\bGam)$. 
For every $L\in \sW^j$ consider the $\pi_L$-approximating function $f_L: \bC_{8r_L} (p_L, \pi_L)\to
\Iq (\pi_L^\perp)$ of Definition~\ref{d:pi-approximations}
and $K_L\subset B_{8r_L} (p_L, \pi_L)$ the projection on $\pi_L$ of $\supp (T)\cap \gr (f_L)$.
In particular we have $\bG_{f_L\vert_{K_L}}= T\res(K_L \times \pi_L^\perp)$.
We then denote by $\mathscr{D} (L)$ the portions of the supports of $T$ and $\gr (f_L)$ which differ:
\[
\mathscr{D} (L) := (\supp (T) \cup \gr (f_L))\cap \big[(B_{8r_L} (p_L, \pi_L) \setminus K_L)\times \pi_L^\perp\big]\, . 
\]
Observe that, by \cite[Theorem 1.4]{DS3} and Assumption \ref{parametri}, 
we have 
\begin{equation}\label{e:poco_scazzo}
\cH^m (\mathscr{D} (L)) + \|T\| (\mathscr{D} (L)) \leq C_0 E^{\gamma_1} (E + \ell (L)^2 \bA^2) \ell(L)^m
\leq C \bmo^{1+\gamma_2} \ell (L)^{m+2+\gamma_2}\, ,
\end{equation}
where $E=\bE (T, \bC_{32 r_L} (p_L, \pi_L))$ (cf. with \eqref{e:mu}).
Let $\cL$ be the Whitney region in Definition~\ref{d:cm} and
set $\cL':= \Phii (J)$ where $J$ is the cube concentric to $L$ with
$\ell (J) = \frac{9}{8} \ell (L)$.
Observe that our choice of the constants is done in such a way that, 
\begin{align}
&L\cap H = \emptyset \quad \iff \quad
\mathcal{L}' \cap \mathcal{H}' = \emptyset\qquad \forall H, L \in \sW\, ,\label{e:non_si_toccano}\\
&\Phii (\bGam) \cap \cL' = \emptyset \qquad \forall L\in \sW\, .\label{e:nonsitoccano_2}
\end{align}
We then apply \cite[Theorem 5.1]{DS2} to obtain maps $F_L: \cL' \to \Iq (\bU)$, $N_L : \cL'\to \Iq (\R^{m+n})$ with the following
poperties:
\begin{itemize}
\item $F_L (p) = \sum_i \a{p+(N_L)_i (p)}$,
\item $(N_L)_i (p) \perp T_p \cM$ for every $p\in \cL'$ 
\item and $\bG_{f_L}\res (\p^{-1} (\cL')) = \bT_{F_L} \res (\p^{-1} (\cL'))$.
\end{itemize}
For each $L$ consider the set $\sW (L)$ of elements in $\sW$ which have a nonempty intersection with $L$.
We then define the set $\cK$ in the following way:
\begin{equation}\label{e:def_cK}
\cK = \cM \setminus \Big(\bigcup_{L\in \sW} \Big(\cL' \cap \bigcup_{M\in \sW (L)} \p (\mathscr{D} (M))\Big)\Big)\, . 
\end{equation}
In other words $\cK$ is obtained from $\cM$ by removing in each $\cL'$ those points $x$ for which there is
a neighboring cube $M$ such that the
slice of $\bT_{F_M}$ at $x$ (relative to the projection $\p$) does
not coincide with the slice of $T$.
Observe that, by \eqref{e:nonsitoccano_2}, $\cK$ contains necessarily $\Phii (\bGam)$.
Moreover, recall that $\Lip(\p)\leq C$, that the cardinality $\sW (L)$ is 
bounded by a geometric constant and that each element of $\sW (L)$ has side-length at most twice that
of $L$. Thus \eqref{e:poco_scazzo} implies
\begin{equation}\label{e:bound_cK}
|\cL\setminus \cK|\leq |\cL'\setminus \cK|\leq \sum_{M\in \sW (L)} \sum_{H\in \sW (M)} 
\p (\mathscr{D} (H)) \leq C \bmo^{1+\gamma_2} \ell (L)^{m+2+\gamma_2}\, .
\end{equation}
On $\Phii (\bGam)$ we define $F (p) = Q\a{p}$. By \eqref{e:non_si_toccano}, if $J$ and $L$ are such that $\cJ'\cap \cL'\neq \emptyset$, then $J\in \sW (L)$ and therefore $F_L= 
F_J$ on $\cK\cap (\cJ'\cap \cL')$. We can therefore define a unique map on $\cK$ by simply setting
$F (p) = F_L (p)$ if $p\in \cK\cap \cL'$. Our resulting map 
has the Lipschitz bound of \eqref{e:Lip_regional} in each $\cL\cap \cK$.
{ Indeed, notice that, by the $C^2$ estimate on $\phii$ and Proposition~\ref{p:stime_chiave}(iv), $\cM$ is given on $\bC_{r_L} (p_L, \pi_L)$ as the graph of a map $\phii': B_{r_L} (p_L, \pi_L)\to \pi_L^\perp$ with $\|D\phii'\|_{C^0} \leq C \bmo^{\sfrac{1}{2}} \ell (H)^{1-\delta_1}$ and $\|D^2 \phii'\|_{C^0} \leq C \bmo^{\sfrac{1}{2}}$. Hence, the Lipschitz constant of $N_L$ can be 
estimated using \cite[Theorem 5.1]{DS2} as
\begin{equation}\label{e:Lip_aggiunto_luca}
 \Lip(N_L)\leq C\, \left(\|D^2\phii'\|_{C^0}\,\|N\|_{C^0}+\|D\phii'\|_{C^0}+\Lip(f_L)\right)
\leq C\, \bmo^{\gamma_2}\,\ell(L)^{\gamma_2}\,,
\end{equation}}

Moreover, $\bT_F = T\res \p^{-1} (\cK)$, which implies two facts. First, by Corollary \ref{c:cover}(ii) 
we also have that $N (p) := \sum_i \a{F_i (p) -p}$ enjoys
the bound $\|N|_{\cL\cap \cK}\|_{C^0} \leq C \bmo^{\sfrac{1}{2m}} \ell (L)^{1+\beta_2}$. Secondly,
\begin{equation}\label{e:bound_scazzo_corrente}
\|T\| (\p^{-1} (\cL\setminus \cK)) \leq \sum_{M\in \sW (L)} \sum_{H\in \sW (M)} \| T\| (\mathscr{D} (H))
\leq C \bmo^{1+\gamma_2} \ell (L)^{m+2+\gamma_2}\, .
\end{equation}
Hence, $F$ and $N$ satisfy the bounds \eqref{e:Lip_regional} on $\cK$. We next 
extend them to  the whole center manifold and conclude \eqref{e:err_regional}
from \eqref{e:bound_scazzo_corrente} and \eqref{e:bound_cK}.
The extension is achieved in three steps:
\begin{itemize}
\item we first extend the map $F$ to a map $\bar{F}$ taking values
in $\Iq (\bU)$;
\item we then modify $\bar{F}$ to achieve the form $
\hat{F} (x) = \sum_i \llbracket x+\hat{N}_i (x)\rrbracket$
with $\hat{N}_i (x) \perp T_x \cM$ for every $x$;
\item we finally modify $\hat{F}$ to reach the desired extension $F (x) =
\sum_i \a{x+ N_i (x)}$, with $N_i (x) \perp T_x \cM$ and
$x+ N_i (x) \in \Sigma$ for every $x$.
\end{itemize}

\medskip

{\bf First extension}. We use on $\cM$ the coordinates
induced by its graphical structure,
i.e.~we work with variables in flat domains.
Note that the domain parameterizing the
Whitney region for $L\in \sW$ is then the cube concentric to $L$ and
with side-length $\frac{17}{16} \ell (L)$.  
The multivalued map $N$ is extended to a multivalued $\bar{N}$ inductively to appropriate
neighborhoods of the skeleta of the Whitney decomposition (a similar argument has been
used in 
\cite[Section 1.2.2]{DS1}). The extension of $F$ will obviously be 
$\bar{F} (x) = \sum_i \llbracket\bar{N}_i (x)+x\rrbracket$.
The neighborhoods of the skeleta are defined in this way:
\begin{enumerate}
\item if $p$ belongs to the $0$-skeleton, we let $L\in \sW$ be (one of) the smallest cubes 
containing it and define $U^p := B_{\ell(L)/16} (p)$;
\item if $\sigma= [p,q]\subset L$ is the edge of a cube and $L\in \sW$ is (one of) the smallest cube intersecting $\sigma$, we then define $U^\sigma$ to be the neighborhood of size $\frac{1}{4}\frac{\ell(L)}{16}$ of $\sigma$ minus the
closure of the unions of the $U^r$'s, where $r$ runs in the $0$-skeleton; 
\item we proceed inductively till the $m-1$-skeleton: given a $k$-dimensional facet $\sigma$ and (one of) the smallest cube $L\in \sW$ which intersects it,
$U^\sigma$ is its neighborhood of size $4^{-k} \frac{\ell (L)}{16}$ minus the closure of the union of all 
$U^\tau$'s, where $\tau$ runs among all facets of dimension at most $k-1$.
\end{enumerate} 
Denote by $\bar{U}$ the closure of the union of all
these neighborhoods and let $\{V_i\}$ be the connected components of the complement.
For each $V_i$ there is a $L_i\in \sW$ such that $V_i\subset L_i$.
Moreover, $V_i$ has distance $c_0 \ell(L)$ from $\partial L_i$, where $c_0$ is a geometric constant. 
It is also clear that if $\tau$ and $\sigma$ are two distinct facets of the same cube $L$ with the same
dimension, then the distance between any pair of points $x,y$ with $x\in U^\tau$ and $y\in U^\sigma$ is at least $c_0 \ell (L)$.
In Figure \ref{f:whitney} the various domains are shown in a piece of a $2$-dimensional decomposition.

\begin{figure}[htbp]
\begin{center}
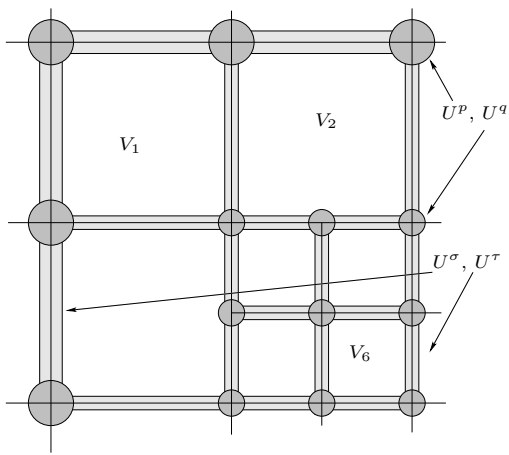
\caption{The sets $U^p$, $U^\sigma$ and $V_i$.}
\label{f:whitney}
\end{center}
\end{figure}

At a first step we extend $N$ to a new map $\bar{N}$ separately on each $U^p$, where $p$ are 
the points in the $0$-skeleton.
Fix $p\in L$ and let ${\rm St} (p)$ be the union of all cubes which contain $p$. Observe that the Lipschitz constant of $N|_{\cK\cap {\rm St} (p)}$ is smaller than $C \bmo^{\gamma_2} \ell (L)^{\gamma_2}$ and that $|N| \leq C \bmo^{\sfrac{1}{2m}}
\ell (L)^{1+\beta_2}$ on ${\rm St} (p)$. We can therefore extend the map $N$ to $U^p$ at the price
of slightly enlarging this Lipschitz constant and this height bound, using \cite[Theorem 1.7]{DS1}. 
Being the $U^p$ disjoint, the resulting map, for which we use the symbol $\bar{N}$, is well-defined.

It is obvious that this map has the desired height bound in each Whitney region. We therefore 
want to estimate its
Lipschitz constant.
Consider $L\in \sW$ and $H$ concentric to $L$ with side-length $\ell (H) = \frac{17}{16} \ell (L)$.
Let $x,y\in H$. If $x,y\in \cK$, then there is nothing to check. If $y\in U^p$ for some $p$ and $x\not\in \bigcup_q U^q$,
then $x\in {\rm St} (p)$ and $\cG (\bar{N}(x), \bar{N}(y)) \leq C 
\bmo^{\gamma_2} \ell (L)^{\gamma_2} |x-y|$. The same holds when $x,y\in U^p$. The remaining case is $x\in U^p$
and $y\in U^q$ with $p\neq q$. Observe however that this would imply that $p,q$ are both vertices of $L$. Given
that $L\setminus \cK$ has much smaller measure than $L$ there is at least one point $z\in L\cap \cK$. It
is then obvious that 
\[
\cG (\bar{N} (x), \bar{N} (y)) \leq \cG (\bar{N} (x), \bar{N} (z)) +
\cG (\bar{N} (z), \bar{N} (y)) \leq C \bmo^{\gamma_2} \ell (L)^{\gamma_2} \ell (L),
\]
and, since $|x-y|\geq c_0 \ell (L)$, the desired bound readily follows. 
Observe moreover that, if $x$ is in the closure of some $U^q$, then we can extend the map continuously to it.
By the properties of the Whitney decomposition it follows that the union of the closures of the $U^q$ and of $\mathcal{K}$ is closed and thus, w.l.o.g., we can assume that the domain of this new $\bar{N}$ is in fact closed.

This procedure can now be iterated over all skeleta inductively on the dimension $k$ of the corresponding
skeleton, up to
$k=m-1$: in the argument above we simply replace 
points $p$ with $k$-dimensional faces $\sigma$, defining ${\rm St} ( \sigma)$ as the union of the cubes which
contain $\sigma$. In the final step we then extend over the domains $V_i$'s:
this time ${\rm St} (V_i)$
will be defined as the union of the cubes which intersect the cube $L_i\supset V_i$. 
The correct height and Lipschitz bounds follow from  the same arguments. Since the algorithm is applied
$m+1$ times, the original constants have been enlarged by a geometric factor.

\medskip

{\bf Second extension: orthogonality.} For each $x\in \cM$ let $\p^\perp (x, \cdot) : \R^{m+n}\to \R^{m+n}$ 
be the orthogonal projection on $(T_x\cM)^\perp$ and set
$\hat{N} (x) = \sum_i \llbracket \p^\perp (x,\bar{N}_i (x))\rrbracket$. Obviously $|\hat{N} (x)| \leq |\bar{N} (x)|$,
so the $L^\infty$ bound is trivial. We now want
to show the estimate on the Lipschitz constant. 
To this aim, fix two points $p,q$ in the same Whitney region associated to $L$ and parameterize the corresponding
geodesic segment $\sigma\subset \cM$ by arc-length $\gamma :[0, d (p,q)] \to\sigma$, where $d(p,q)$ 
denotes the geodesic distance on $\cM$. Use \cite[Proposition~1.2]{DS1}
to select $Q$ Lipschitz functions $N'_i: \sigma \to \bU$ such that
$\bar{N}|_{\gamma} = \sum \a{N'_i}$ and 
$\Lip (N'_i) \leq \Lip (\bar{N})$. Fix a frame $\nu_1, \ldots, \nu_n$
on the normal bundle of $\cM$ with the property that $\|D\nu_i\|_{C^0}\leq C\bmo^{\sfrac{1}{2m}}$ (which is possible
since $\cM$ is the graph of a $C^{3,\kappa}$ function, cf.~\cite[Appendix A]{DS2}). We have
$\hat{N} (\gamma (t)) = \sum_i \llbracket\hat{N}_i (t)\rrbracket$, where
\[
\hat{N}_i (t) = \sum [\nu_j (\gamma (t)) \cdot N'_i (\gamma (t))] \, \nu_j (t).
\]
Hence we can estimate
\[
\left|\frac{d \hat{N}_i}{dt}\right| \leq C_0 \Lip (N'_i) + C_0 \sum_j \|D\nu_j\| \|N'_i\|_{C^0}
\leq C \bmo^{\gamma_2} \ell(L)^{\gamma_2} + C \bmo^{\sfrac{1}{2m}} \ell(L)^{1+\beta_2}
\leq C \bmo^{\gamma_2} \ell (L)^{\gamma_2}.
\]
Integrating this inequality we find
\[
\cG (\hat{N} (p), \hat{N} (q)) \leq C_0
\sum_{i=1}^Q |\hat{N}_i (d (p,q))-\hat{N}_i(0)|\leq C \bmo^{\gamma_2} \ell (L)^{\gamma_2} d (p,q) \, .
\]
Since $d(p,q)$ is comparable to $|p-q|$, we achieve the desired
Lipschitz bound.

\medskip

{\bf Third extension and conclusion.} For each $x\in \cM\subset \Sigma$ consider the orthogonal
complement $\varkappa_x$ of $T_x \cM$ in $T_x \Sigma$. Let $\cT$ be the fiber bundle 
$\bigcup_{x\in \cM} \varkappa_x$ and observe that, by the regularity of both $\cM$ and $\Sigma$
there is a global $C^{2,\kappa}$ trivialization (argue as in \cite[Appendix A]{DS2}).
It is then obvious that there is a $C^{2,\kappa}$ map $\Xi: \cT \to \R^{m+n}$
with the following
property: for each $(x,v)$, $q:= x+\Xi (x,v)$ is the only point in $\Sigma$ which is
orthogonal to $T_x \cM$ and such that $\p_{\varkappa_x} (q-x) = v$.
We then set $N (x) = \sum_i \llbracket\Xi (x, \p_{\varkappa_x} (\hat{N}_i (x)))\rrbracket$. 
Obviously, $N(x) = \hat{N} (x)$ for $x\in \cK$, simply because in this case $x+N_i (x)$ belongs
to $\Sigma$. 

In order to show the Lipschitz bound, denote by $\Omega (x, q)$ the map $\Xi (x, \p_{\varkappa_x} (q))$.
$\Omega$ is a $C^{2, \kappa}$ map.  Thus
\begin{equation}\label{e:Lip_1}
|\Omega (x, q) - \Omega (x, p)|\leq C_0 |q-p|\, .
\end{equation}
Moreover, since $\Omega (x, 0) = 0$ for every $x$, we have $D_x \Omega (x, 0) =0$. We therefore conclude that
$|D_x \Omega (x, q)|\leq C_0 |q|$ and hence that
\begin{equation}\label{e:Lip_2}
|\Omega (x, q) - \Omega (y, q)|\leq C_0 |q| |y-x|\, .
\end{equation}
Thus, fix two points $x,y\in \cL$ and let us assume that $\cG (\hat{N} (x), \hat{N} (y))^2 =
\sum_i |\hat{N}_i (x) - \hat{N}_i (y)|^2$ (which can be achieved by a simple relabeling).
We then conclude
\begin{align}
\cG (N (x), N(y))^2 &\leq 2 \sum_i |\Omega (x, \hat{N}_i (x)) - \Omega (x, \hat{N}_i (y))|^2 
+ 2 \sum_i |\Omega (x, \hat{N}_i (y)) - \Omega (y, \hat{N}_i (y))|^2\nonumber\\
&\leq C_0 \cG (\hat{N} (x), \hat{N} (y))^2 + C \sum_i |\hat{N}_i (y)|^2 |x-y|^2\nonumber\\
&\leq C \bmo^{2\gamma_2} \ell (L)^{2\gamma_2} |x-y|^2 + C \bmo^{\sfrac{1}{2m}} \ell (L)^{1+\beta_2} |x-y|^2\, .
\end{align}
This proves the desired Lipschitz bound.
Finally, using the fact that
$\Omega (x,0) =0$, we have $|\Omega (x,v)|\leq C_0 |v|$ and the $L^\infty$ bound readily follows. 

\subsection{Estimates \eqref{e:Dir_regional} and \eqref{e:av_region}.} First consider
the cylinder $\bC:= \bC_{8r_L} (p_L, \pi_L)$. Denote by $\vec{\cM}$ the unit $m$-vector 
orienting $T\cM$ and by $\vec\tau$ the one orienting $T \bG_{h_L} = T\bG_{g_L}$. 
Recalling that $g_L$ and $\phii$ coincide in a neighborhood of $x_L$, we have
\[
\sup_{p\in \cM\cap \bC} |\vec{\tau} (x_L, g_L (x_L)) -\vec{\cM} (p)|\leq C \|D^2 \phii\|_{C^0} \,\ell (L)
\leq C \bmo^{\sfrac{1}{2}}\ell (L).
\]
Since $\|D^2 h_L\|\leq C \bmo^{\sfrac{1}{2}}$ we have $|\vec{\tau} (x_L, g_L (x_L)) -
\vec\tau (q)|\leq C \bmo^{\sfrac{1}{2}} \ell (L)$ $\forall q\in \cM\cap \bC$. Combining the last two inequalities with
Proposition~\ref{p:stime_chiave}(iv) we infer $\sup_{\bC\cap \cM} |\vec{\cM}-\pi_L| \leq C \bmo^{\sfrac{1}{2}} \ell (L)^{1-\delta_2}$.
Thus, since $\p^{-1} (\cL) \cap \supp (T) \subset \bC$, we can estimate
\begin{align}
\int_{\p^{-1} (\cL)} |\vec{\bT}_F (x) -& \vec{\cM} (\p (x))|^2 d\|\bT_F\| (x)\nonumber\\ 
\leq\; &\int_{\p^{-1} (\cL)} |\vec{T} (x) - \vec{\cM} (\p (x))|^2 d\|T\| (x) + C \bmo^{1+\gamma_2}
\ell (L)^{m+2+\gamma_2}\nonumber\\
\leq\; &\int_{\p^{-1} (\cL)} |\vec{T} (x) - \vec{\pi}_L|^2 d \|T\| (x) + C \bmo \ell (L)^{m+2-2\delta_2}\label{e:eccesso_storto}
\end{align}
In turn the integral in \eqref{e:eccesso_storto} 
is smaller than $C \ell (L)^m \bE (T, \bC, \pi_L)$.
By \cite[Proposition 3.4]{DS2} we then conclude
\begin{align*}
\int_{\cL} |DN|^2 \leq &C_0 \int_{\p^{-1} (\cL)} |\vec{\bT}_F (x) - \vec{\cM} (\p (x))|^2 d \|\bT_F\| (x) + 
C_0 \|A_{\cM}\|_{C^0}^2 \int_{\cL} |N|^2\\
& + C_0 \Lip (N)^2 \int_{\mathcal{L}} |DN|^2 \\
&\leq C \bmo \,\ell(L)^{m+2-2\delta_2} + C \bmo\, \ell (L)^{m+2+2\beta_2} + C \bmo^{2\gamma_2} \int_{\mathcal{L}} |DN|^2\, ,
\end{align*}
where we have used $\|A_\cM\|_{C^0}  \leq C\, \|D^2 \phii\|_{C^0}\leq C \bmo^{\sfrac{1}{2}}$. Thus \eqref{e:Dir_regional} follows
provided $\eps_2$ is sufficiently small

We finally come to \eqref{e:av_region}. First observe that, by \eqref{e:Lip_regional} and \eqref{e:err_regional},
\begin{equation}\label{e:sfava1}
\int_{\cL\setminus \cK} |\etaa \,\circ N| 
\leq C \, \bmo^{\sfrac{1}{2m}} \ell(L)^{1+\beta_2} |\cL\setminus \cK|
\leq C \bmo^{1+\gamma_2 + \sfrac{1}{2m}} \ell(L)^{m+3+\beta_2+\gamma_2}\, .
\end{equation}
Fix now $p\in \cK$. 
Recalling that $F_L (x) = \sum_i \a{p+ (N_L)_i (p)}$ is given
by \cite[Theorem 5.1]{DS2} applied to the map $f_L$, we can use \cite[Theorem 5.1(5.4)]{DS2}
to conclude
\begin{align}
|\etaa \circ N_L (p)| \leq {}& C\, |\etaa \circ f_L (\p_{\pi_L}(p)) - \p_{\pi_L}^\perp (p)| + 
C \, \Lip (N_L|_{\cL}) \, |T_p \cM - \pi_L| \, |N_L| (p)\nonumber\\
\leq {}& C |\etaa \circ f_L (\p_{\pi_L}(p)) - \p_{\pi_L}^\perp (p)| \notag\\
&+ C \bmo^{\sfrac{1}{2} + \gamma_2} \ell (L)^{1+\gamma_2 -\delta_2}
\left(\cG (N_L(p), Q \a{\etaa \circ N_L (p)}) + Q|\etaa \circ N_L| (p)\right)\, .\nonumber
\end{align}
For $\eps_2$ sufficiently small (depending only on $\beta_2, \gamma_2, M_0, N_0, C_e, C_h$), we then conclude that
\begin{align}
|\etaa \circ N_L (p)| \leq {}& C\, |\etaa \circ f_L (\p_{\pi_L}(p)) - \p_{\pi_L^\perp} (p)| + C  \bmo^{\sfrac{1}{2} + \gamma_2} \ell (L)^{1+\gamma_2 -\delta_2}
\cG (N_L(p), Q \a{\etaa \circ N_L (p)}) \nonumber\\
\leq {}& C \, |\etaa \circ f_L (\p_{\pi_L} (p)) -\p_{\pi_L}^\perp (p)|+ C\,a\, \bmo^{1+\gamma_2} \, 
\ell(L)^{(1+\gamma_2-\delta_2) \frac{2+\gamma_2}{1+\gamma_2}}\notag\\
&+ \frac{C}{a}\, \cG (N_L(p), Q\a{\etaa \circ N_L (p)})^{2+\gamma_2}\, .\label{e:da_integrare_10}
\end{align}
Our choice of $\delta_2$ makes the exponent $(1+\gamma_2-\delta_2) \frac{2+\gamma_2}{1+\gamma_2}$ larger than $2+\gamma_2/2$.
Let next $\ph': \pi_L\to \pi_L^\perp$ be such that $\bG_{\ph'} = \cM$. 
Applying Lemma \ref{l:rotazioni_semplici} we conclude that
\begin{align*}
\int_{\cK \cap \cV} |\etaa\circ f_L (\p_{\pi_L}(p)) - \p_{\pi_L^\perp} (p))| \leq &
C \int_{\p_{\pi_L} (\cK \cap \cV)} |\etaa \circ f_L (x) - \ph' (x)|\\
\leq & C\|g_L (x) - \phii (x)\|_{C^0 (H)} \ell (L)^m\, ,
\end{align*}
where $H$ is a cube concentric to $L$ with side-length $\ell (H) = \frac{9}{8} \ell (L)$.
From Proposition \ref{p:stime_chiave}(v) we get 
$\|\phii- g_L\|_{C^0 (H)} \leq C \bmo \ell (L)^{m+3+\sfrac{\beta_2}{3}}$ and
\eqref{e:av_region} follows integrating \eqref{e:da_integrare_10} over $\cV\cap \cK$ and using \eqref{e:sfava1}.

\subsection{Proof of Corollary \ref{c:globali}}
Observe that $N \equiv 0$ over $\Phii (\bGam)$ and thus the second inequality in
\eqref{e:global_Lip} follows easily from the second inequality of \eqref{e:Lip_regional}, recalling that
$\ell (L) \leq 1$ for any cube $L\in \sW$. For the same reasons, from \eqref{e:Dir_regional} we conclude
\[
\int_{\cM'} |DN|^2 \leq C \bmo \sum_{L\in \sW} \ell (L)^{m+2-2\delta_2} \leq C
\bmo \sum_{L\in\sW} \ell (L)^m \leq C \bmo\, .
\]
\eqref{e:global_masserr} follows from \eqref{e:err_regional} with similar considerations.
Coming to the first inequality in \eqref{e:global_Lip} fix any two points $p = \Phii (x),q = \Phii (y)\in \cM'$. Observe that the length of the geodesic segment joining $p$ and $q$ is comparable, up to constants, to $|x-y|$. If $x,y\in \bGam$, then $N(p)=N(q)=Q\a{0}$ and so $\cG( N(p), N(q)) = 0$. If $x\in \bGam$ and $y\not \in \bGam$, then $y$
belongs to some $L\in \sW$ and, by the properties of the Whitney decomposition, $\ell (L) \leq \frac{1}{2} |x-y|$. Thus, using the second inequality in \eqref{e:Lip_regional} we conclude 
$\cG (N(q), N(p)) = \cG (N (q), Q\a{0}) \leq \|N|_{\mathcal{L}}\|_{C^0}\leq C \bmo^{\sfrac{1}{2m}} \ell (L)^{1+\beta_2} \leq C \bmo^{\sfrac{1}{2m}} |x-y|$. Finally, if $x, y\not \in \bGam$ we analyze two cases. If the geodesic segment
$[x,y]$ intersects $\bGam$, then we conclude the same inequality as above. Otherwise there are points $x = z_0, z_1, \ldots, z_N = y$ in $[x,y]$ such that each segment $[z_{i-1}, z_i]$ is contained in some single $L_i\in \sW$ and $\sum_i |z_i-z_{i-1}| = |x-y|$. It then follows from the first bound in \eqref{e:Lip_regional} that
\[
\cG (N(p), N(q)) \leq \sum_i \cG (N (\Phii (z_i), N (\Phii (z_{i-1}))) \leq C \bmo^{\gamma_2} \sum_i |z_i - z_{i-1}|
= C \bmo^{\gamma_2} |x-y|\, .
\]
Recalling that $\gamma_2 \leq \frac{1}{2m}$, all the cases examined prove the first inequality in \eqref{e:global_Lip}.

\section{Separation and splitting before tilting}

As in the previous sections, $C_0$ will be used for geometric constants, $\bar{C}$ for constants depending on $\beta_2, \delta_2, M_0, N_0$ and $C_e$, whereas
$C$ will be used for constants depending on all the latter parameters and also $C_h$. 

\subsection{Vertical separation} 
In this section we prove Proposition \ref{p:separ} and Corollary \ref{c:domains}.

\begin{proof}[Proof of Proposition \ref{p:separ}]
Let $J$ be the father of $L$. 
By Proposition~\ref{p:tilting opt}, Theorem \ref{t:height_bound} can
be applied to the cylinder $\bC := \bC_{36r_J} (p_J, \pi_J)$. 
Moreover, 
$|p_J-p_L|\leq 3\sqrt{m} \ell (J)$. 
Thus, if $M_0$ is
larger than a geometric constant, 
we have
$\bB_L \subset \bC_{34 r_J} (p_J, \pi_J)$. 
Denote by $\q_L$, $\q_J$ the projections $\p_{\hat\pi_L^\perp}$
and $\p_{\pi_J^\perp}$ respectively. Since $L \in \sW_h$, 
there are two points $p_1, p_2\in
\supp (T) \cap \bB_L$ such that $|\q_L (p_1-p_2)| \geq C_h \bmo^{\sfrac{1}{2m}} \ell (L)^{1+\beta_2}$. 
On the other hand, recalling Proposition \ref{p:tilting opt}, $|\pi_J- \hat{\pi}_L|\leq \bar{C}\bmo^{\sfrac{1}{2}}
\ell (L)^{1-\delta_2}$, where $\bar{C}$ depends upon $\beta_2, \delta_2, M_0, N_0, C_e$ but not $C_h$. 
Thus,
\begin{equation*}
|\q_J (p_1-p_2)| \geq |\q_L (p_1-p_2)| - C_0 |\hat{\pi}_L-\pi_J| |p_1-p_2|
\geq C_h \bmo^{\sfrac{1}{2m}} \ell (L)^{1+\beta_2} - \bar{C} \bmo^{\sfrac{1}{2}}\ell (L)^{2-\delta_2}\, .
\end{equation*}
Hence, if $\eps_2$ is sufficiently small, we actually conclude
\begin{equation}
|\q_J (p_1-p_2)| \geq \frac{15}{16} C_h \bmo^{\sfrac{1}{2m}} \ell (L)^{1+\beta_2}\, .
\end{equation}
Set $E:= \bE (T, \bC)$ and apply Theorem \ref{t:height_bound} to $\bC$: the union of the corresponding
``stripes'' $\bS_i$  contains the set $\supp (T) \cap \bC _{36 r_J (1 - C_0 E^{\sfrac{1}{2m}} |\log E|)} (p_J, \pi_J))$. We can therefore assume that they contain
$\supp (T) \cap \bC_{34r_J} (p_J, \pi_J)$. The width of these stripes is bounded as follows:
\[
\sup \big\{|\q_J (x-y)| :x,y\in \bS_i\big\}
\leq C_0 \,E^{\sfrac{1}{2m}} r_J \leq C_0\, C_e^{\sfrac{1}{2m}} M_0 \bmo^{\sfrac{1}{2m}} \ell (L)^{1+ (2-2\delta_2)/2m}\, .
\]
So, if $C^\sharp$ is chosen large enough (depending only upon $M_0$ $m$, $n$ and $Q$), we actually conclude that $p_1$
and $p_2$ must belong to two different stripes, say $\bS_1$ and $\bS_2$. 
By Theorem \ref{t:height_bound}(iii) we conclude that all points in $\bC_{34 r_J} (p_J, \pi_J)$ have density $\Theta$ strictly
smaller than $Q-\frac{1}{2}$, thereby implying (S1).
Moreover, by choosing $C^\sharp$ appropriately, we achieve that
\begin{equation}\label{e:vert_separ}
|\q_J (x-y)| \geq  \frac{7}{8} C_h \bmo^{\sfrac{1}{2m}} \ell (L)^{1+\beta_2} \, .
\quad \forall x\in \bS_1, y\in \bS_2\, .
\end{equation}
Assume next there is $H\in \sW$ with $\ell (H) \leq \frac{1}{2} \ell (L)$ and
 $H\cap L\neq\emptyset$. 
From our construction it follows that $\ell (H) = \frac{1}{2} \ell (L)$, $\bB_H\subset
\bC_{34 r_J} (p_J, \pi_J)$ and $|\pi_H -\pi_J|\leq \bar{C} \bmo^{\sfrac{1}{2}} \ell (H)^{1-\delta_2}$
(see again Proposition \ref{p:tilting opt}). Arguing as above (and possibly choosing $\eps_2$ smaller, but
only depending upon $\beta_2, \delta_2, M_0, N_0, C_e$ and $C_h$)
we then conclude
\begin{equation}\label{e:vert_separ_10}
|\p_{\pi_H^\perp} (x-y)| \geq  \frac{3}{4} C_h \bmo^{\sfrac{1}{2m}} \ell (L)^{1+\beta_2} 
\geq \frac{3}{2} C_h \bmo^{\sfrac{1}{2m}} \ell (H)^{1+\beta_2} \qquad \forall x\in \bS_1, y\in \bS_2\, .
\end{equation}
Now, recalling Proposition~\ref{p:tilting opt}, if $\eps_2$ is sufficiently small, $\bC_{32 r_H} (p_H, \pi_H) \cap
\supp (T) \subset \bB_H$. Moreover, by Theorem \ref{t:height_bound}(ii) ,
\[
(\p_{\pi_J})_\sharp (T\res (\bS_i\cap \bC_{32 r_H} (p_H, \pi_J))) = Q_i \a{B_{32 r_H} (p_H, \pi_J)}
\quad \text{for } \, i=1,2,\;\; Q_i\geq 1.
\] 
A simple argument already used several other times allows to conclude that indeed
 \[
(\p_{\pi_H})_\sharp (T\res (\bS_i\cap \bC_{32 r_H} (p_H, \pi_H))) = Q_i \a{B_{32 r_H} (p_H, \pi_H)}
\quad \text{for } \, i=1,2,\;\; Q_i\geq 1.
\] 
Thus, $\bB_H$ must necessarily contain two points $x,y$ with $|\p_{\pi_H^\perp} (x-y)|\geq \frac{3}{2} 
C_h \bmo^{\sfrac{1}{2m}} \ell (H)^{1+\beta_2}$. Given that $|\hat{\pi}_H-\pi_H| \leq \bar{C}
\bmo^{\sfrac{1}{2}} \ell (H)^{1-\delta_2}$,
we conclude (again imposing that $\eps_2$ is sufficiently small) that $|\p_{\pi^\perp_{\hat{H}}} (x-y)| \geq \frac{5}{4}
C_h \bmo^{\sfrac{1}{2m}} \ell (H)^{1+\beta_2}$, i.e.~the cube $H$
satisfies the stopping condition (HT), which has ``priority over the condition (NN)'' and thus it cannot belong to
$\sW_n$. This shows (S2).

{ Coming to (S3), set $\Omega:= \Phii (B_{2\sqrt{m} \ell (L)} (x_L, \pi_0)$ and observe that $\p_\sharp (T\res (\Omega\cap \bS_i)) = Q_i \a{\Omega}$. Thus, for each $p\in \cK\cap \Omega$,}
the support of $p + N(p)$ must contain at least
one point $p+N_1 (p) \in \bS_1$ and at least one point $p+N_2 (p)\in \bS_2$. 
Now, by \eqref{e:vert_separ}
\begin{equation}\label{e:separazione parametrica}
|N_1 (p)- N_2 (p)| \geq  \frac{7}{8} C_h \bmo^{\sfrac{1}{2m}} \ell (L)^{1+\beta_2} - C_0 \ell (L) \,
|T_p \cM - \pi_J|\, .
\end{equation}
Recalling, however, Proposition~\ref{p:stime_chiave} and
that $\cM$ and $\gr (g_J)$ coincide on a nonempty open set, we easily conclude that
$|T_p \cM - \pi_J|\leq C \bmo^{\sfrac{1}{2}} \ell (L)^{1-\delta_2}$
and, via \eqref{e:separazione parametrica},
\[
\cG \big(N (p), Q \a{\etaa \circ N (p)}\big) \geq \frac{1}{2} |N_1 (p)- N_2 (p)|\geq \frac{3}{8} C_h \bmo^{\sfrac{1}{2m}} \ell (L)^{1+\beta_2}\, .
\]
{ Next observe that, by the property of the Whitney decomposition, any cube touching $B_{2\sqrt{m} \ell (L)} (x_L, \pi)$ has sidelength at most $4 \ell (L)$. Thus
$|\Omega\setminus \cK| \leq C \bmo^{1+\gamma_2} \ell (L)^{m+2+\gamma_2}$} and for every point 
$p\in \bOmega$ there exists $q\in \cK\cap \Omega$ which has geodesic distance to $p$ at most 
$C \bmo^{\sfrac{1}{m}+\sfrac{\gamma_2}{m}} \ell(L)^{1+ \sfrac{2}{m}+\sfrac{\gamma_2}{m}}$. Given the Lipschitz bound for $N$ and the choice $\beta_2 \leq \frac{1}{2m}$,
we then easily conclude (S3):
\[
\cG (N (q), Q \a{\etaa \circ N (q)}) \geq \frac{3}{8} C_h \bmo^{\sfrac{1}{2m}} \ell (L)^{1+\beta_2}
- C \bmo^{\sfrac{1}{m}} \ell (L)^{1+\sfrac{2}{m}} \geq \frac{1}{4} \,C_h \bmo^{\sfrac{1}{2m}} \ell (L)^{1+\beta_2}\, ,
\]
where again we need $\eps_2 < c (\beta_2, \delta_2, M_0, N_0, C_e, C_h)$ for a sufficiently small $c$. 
\end{proof}

\begin{proof}[Proof of Corollary \ref{c:domains}]
The proof is straightforward. Consider any $H\in \sW^j_n$. By definition it has
a nonempty intersection with some cube $J\in\sW^{j-1}$: this cube cannot belong to $\sW_h$ by
Proposition \ref{p:separ}. It is then either an element of $\sW_e$ or an element
$H_{j-1}\in \sW^{j-1}_n$. Proceeding inductively, we then find a chain $H = H_j, H_{j-1}, \ldots,
H_i =: L$, where $H_{\bar{l}}\cap H_{\bar{l}-1} \neq\emptyset$ for every $\bar{l}$, $H_{\bar l} \in \sW^{\bar{l}}_n$ for every $\bar{l}>i$
and $L= H_i\in \sW^i_e$. 
Observe also that
\[
|x_H-x_L| \leq \sum_{\bar{l}=i}^{j-1} |x_{H_{\bar{l}}} - x_{H_{\bar{l}+1}}| \leq \sqrt{m}\, \ell (L) 
\sum_{\bar{l}=0}^\infty 2^{-\bar{l}}
\leq 2\sqrt{m} \,\ell (L)\, . 
\]
It then follows easily that $H\subset B_{3\sqrt{m} \ell (L)} (L)$.
\end{proof}

\subsection{Unique continuation for Dir-minimizers}
Proposition \ref{p:splitting} is based on a De Giorgi-type decay estimate for $\D$-minimizing $Q$-valued maps which are 
close to a classical harmonic function with multiplicity $Q$.
The argument involves a unique continuation-type result for Dir-minimizers. 

\begin{lemma}[Unique continuation for $\D$-minimizers]\label{l:UC}
For every $\eta \in (0,1)$ and $c>0$, there exists $\gamma>0$ with the following property.
If $w: \R^m\supset B_{2\,r} \to \Iq(\R^n)$ is Dir-minimizing,
$\D(w, B_r)\geq c$ and  $\D(w, B_{2r}) =1$,
then
\[
\D (w, B_s (q)) \geq \gamma \quad \text{for every 
$B_s(q)\subset B_{2r}$ with $s \geq \eta\,r$}.
\]
\end{lemma}

\begin{proof} 
We start showing the following claim:
\begin{itemize}
\item[(UC)] if $\Omega$ is a connected open set and $w\in W^{1,2} (\Omega, \Iqs)$ 
is Dir-minimizing in any open
$\Omega'\subset\subset \Omega$, then either $w$ is 
constant or $\int_J |Dw|^2 >0$ on any open $J\subset \Omega$. 
\end{itemize}
We prove (UC) by induction on $Q$.
If $Q=1$, this is the classical unique continuation for harmonic functions.
Assume now it holds for all
$Q^*<Q$ and we prove it for $Q$-valued maps. 
Assume $w\in W^{1,2} (\Omega, \Iqs)$ and $J\subset \Omega$ is an open set
on which $|Dw|\equiv 0$.
Without loss of generality, we can assume $J$ connected and 
$w|_J \equiv T$ for some $T\in \Iq$. 
Let $J'$ be the interior of $\{w=T\}$ and $K:= \overline{J'}\cap \Omega$.
We prove now that $K$ is open, which in turn by connectedness of $\Omega$
concludes (UC).
We distinguish two cases.

{\bf Case (a): the diameter of $T$ is positive.}
Since $w$ is continuous, for every $x\in K$ there is 
$B_\rho (x)$ where $w$ separates into $\a{w_1}+\a{w_2}$ 
and each $w_i$ is a $Q_i$-valued Dir-minimizer. Since 
$J'\cap B_\rho (x)\neq\emptyset$, each $w_i$ is constant in a (nontrivial)
open subset of $B_\rho (x)$. 
By inductive hypothesis each $w_i$ is constant in $B_\rho (x)$ and
therefore $w=T$ in $B_\rho (x)$, that is 
$B_\rho (x)\subset J'\subset K$. 

{\bf Case (b): $T= Q\a{p}$ for some $p$.} In this case let $J''$ be the interior of 
$\{w=Q\a{\etaa\circ w}\}$.
By \cite[Definition 0.10]{DS1}, 
$\partial J'' \cap \Omega$ is contained in the singular set of $w$. 
By \cite[Theorem 0.11]{DS1}, $\cH^{m-2+\varepsilon} (\Omega\cap \partial J'') =0$ for every 
$\varepsilon >0$. Consider now a point $p\in \partial J''\cap \Omega$ and a small ball $B_\rho (x)\subset \Omega$.
Since $\cH^{m-1} (\partial J'' \cap B_\rho (x)) =0$, by the isoperimetric inequality, either $|B_\rho (x)\setminus J''|=0$
or $|J''|=0$. The latter alternative is impossible because $J''$ is open and has nonempty intersection with
$B_\rho (x)$. It then turns out that $|B_\rho (x)\setminus J''|=0$ and thus the closure of $J''$ contains $B_\rho (x)$.
But then $w = Q\a{\etaa\circ w}$ on $B_\rho (x)$ and thus $x$ cannot belong to $\partial J''$. So $\partial J''\cap \Omega$
is empty and thus $w= Q \a{\etaa \circ w}$ on $\Omega$. On the other hand
$\etaa\circ w$ is an harmonic function (cf.~\cite[Lemma 3.23]{DS1}).
Being $\etaa \circ w|_{J'} \equiv p$,
by the classical unique continuation $\etaa\circ w\equiv p$ on $\Omega$.

\medskip

We now come to the proof of the lemma. Without loss of generality, we can assume
$r=1$. Arguing by contradiction, 
there exists sequences $\{w_k\}_{k\in\N}\subset W^{1,2} (B_2, \Iqs)$ and $\{B_{s_k}(q_k)\}_{k\in\N}$ 
with $s_k \geq \eta$ and such that $\D (w_k, B_{s_k} (q_k)) \leq \frac{1}{k}$. { Without loss of generality,
after applying a translation, we can assume that $\etaa\circ w_k (0) =0$. Next, passing to a subsequence,
we can either assume that $\sup_k \cG (w_k (0), Q\a{0})<\infty$ or that $\lim_k \cG (w_k, Q \a{0})=\infty$.} In the first case, by \cite[Proposition~3.20]{DS1}, a subsequence 
(not relabeled) converges to $w\in W^{1,2} (B_{2}, \Iqs)$ 
Dir-minimizing in every open $\Omega'\subset\subset B_{2}$.
Up to subsequences, we can also assume that 
$q_k\to q$ and $s_k\to s \geq \eta >0$. 
Thus, $B_s(q) \subset B_{2}$ and $\D (w, B_s (q)) = 0$. By (UC) this implies that 
$w$ is constant. On the other hand, by \cite[Proposition~3.20]{DS1} $\D (w, B_1) =
\lim_k \D (w_k, B_1) \geq c>0$ gives the desired contradiction. { In the second case, by the H\"older continuity of
$\D$-minimizers, each $w_k$ splits in $B_{3/2}$ as $w_k = w_k^1+w_k^2$ where $w_k^i$ is $\D$-minimizing and $Q_i$-valued.
After extracting a subsequence we can assume that $Q_1$ is independent of $k$ and that $\D (w_k^1, B_1)\geq \frac{c}{2}$. We can then repeat the argument above and either reach a contradiction or split further the sequence in the ball $B_{5/4}$. The splitting procedure must stop after at most $Q$ iterations.}
\end{proof}

Next we show that if the energy of a Dir-minimizer $w$ does not decay appropriately, 
then $w$ must split. In order to simplify the exposition, in the sequel we fix $\lambda>0$ such that 
\begin{equation}\label{e:lambda}
(1+\lambda)^{(m+2)}< 2^{\delta_2}\, .
\end{equation}

\begin{proposition}[Decay estimate for $\D$-minimizers]\label{p:harmonic split}
For every $\eta>0$, there is $\gamma >0 $ with the following property.
Let $w: \R^m \supset B_{2r} \to \Iq(\R^n)$ be Dir-minimizing 
in every $\Omega'\subset\subset B_{2r}$ such that
\begin{equation}\label{e.no decay}
\int_{B_{(1+\lambda) r}} \cG \big(Dw, Q \a{D (\etaa \circ w) (0)}\big)^2 \geq 2^{\delta_2-m-2} \D (w, B_{2r})\, .
\end{equation}
Then, if we set $\bar{w} = \sum_i \a{w_i - \etaa\circ w}$, the following holds:
\begin{equation}\label{e.harm split 2}
\gamma \, \D (w, B_{(1+\lambda) r}) \leq \D (\bar{w}, B_{(1+\lambda) r}) 
\leq \frac{1}{\gamma \, r^2} \int_{B_{s}(q)} |\bar{w}|^2 \quad\forall \;B_s(q) \subset B_{2\,r} \;\text{with }\, s\geq \eta \,r\, .
\end{equation}
\end{proposition}

Before coming to the proof of the Proposition we point out an elementary fact which will be used repeatedly in this section. 

{ \begin{lemma}\label{l:medie}
Let $B\subset \R^m$ be a ball centered at $0$, $w\in W^{1,2} (B , \Iqs)$ $\D$-minimizing and $\bar{w} = \sum_i \a{w_i - \etaa\circ w}$.
We then have
\begin{align}
Q \int_B |D (\etaa\circ w) - D (\etaa \circ w) (0)|^2
&=\int_B \cG (Dw, Q \a{ D (\etaa \circ w) (0)})^2 - \D (\bar w, B)\, .
\end{align}
\end{lemma}
\begin{proof}
Let $u := \etaa \circ w$ and observe that it is harmonic. Thus, using the mean value property of harmonic functions and a straightforward computation we get
\[
Q \int_B |Du - Du (0)|^2 = Q \int_B |Du|^2 - Q |B| |Du (0)|^2\, .
\]
On the other hand, using again the mean value property of harmonic functions, it is easy to see that
\[
\int_B \cG (Dw, Q \a{Du (0)})^2  + Q |B| |Du (0)|^2 = \int_B |Dw|^2 = \int_B |D \bar w|^2 + \int_B |Du|^2\, .
\]
Combining the last two inequalities we prove the lemma.
\end{proof}}

\begin{proof}[Proof of Proposition \ref{p:harmonic split}]
By a simple scaling argument we can assume $r=1$ and we argue by contradiction. 
Let $w_k$ be a sequence of local $\D$-minimizers 
which satisfy \eqref{e.no decay}, $\D (w_k, B_2) =1$ and 
\begin{itemize}
\item[(a)] either $\int_{B_{s_k} (q_k)} |\bar{w}_k|^2 \leq \frac{1}{k}$ for some ball $B_{s_k} (q_k) \subset B_{2r}$ with $s_k \geq \eta$;
\item[(b)] or $\D  (\bar{w}_k, B_{1+\lambda}) \leq \frac{1}{k}$.
\end{itemize}
{ Passing to a subsequence, if necessary, we can assume that $s_k \to  s$ and $q_k \to q$. Moreover, we can normalize the sequence so that $\mint_{B_2} D (\etaa \circ w_k) =0$ and in particular, passing to a subsequence, assume that $\etaa\circ w_k$ converges strongly in $L^2$. Assume now that (a) holds for an infinite sequence of indices. In that case we can extract a subsequence, not relabeled, which converges locally in $W^{1,2}$ to a $\D$-minimizer $w$: in fact the H\"older bound for $\D$-minimizers and (a) imply necessarily that $\sup_k \cG (w_k (q_k), Q \a{\etaa \circ w_k (q_k)}) < \infty$ and we can argue as in the proof of Lemma \ref{l:UC}}. We then conclude that $\bar{w} = \sum_i \a{w_i-\etaa\circ w}$ vanishes identically on $B_s (q)$ and we can appeal to Lemma \ref{l:UC} to infer that $\bar{w}$ vanishes on $B_2$.
This means in particular that 
$\D (\bar{w}_k, B_{1+\lambda}) \to \D (\bar{w}, B_{1+\lambda}) = 0$. Summarizing we conclude that $\D (\bar{w}_k, B_{1+\lambda})$ converges to $0$ in any case.

Let next $u_k:= \etaa\circ w_k$ and recall that we are assuming that $u_k$ converges to an harmonic function $u$.
Thus from \eqref{e.no decay} and Lemma \ref{l:medie} we get
\begin{align}
\int_{B_{1+\lambda}} Q |Du_k - Du_k (0)|^2 &=  \int_{B_{1+\lambda}} \left(\cG (D w_k, Q \a{Du_k (0)})^2 -
|D\bar{w}_k|^2\right)\nonumber\\
&\geq 2^{\delta_2-m-2} \int_{B_2} |Dw_k|^2 -
\int_{B_{1+\lambda}} |D\bar{w}_k|^2
\, .\label{e:viola}
\end{align}
Letting $k\uparrow \infty$, since $\D (w_k, B_2)\leq 1$ and $\D (\bar w_k, B_{1+\lambda}) \to 0$, 
we conclude
\begin{equation}\label{e:viola2}
\int_{B_{1+\lambda}} |Du - Du (0)|^2 \geq 2^{\delta_2-m-2}\geq 2^{\delta_2-m-2} \int_{B_2} |Du|^2\, .
\end{equation} 
Since $(1+\lambda)^{m+2} < 2^{\delta_2}$, \eqref{e:viola2}
violates the decay estimate for classical harmonic functions: 
\begin{equation}\label{e:decadimento_armonico}
\int_{B_{1+\lambda}} |Du - Du (0)|^2 \leq 2^{-m-2} (1+\lambda)^{m+2} \int_{B_2} |Du|^2\, ,
\end{equation}
thus concluding the proof. In order to show \eqref{e:decadimento_armonico} it suffices to decompose
$Du$ in series of homogeneous harmonic polynomials $Du (x) = \sum_{i =0}^\infty P_i (x)$, where $i$ is the degree. 
In particular the restriction of this decomposition
on any sphere $S := \partial B_\rho$
gives the decomposition of $Du|_S$ in spherical harmonics, see \cite[Chapter 5, Section 2]{SW}. It turns out, therefore, that the $P_i$ are $L^2 (B_\rho)$ orthogonal.  Since the constant polynomial $P_0$ is $Du (0)$ and $\int_{B_{1+\lambda}} |P_i|^2 \leq 2^{-m-2i}\int_{B_2} |P_i|^2$, \eqref{e:decadimento_armonico} follows at once.
\end{proof}


\subsection{Splitting before tilting I: Proof of Proposition \ref{p:splitting}}\label{ss:splitting}
As customary we use the convention that constants denoted by $C$ depend upon all the parameters but $\eps_2$, whereas constants denoted by $C_0$ depend only upon $m,n,\bar{n}$ and $Q$.

Given $L\in \sW^j_e$, let us consider
its ancestors $H\in \sS^{j-1}$ and $J\in \sS^{j-6}$. 
Set $\ell = \ell(L)$,$\pi = \hat{\pi}_H$ and $\bC := \bC_{8r_J} (p_J, \pi)$,  and
let $f:B_{8r_J} (p_J, \pi)\to \Iq (\pi^\perp)$ be the $\pi$-approximation of Definition \ref{d:pi-approximations},
which is the result of \cite[Theorem 1.4]{DS3} applied to $\bC_{32r_J} (p_J, \pi)$
(recall that Proposition~\ref{p:gira_e_rigira}(i) ensures the applicability of \cite[Theorem 1.4]{DS3} in the latter cylinder).
We let $K\subset B_{8r_J} (p_J, \pi)$ denote the set of \cite[Theorem 1.4]{DS3} and recall that
$\bG_{f\vert_K} = T\res K\times \pi^\perp$. Observe that $\bB_L\subset \bB_H \subset \bC$  (this requires, as usual, $\eps_2 \leq c( \beta_2, \delta_2, M_0, N_0, C_e, C_h)$).
The following are simple consequences of Proposition~\ref{p:tilting opt}:
\begin{gather}
E := \bE (T, \bC_{32 r_J} (p_J, \pi)) \leq
C \bmo \,\ell^{2-2\delta_2}\, ,\label{e:eccesso J}\\
\bh (T, \bC, \pi) 
\leq C\,\bmo^{\sfrac{1}{2m}} \ell^{1+\beta_2} .\label{e:altezza J}
\end{gather}
In particular the positive constant $C$ does not depend on $\eps_2$.
Moreover, since $\bB_L \subset \bC$, $L\in \sW_e$ and $r_L / r_J = 2^{-6}$, we have
\begin{equation}\label{e:eccesso alto}
c C_e \,\bmo\,r_L^{2-2\delta} \leq E\, ,
\end{equation}
where $c$ is only a geometric constant.
We divide the proof of Proposition~\ref{p:splitting} in three steps.

\medskip

{\bf Step 1: decay estimate for $f$.}
Let $2\rho:= 64 r_H - C^\sharp \bmo^{\sfrac{1}{2m}} \ell^{1+\beta_2}$:
since $p_H \in \supp (T)$, it follows from \eqref{e:altezza J} that, upon chosing $C^\sharp$ appropriately,
$\supp (T) \cap \bC_{2\rho} (p_H, \pi)\subset \B_H\subset \bC$ (observe that $C^\sharp$ depends
upon the parameters $\beta_2, \delta_2, M_0, N_0, C_e$ and $C_h$, but not on $\eps_2$).
Setting $B=B_{2\rho}(x, \pi)$ with $x = \p_{\pi} (p_H)$,
using the Taylor expansion in \cite[Corollary 3.3]{DS2} and the estimates
in \cite[Theorem 1.4]{DS3}, we then get
\begin{align}
\D (B, f) &\leq 2 |B|\, \bE (T, \bC_{2\rho} (x_H, \pi)) +
C \bmo^{1+\gamma_1} \ell^{m+2+\sfrac{\gamma_1}{2}}\nonumber\\
& \leq 2 \omega_m (2\rho)^m \bE (T, \bB_H) + C\bmo^{1+\gamma_1} \ell^{m+2+\sfrac{\gamma_1}{2}}\, .\label{e:Dir-Ex}
\end{align}
Consider next the cylinder $\bC_{64 r_L} (p_L, \pi)$ and
set $x':= \p_{\pi}(p_L)$. Recall that $|x-x'|\leq |p_H-p_L| \leq C \ell (H)$, where $C$ is a geometric constant (cf.~Proposition~\ref{p:tilting opt}) and set
$\sigma:= 64 r_L+ C \ell (H) =  32 r_H + C \ell (H)$.
If $\lambda$ is the constant in \eqref{e:lambda} and $M_0$ is chosen sufficiently
large (thus fixing a lower bound for 
$M_0$ which depends only on $\delta_2$) we reach
\[
\sigma \leq \left(\frac{1}{2} +\frac{\lambda}{4}\right) \,64\, r_H
\leq \left(1+\frac{\lambda}{2}\right) \rho + C^\sharp \bmo^{\sfrac{1}{2m}} \ell^{1+\beta_2}\, .
\]
In particular, choosing $\eps_2$ sufficiently small we conclude $\sigma \leq (1+\lambda) \rho$ and thus also $\bB_L\subset \bC_{64 r_L} (x', \pi)\subset \bC_{(1+\lambda)\rho} (x, \pi) =: \bC'$. Define $B':= B_{(1+\lambda)\rho} (x, \pi)$, set
$A := \mint_{B'} D (\etaa \circ f)$, let
$\mathcal{A}: \pi\to \pi^\perp$ be the linear map $x\mapsto A\cdot x$
and let $\tau$ be the plane corresponding to $\bG_{\mathcal{A}}$.
Using \cite[Theorem 3.5]{DS2}, we can  estimate
\begin{align}
{\textstyle{\frac{1}{2}}} \int_{B'} \cG (Df, Q \a{A})^2 &\geq |B'| \, \bE (T, \bC', \tau) - C \bmo^{1+\gamma_1} \ell^{m+2+\sfrac{\gamma_1}{2}}\nonumber\\
&\geq |B'| \bE (T, \bB_L, \tau) - C \bmo^{1+\gamma_1} \ell^{m+2+\sfrac{\gamma_1}{2}}\nonumber\\
&\geq \omega_m ((1+\lambda)\rho)^m \bE (T, \bB_L) - C \bmo^{1+\gamma_1} \ell^{m+2+\sfrac{\gamma_1}{2}}\, . \label{e:da_sotto}
\end{align}
Now let $\varpi$ be the $(m+\bar{n})$-dimensional plane containing $\pi = \hat{\pi}_H$ so that $\pi\times\varkappa$ has the least distance to the plane $T_{p_H} \Sigma$. From the bound $|\pi_H - \hat{\pi}_H|\leq C \bmo^{\sfrac{1}{2}} \ell^{1-\delta}$ we conclude that $|\varpi - T_{p_H} \Sigma| \leq C \bmo^{\sfrac{1}{2}}\ell^{1-\delta_2}$. In particular we can apply Lemma \ref{l:rotazioni_semplici} to infer the existence of a $C^{3, \eps_0}$ map $\Psi: \varpi \to \varpi^\perp$ whose graph coincides with $\Sigma$ and satisfies the bounds $\|D\Psi\|_0 \leq C_0 \|D\Psi_H\|_0 + C \bmo^{\sfrac{1}{2}}\leq C \bmo^{\sfrac{1}{2}}\ell^{1-\delta_2} \leq 1$ and $\|D^2 \Psi\|_0 \leq C_0 \bA \leq C_0 \bmo^{\sfrac{1}{2}}$ (recall that $\bA$ denotes the $C^0$ norm
of the second fundamental form of $\Sigma$). 

Let $\varkappa$ be the orthogonal complement of $\pi$ in $\varpi$ and establish the notation $\pi\times \varkappa \ni (y,v) \to \Psi (y,v)$ and $(v,z)\in \varkappa \times \varpi^\perp$. Since the approximation $f$ takes values in $\Sigma$, we infer the existence of a $Q$-valued map $g = \sum_i \a{g_i}$ so that $f (y) = \sum_i \a{g_i (y), \Psi (y, g_i (y)))}$.  By the chain rule we have $D (\Psi (y, g (y))) = \sum_i \a{D_y \Psi (y, g_i (y)) + D_v \Psi (y, g_i (y)) \cdot Dg_i (y)}$.  Recalling that ${\rm osc}\, f \leq C \bmo^{\sfrac{1}{2m}} \ell^{1+\beta_2}$ we obtain the same bound for the oscillation of $g$ and
thus conclude the existence of a constant vector $\bar{v}\in \varkappa$ such that $|g_i (y) - \bar{v}|\leq C \bmo^{\sfrac{1}{2m}} \ell^{1+\beta_2}$
for every $i$ and every $y\in B$. We thus achieve 
\[
\cG (D (\Psi (y, g(y))), Q \a{D\Psi (y,\bar{v})}) \leq C \bmo^{\sfrac{1}{2}+\sfrac{1}{2m}} \ell^{1+\beta_2} + C \bmo^{\sfrac{1}{2}} \ell^{1-\delta_2} |Dg| (y)\qquad \forall y\in B\, .
\]
Next, $|D\Psi (y, \bar{v}) - D\Psi (x, \bar{v})|\leq C_0 \bmo^{\sfrac{1}{2}} \rho$, where the latter constant $C_0$ is indeed independent of $\beta_2, \delta_2, M_0, N_0, C_e$ and $C_h$.
Therefore, if we set $\tilde{A} = \mint_{B'} \etaa (D (\Psi (y, g ))) = \mint_{B'} D (\etaa\circ \Psi (y, g))$, we infer
\[
\int_{B'} \cG (D (\Psi (y, g (y))), Q \llbracket \tilde{A}\rrbracket)^2\, dy \leq C_0 \bmo \rho^{m+2} + C \bmo \D (B, g) + C \bmo^{1+\sfrac{1}{m}} \rho^{m+2}\, .
\]
Observe next that $\cG (Df, Q \a{A})^2 = \cG (Dg, Q \a{\bar{A}})^2 + \cG (D (\Psi (y, g)), Q 
\llbracket\tilde A\rrbracket)^2$, where $\bar{A} = \mint_{B'} D (\etaa\circ g)$. We thus conclude
\begin{align}
\D (B, g) &\leq 2 \omega_m (2\rho)^m  \bE (T, \bB_H) +C  \bmo^{1+\gamma_1} \rho^{m+2}\, .\label{e:da_sopra_again}\\
\int_{B'} \cG (Dg, Q \llbracket \bar{A}\rrbracket)^2&\geq 2 \omega_m ((1+\lambda) \rho)^m \bE (T, \bB_L) - C \bmo \D (B, g)\nonumber\\ 
&\qquad - C_0\bmo \rho^{m+2}- C  \bmo^{1+\gamma_1} \rho^{m+2}\, .\label{e:da_sotto_again}
\end{align}

\medskip

%

{\bf Step 2: harmonic approximation.} From now on, to simplify our notation,
we use $B_s (y)$ in place of $B_s (y, \pi)$. Set $p:=\p_{\pi} (p_J)$. 
From \eqref{e:eccesso alto} we infer that $8 r_J\, \bA \leq 8 r_J \bmo^{\sfrac{1}{2}} \leq E^{\sfrac{3}{8}}$ for $\eps_2$
sufficiently small. 
Therefore, for every positive $\bar{\eta}$,
we can apply \cite[Theorem 1.6]{DS3} to the cylinder $\bC$
and achieve a map $w: B_{8r_J} (p, \pi)
\to \Iq (\pi^\perp)$ of the form
$w = (u, \Psi (y, u))$ for a $\D$-minimizer $u$ and such that
\begin{gather}
(8\,r_J)^{-2} \int_{B_{8r_J} (p)} \cG (f,w)^2 + \int_{B_{8r_J} (p)} (|Df|-|Dw|)^2 \leq \bar{\eta} \,E \,(8\,r_J)^m,\label{e:armonica_vicina}\\
\int_{B_{8r_J} (p)} |D (\etaa \circ f) - D (\etaa \circ w)|^2 \leq \bar{\eta}\, E \, (8\,r_J)^m\, .
\label{e:armonica_vicina_2}
\end{gather}
Now, since $D (\etaa\circ u) = \etaa\circ Du$ is harmonic we have $D (\etaa \circ u) (x) = \mint_{B'} (\etaa \circ Du)$. So we can combine  \eqref{e:armonica_vicina}
and \eqref{e:armonica_vicina_2} with \eqref{e:da_sotto_again} to infer
\begin{align}
\int_{B_{(1+\lambda) \rho} (x)} \cG \big(Du, Q\llbracket\etaa \circ Du (x)\rrbracket\big)^2 &\geq
2 \omega_m ((1+\lambda) \rho)^m \bE (T, \bB_L) - C \bmo \D (B, u)\nonumber\\ 
& - C_0\bmo \rho^{m+2} - C  \bmo^{1+\gamma_1} \rho^{m+2} - C_0 \bar\eta^{\sfrac{1}{2}} E \rho^m\, .\label{e:stima_intermedia_2}
\end{align}
Now, recall that $\bE (T, \bB_L) \geq C_e \bmo \ell (L)^{2-2\delta_2} \geq 2^{2\delta_2-2} \bE (T, \bB_H)$ and that $E \leq C \bmo \rho^{2-2\delta_2}$. We can therefore combine \eqref{e:stima_intermedia_2} with \eqref{e:eccesso J}, \eqref{e:da_sopra_again} and \eqref{e:armonica_vicina} to achieve
\begin{align}
\int_{B_{(1+\lambda) \rho} (x)} \cG \big(Du, Q\llbracket D(\etaa \circ u) (x)\rrbracket\big)^2
&\geq \big(2^{2\delta_2-2 - m} - \frac{C_0}{C_e} - C \bar{\eta}^{\sfrac{1}{2}} - C \bmo^{\gamma_1}\big)\int_{B_{2\rho} (x)} |Du|^2\, .
\nonumber
\end{align}
It is crucial that the constant $C$, although depending upon $\beta_2, \delta_2, M_0, N_0, C_e$ and $C_h$, does not depend on $\bar\eta$ and $\eps_2$, whereas $C_0$ depends only upon $Q, m, \bar{n}$ and $n$. 
So, if $C_e$ is chosen sufficiently large, depending only upon $\lambda$ (and hence upon $\delta_2$), we can require that
$2^{2\delta_2 -2-m} - \frac{C_0}{C_e} \geq 2^{3\delta_2/4 -2 -m}$. We then require $\bar{\eta}$ and $\eps_2$ to be sufficiently small
so that  $2^{3\delta_2/4 -2 -m} - C \bmo^{\sfrac{1}{2m}} - C \bar{\eta}^{\sfrac{1}{2}} \geq 2^{\delta_2-2-m}$.
We can now apply Lemma \ref{l:UC} and Proposition \ref{p:harmonic split} to $u$ and conclude
\begin{align*}
\hat{C}^{-1} \int_{B_{(1+\lambda) \rho} (x)} |Du|^2 \leq \int_{B_{\ell/8} (q)} \cG (Du, Q \a{D (\etaa\circ u})^2
\leq \hat{C} \ell^{-2} \int_{B_{\ell/8} (q)} \cG (u, Q \a{\etaa\circ u})^2\, ,
\end{align*}
for any ball $B_{\ell/8} (q) = B_{\ell/8} (q, \pi) \subset B_{8r_J} (p, \pi)$,
where $\hat{C}$ depends upon $\delta_2$ and $M_0$. In particular, being these constants independent of
$\eps_2$ and $C_e$, we can use the previous estimates and reabsorb error terms (possibly choosing $\eps_2$ even smaller and $C_e$ larger) to conclude
\begin{align}
\bmo \, \ell^{m+2-2\delta_2} &\leq \tilde C \ell^m \,\bE (T, \bB_L) \leq \bar{C} \int_{B_{\ell/8} (q)} 
\cG (Df, Q\a{D (\etaa \circ f)})^2\notag\\
&\leq \check{C} \ell^{-2} \int_{B_{\ell/8} (q)} \cG (f, Q \a{\etaa \circ f})^2,\label{e:stima dritta}
\end{align}
where $\tilde C$, $\bar{C}$ and $\check{C}$ are constants which depend upon $\delta_2$, $M_0$ and $C_e$, but not
on $\eps_2$.

\medskip

{\bf Step 3: Estimate for the $\cM$-normal approximation.}
Now, consider any ball $B_{\ell/4} (q, \pi_0)$ with $\dist (L, q)\leq 4\sqrt{m} \,\ell$
and  let $\Omega:= \Phii (B_{\ell/4} (q, \pi_0))$.
Observe that $\p_{\pi} (\Omega)$ must contain a ball
$B_{\ell/8} (q', \pi)$, because of the estimates on $\phii$ and $|\pi_0-\hat{\pi}_H|$, and in turn
it must be contained in $B_{8r_J} (p, \pi)$.
Moreover, $\p^{-1} (\Omega) \cap \supp (T) \supset \bC_{\ell/8} (q', \pi)\cap \supp (T)$
and, for an appropriate geometric constant $C_0$, 
$\Omega$ cannot intersect a Whitney region $\cL'$
corresponding to an $L'$ with $\ell (L') \geq C_0 \ell (L)$.
In particular, Theorem~\ref{t:approx} implies that
\begin{equation}\label{e:masses}
\|\bT_F - T\| (\p^{-1} (\Omega)) + \|\bT_F - \bG_f\| (\p^{-1} (\Omega)) \leq C \bmo^{1+\gamma_2} 
\ell^{m+2+\gamma_2}.
\end{equation}
Let now $F'$ be the map such that 
$\bT_{F'} \res (\p^{-1} (\Omega)) = \bG_f \res (\p^{-1} (\Omega))$ and let $N'$ be the corresponding
normal part, i.e. $F' (x) = \sum_i \a{x+N'_i (x)}$.
The region over which $F$ and $F'$ differ is contained in the projection
onto $\Omega$ of $(\im (F) \setminus \supp (T))
\cup (\im(F') \setminus \supp (T))$ and
therefore its $\cH^m$ measure is bounded as in \eqref{e:masses}.
Recalling the height bound on $N$ and $f$, we easily conclude $|N|+|N'| \leq C \bmo^{\sfrac{1}{2m}}
\ell^{1+\beta_2}$, which in turn implies
\begin{equation}
\int_{\Omega} |N|^2 \geq \int_{\Omega} |N'|^2 - C \bmo^{1+\sfrac{1}{m}+\gamma_2} \ell^{m+4+2\beta_2+\gamma_2}\, .
\end{equation}
On the other hand, let $\phii': B_{8r_J} (p, \pi)\to\pi^\perp$
be such that $\bG_{\phii'} =\a{\cM}$ and $\Phii'(z) = (z, \phii'(z))$; then,
applying \cite[Theorem 5.1 (5.3)]{DS2}, we conclude
\[
|N' (\Phii' (z))| \geq \frac{1}{2\sqrt{Q}}\, \cG (f(z), Q \a{\phii' (z)}) \geq 
\frac{1}{4\sqrt{Q}}\, \cG (f(z), Q\a{\etaa\circ f (z)})\,,
\]
which in turn implies
\begin{align}
\bmo\, \ell^{m+2-2\delta_2} &\stackrel{\eqref{e:stima dritta}}{\leq} C \ell^{-2} \int_{B_{\ell/8} (q', \pi)}  \cG (f, Q \a{\etaa \circ f})^2
\leq C \ell^{-2} \int_\Omega |N'|^2\nonumber\\ 
&\leq C \ell^{-2} \int_\Omega |N|^2 + C \bmo^{1+\gamma_2+\sfrac{1}{2m}} \ell^{m+2+2\beta_2+\gamma_2}\, .
\end{align}
For $\eps_2$ sufficiently small, this leads to the second inequality of \eqref{e:split_2}, while
the first one comes from Theorem~\ref{t:approx} and $\bE (T, \bB_L) \geq C_e \bmo \,\ell^{2-2\delta_2}$.

We next complete the proof showing \eqref{e:split_1}. Since $D (\etaa \circ f) (z) =
\etaa \circ Df (z)$ for a.e. $z$, we obviously have
\begin{equation}
\int_{B_{\ell/8} (q', \pi)} \cG (Df, Q \a{D(\etaa \circ f)})^2 \leq  
\int_{B_{\ell/8} (q', \pi)} \cG (Df, Q \a{D\phii'})^2\, .
\end{equation}
Let now $\vec{\bG}_f$ be the orienting tangent $m$-vector to $\bG_f$ and $\tau$
 the one to $\cM$. For a.e. $z$ we have the inequality
\[
C_0 \sum_i |\vec{\bG}_f (f_i (z)) - \vec{\tau} (\phii' (z))|^2 \geq  \cG (Df (z), Q \a{D\phii' (z)})^2\, ,  
\]
for some geometric constant $C_0$, because $|\vec{\bG}_f (f_i (z)) - \vec{\tau} (\phii' (z))| \leq C \bmo^{\gamma_2}$
(thus it suffices to have $\eps_2$ sufficiently small). Hence
\begin{align}
\int_{B_{\ell/8} (q', \pi)} &\cG (Df, Q \a{D\phii'})^2 \leq
C \int_{\bC_{\ell/8} (q', \pi)} |\vec{\bG}_f (z) - \vec\tau (\phii' (\p_{\pi} (z))|^2 d\|\bG_f\| (z)\nonumber\\
&\leq C \int_{\bC_{\ell/8} (q', \pi)} |\vec{T} (z) - \vec\tau (\phii' (\p_{\pi} (z))|^2 d\|T\| (z)
+ C \bmo^{1+\gamma_1} \ell^{m+2+ \gamma_1}\, .\label{e:quasi finale}
\end{align}
Now, thanks to the height bound and to the fact that $|\vec{\tau} - \pi|\leq |\vec{\tau} - \pi_H| + |\pi_H - \pi| \leq C \bmo^{\sfrac{1}{2}} \ell^{1-\delta_2}$
in the cylinder $\hat\bC = \bC_{\ell/8} (q', \pi)$, we have the inequality
\[
|\p (z) - \phii' (\p_{\pi} (z))|
\leq C \bmo^{\sfrac{1}{2m} + \sfrac{1}{2}} \ell^{2+\beta_2-\delta_2} 
\leq C \bmo^{\sfrac{1}{2m} + \sfrac{1}{2}} \ell^{2+\sfrac{\beta_2}{2}} \qquad
\forall z\in \supp (T)\cap \hat\bC\, .
\]
Using $\|\phii'\|_{C^2} \leq C \bmo^{\sfrac{1}{2}}$
we then easily conclude from \eqref{e:quasi finale} that
\begin{align*}
 \int_{B_{\ell/8} (p, \pi)} \cG (Df, Q \a{D\phii'})^2 &\leq C_0 \int_{\hat\bC} 
|\vec{T} (z) - \vec\tau (\p (z))|^2 d\|T\| (z)
+ C \bmo^{1+\gamma_1} \ell^{m+2+ \sfrac{\beta_2}{2}}\, \\
& \leq C_0 \int_{\p^{-1} (\Omega)} |\vec{\bT}_F (z) - \vec{\tau} (\p (z))|^2 d\|\bT_F\| (z) + 
C \bmo^{1+\gamma_2} \ell^{m+2+\gamma_2},
\end{align*}
where we used \eqref{e:masses}.

Since $|DN| \leq C \bmo^{\gamma_2} \ell^{\gamma_2}$,
$|N|\leq C \bmo^{\sfrac{1}{2m}} \ell^{1+\beta_2}$
on $\Omega$ and $\|A_{\cM}\|^2 \leq C \bmo$,
applying now  \cite[Proposition 3.4]{DS2} we conclude
\begin{equation*}
\int_{\p^{-1} (\Omega)}
|\vec{\bT}_F (x) - \tau (\p (x))|^2 d\|\bT_F\| (x)
\leq (1+ C \bmo^{2\gamma_2} \ell^{2\gamma_2}) \int_{\Omega} |DN|^2 + C \bmo^{1+\sfrac{1}{m}} \ell^{m+2+2\beta_2}\, .
\end{equation*}
Thus, putting all these estimates together we achieve
\begin{equation}
\bmo \, \ell^{m+2-2\delta_2} \leq C (1 + C \bmo^{2\gamma_2} \ell^{2\gamma_2}) \int_{\Omega} |DN|^2
+ C \bmo^{1+\gamma_2} \ell^{m+2+\gamma_2}\, .
\end{equation}
Since the constant $C$ might depend on the various other parameters but not on $\eps_2$, we conclude that
for a sufficiently small $\eps_2$ we have
\begin{equation}
\bmo \ell^{m+2-2\delta_2} \leq C \int_\Omega |DN|^2\, .
\end{equation}
But $\bE (T, \B_L) \leq C \bmo\, \ell^{2-2\delta_2}$ and thus \eqref{e:split_1}  follows.

\section{Persistence of $Q$-points}

\subsection{Proof of Proposition \ref{p:splitting_II}}
We argue by contradiction. Assuming the proposition
does not hold, there are sequences $T_k$'s and $\Sigma_k$'s satisfying the Assumption~\ref{ipotesi} and radii $s_k$ for which
\begin{itemize}
\item[(a)] either $\bmo (k) := \max \{\bE (T_k, \bB_{6\sqrt{m}}), \mathbf{c} (\Sigma_k)^2\} \to 0$
and $1\geq \bar{s} = \lim_k s_k >0$;  or $s_k \downarrow 0$;
\item[(b)] the sets $\Lambda_k := \{\Theta (x, T_k) = Q\}\cap \bB_{s_k}$ satisfy
$\cH^{m-2+\alpha}_\infty (\Lambda_k) \geq \bar{\alpha} s_k^{m-2+\alpha}$;
\item[(c)] denoting by $\sW (k)$ and $\sS (k)$ the families of 
cubes in the Whitney decompositions related
to $T_k$ with respect to $\pi_0$,
$\sup \big\{\ell (L): L\in \sW(k), L\cap B_{3s} (0, \pi_0) \neq \emptyset\big\} \leq s_k$;
\item[(d)] there exists $L_k \in \sW_e (k)$ with
$L_k \cap B_{19 s/16} (0, \pi_0) \neq \emptyset$ and $\hat\alpha s_k < \ell (L_k) \leq s_k$.
\end{itemize}
It is not difficult to see that $\bE (T_k, \bB_{6\sqrt{m}s_k}) \leq C \bmo (k) s_k^{2-2\delta_2}$,
where the constant $C$ depends only on $\beta_2, \delta_2, M_0, N_0, C_e, C_h$. 
Indeed this follows obviously if $s_k \geq c (M_0, N_0)>0$. Otherwise there is some ancestor $H'_k$ of $L_k$ with $s_k \leq \ell (H'_k) \leq C_0 s_k$
for which $\bB_{6\sqrt{m}s_k}\subset \bB_{H'_k}$.

Consider now the ancestors $H_k$ and $J_k$ of $L_k$ as in Section~\ref{ss:splitting}, and the corresponding Lipschitz approximations $f_k$. 
Consider next the radius $\rho_k := 5/4 s_k  + 2 r_{L_k}$ and observe that 
\cite[Theorem 1.4]{DS3} can be applied to the cylinder $\hat{\bC}_k:= \bC_{5\rho_k } (0, \hat{\pi}_{H_k})$: again as above, either $s_k \geq c (M_0, N_0)$, and the theorem
can be applied using the estimates on the height of $T$ in $\bC_{5\sqrt{m}} (0, \pi_0)$ and of its excess in $\bB_{6\sqrt{m}}$, or $s_k$ is smaller and then we can use the ancestor $H'_k$ of the argument above. 
We thus have
\begin{equation}\label{e:decays}
\bE (T_k, \hat{\bC}_k, \hat{\pi}_{H_k})
\leq C \bmo (k)\, s_k^{2-2\delta_2}
\quad\text{and}\quad \bh (T_k, \hat{\bC}_k (0, \hat{\pi}_{H_k}), \hat{\pi}_{H_k}) \leq C \bmo (k)^{\sfrac{1}{2m}} s_k^{1+\beta_2}.
\end{equation}
We denote by $g_k$ the $\hat{\pi}_{H_k}$ approximation in the cylinder $\bC_k := \bC_{\rho_k} (0, \hat{\pi}_{H_k})$.
Observe that $f_k$ and $g_k$ are defined on the same plane and we also denote by
$B_k$ the ball on which $f_k$ is defined. On $B_k$, which is contained in the domain of definition of $g_k$, the two maps $g_k$ and $f_k$ coincide outside of a set of measure at most $C \bmo (k)^{1+\gamma_1} s_k^{m + 2-2\delta_2 + \gamma_1}$  and their oscillation is estimated with $C \bmo^{\sfrac{1}{2m}} s_k^{1+\beta_2}$. We can therefore conclude that
\[
\int_{B_k} \cG (f_k, g_k)^2 \leq C \bmo (k)^{1+\gamma_1 + \sfrac{1}{2m}} s_k^{m+4 + 2\beta_2-2\delta_2 + \gamma_1}
\]
From Proposition~\ref{p:splitting} \eqref{e:split_1} we easily conclude
\begin{equation}\label{e:dal basso}
E_k := \bE (T_k, \bC_k, \hat{\pi}_{H_k})\geq c_0 \bE (T_k, \bB_{L_k})\geq c_0 C_e \bmo (k) \ell (L_k)^{2-2\delta_2}
\geq c_0 (\hat{\alpha}) \bmo (k) s_k^{2-2\delta_2}.
\end{equation}
Moreover, applying Proposition~\ref{p:splitting} and arguing as in Step 1 and Step 2 in Section~\ref{ss:splitting}, we find a ball 
$B'_k\subset \hat{\pi}_{H_k}$ contained in $B_{5s_k/4}$ and with radius at least $\ell (L_k)/8$ such that
\begin{equation}
\int_{B'_k} \cG (f_k, Q\a{\etaa\circ f_k})^2 \geq \bar{c}\,\bmo (k)\, \ell (L_k)^{m+4-2\delta_2}
\geq c_1 (\hat\alpha) \,\bmo (k) \, s_k^{m+4-2\delta_2}\, 
\end{equation}
(cf.~\eqref{e:stima dritta}). Since either $\bmo (k)\downarrow 0$ or $s_k\downarrow 0$, we obviously conclude from \eqref{e:decays} that
\begin{equation}
\int_{B'_k} \cG (g_k, Q\a{\etaa\circ g_k})^2 \geq c (\hat\alpha) \, s_k^{m+2} E_k\, ,
\end{equation}
where the constant $c (\hat{\alpha})$ is positive and depends also upon $\beta_2, \delta_2, M_0, N_0, C_e$ and $C_h$.

Define next $\bA_k^2 := \|A_{\Sigma_k\cap \bC_k}\|^2 \leq C_0 \bmo (k)$.
Note that by \eqref{e:dal basso}, we have that $\bA_k^2\,s_k^2\leq C^\star\, E_k$, for some
$C^\star$ independent of $k$. In particular, 
since either  $s_k\downarrow 0$ or $\bmo (k)\downarrow 0$, it turns out that,
for $k$ large enough, $\bA_k s_k \leq E_k^{\sfrac{3}{8}}$.
For any given $\eta>0$ we can then apply \cite[Theorem 1.6]{DS3} whenever $k$ is large enough. We thus find a sequence of multivalued 
maps $w_k = (u_k, \Psi_k (x, u_k))$ on 
$B_{5s_k/4} (0, \hat{\pi}_{H_k})$ so that each $u_k$ is $\D$-minimizing and 
\begin{equation}
s_k^{-2} \int_{B_{5s_k/4} (0, \hat{\pi}_{H_k})} \cG (g_k, w_k)^2 + \int_{B_{5s_k/4} (0, \hat{\pi}_{H_k})} (|Dg_k| - |Dw_k|)^2 = 
o (E_k) s_k^m\, ,
\end{equation}
{ where the domain of $\Psi_k$ is an $m+\bar{n}$-dimensional plane which includes $\hat{\pi}_{H_k}$ but might change with $k$, cf.~\cite[Remark 1.5]{DS3}; observe also that $\Lip (\Psi_k) \leq C E_k^{\sfrac{1}{2}}$, again cf.~\cite[Remark 1.5]{DS3}}.

Up to rotations (so to get $\hat{\pi}_{H_k} = \pi_0 = \R^m \times \{0\}$ { and ${\rm Dom}\, (\Psi_k)= \R^{m+\bar n}\times \{0\}$}) and dilations (of a factor $s_k$) of the system
of coordinates,  we then end up with
a sequence of $C^{3, \eps_0}$ $(m+\bar{n})$-dimensional submanifolds $\Gamma_k$
of $\R^{m+n}$, area-minimizing currents $S_k$ in $\Gamma_k$,
functions $h_k $ and $\bar{w}_k$
with the following properties:
\begin{enumerate}
\item the excess $E_k := \bE (S_k, \bC_5 (0, \pi_0))$ and the height $\bh (S_k, \bC_5 (0, \pi_0), \pi_0)$ converge to $0$
{ (note that the constant $E_k$ defined here equals the one in \eqref{e:dal basso})};
\item $\bA_k^2 := \|A_{\Gamma_k}\|^2 \leq C^\star E_k$ and hence it also converges to $0$;
\item $\Lip (h_k) \leq C E_k^{\gamma_1}$;
\item $\|\mathbf{G}_{h_k} - S_k \| (\bC_{5/4} (0, \pi_0)) \leq C E_k^{1+\gamma_1}$;
\item $\bar{w}_k = (\bar{u}_k, \Psi_k (x, \bar{u}_k))$ for some $\D$-minimizing $\bar{u}_k$ in $B_{5/4} (0, \pi_0)$ and
\begin{equation}\label{e:vicinanza}
\int_{B_{5/4}} \left((|Dh_k| - |D\bar{w}_k|)^2 + \cG (h_k, \bar w_k)^2 \right) = o (E_k)\, ,
\end{equation}
{ (where with abuse of notation we keep the symbol $\Psi_k$ for the map whose graph coincides with $\Gamma_k$});
\item for some positive constant $c (\hat{\alpha})$ (depending also upon $\beta_2, \delta_2, M_0, N_0, C_e$ and $C_h$),
\begin{equation}\label{e:dal_basso_20}
\int_{B_{5/4}} \cG (h_k, \etaa \circ h_k)^2 \geq c E_k\, ;
\end{equation}
\item $\Xi_k := \{\Theta (S_k, y) = Q\} \cap \bB_1$ has the property that
$\cH^{m-2+\alpha}_\infty (\Xi_k) \geq \bar\alpha>0$
and $0 \in \Xi_k$.
\end{enumerate}
Consider the projections $\bar{\Xi}_k := \p_{\pi_0} (\Xi_k)$. 
We are therefore in the position of applying \cite[Theorem 1.7]{DS3} to conclude that, for every $\varpi >0$ there is a $\bar{s} (\varpi)>0$
(which depends also upon the various parameters $\alpha, \bar{\alpha}, \hat{\alpha}, \beta_2, \delta_2, M_0, N_0, C_e$ and $C_h$) such that
\begin{equation}\label{e:Qpunto}
\limsup_{k\to \infty} \max_{x\in \bar{\Xi}_k} \mint_{B_\rho (x)} \cG (h_k, Q\a{\etaa\circ h_k})^2 \leq \varpi E_k \qquad \forall \rho < \bar{s} (\varpi)\, .
\end{equation}
Up to subsequences we can assume that $\bar{\Xi}_k$ (and hence also $\Xi_k$) converges,
in the Hausdorff sense, to 
a compact set $\Xi$, which is nonempty. 
Moreover, { consider the Dir-minimizing maps 
$x\mapsto \hat{u}_k (x) = E_k^{-\sfrac{1}{2}} \sum_i \a{(\bar{u}_k)_i (x) - \etaa\circ \bar u_k (x)}$. Note that, by \eqref{e:vicinanza} and \eqref{e:Qpunto} we have
\[
\limsup_k \int_{B_{\hat{s}} (x_k)} |\hat{u}_k|^2 < \infty
\]
for some fixed $\hat{s} = \bar{s} (1) >0$ and some sequence $\{x_k\}\subset B_1$. In particular, since 
\[
\limsup_k \int_{B_{5/4}} |D |\hat{u}_k||^2 \leq \limsup_k \D (\hat{u}_k, B_{5/4}) < \infty\, ,
\]
we easily conclude that $\int_{B_{5/4}} |\hat{u}_k|^2$ is bounded independently of $k$.
Thus, by \cite[Proposition 3.20]{DS1}, $\hat{u}_k$
converges, strongly in $L^2 (B_{5/4})$ and up to subsequences,}
to a $\D$-minimizing function $u$ with $\etaa\circ u =0$. Observe that 
\[
\int_{B_{5/4}} \cG(\Psi_k (x, \bar{u}_k), Q\a{\etaa \circ \Psi_k (x, \bar{u}_k)})^2\leq C \Lip (\Psi_k)^2 \int_{B_{5/4}} \cG (\bar{u}_k,
Q \a{\etaa\circ \bar{u}_k})^2 \leq C E_k^2\, .
\] 
Thus \eqref{e:vicinanza} and \eqref{e:dal_basso_20} easily imply that  
\begin{equation}\label{e:dal_basso_30}
\liminf_k \int_{B_{5/4}} \cG (\hat{u}_k, Q\a{0})^2 \geq \liminf E_k^{-1} \int_{B_{5/4}} \cG (\bar{u}_k, \etaa \circ \bar{u}_k)^2 \geq c > 0\, .
\end{equation}
From the strong $L^2$ convergence of $\hat{u}_k$ we then conclude that $u$ does not vanish identically.
On the other hand, by \eqref{e:Qpunto}, \eqref{e:vicinanza} and the strong convergence of $\hat{u}_k$ we conclude that, for any given $\delta>0$ there is a $\bar{s}>0$ such that
\[
\mint_{B_\rho (x)} \cG (u, Q\a{0})^2 \leq \varpi \qquad \forall x\in \Xi \quad \mbox{and}\quad \forall \rho<\bar{s} (\varpi)\, .
\]
Since $u$ is $\D$-minimizing and hence continuous, the arbitrariness of $\varpi$ implies $u \equiv Q \a{0}$ on $\Xi$. On the other hand,
$\cH^{m-2+\alpha}_\infty (\Xi) \geq \limsup_k \cH^{m-2+\alpha}_\infty (\Xi_k) \geq \bar\alpha >0$.
Then, by \cite[Theorem 0.11]{DS1} and
Lemma~\ref{l:UC} we conclude $\bar\Xi=B_{5/4}$, which contradicts $u\not\equiv 0$.

\subsection{Proof of Proposition \ref{p:persistence}}
We fix the notation as in Section~\ref{ss:splitting}
and notice that
\[
E:= \bE (T, \bC_{32r_J} (p_J, \hat{\pi}_H))
\leq C \bmo \ell (L)^{2-2\delta_2} \leq C \bmo \bar\ell^{2-2\delta_2}.
\]
By Proposition \ref{p:splitting} we have
\begin{equation}\label{e:Dir da sotto}
\int_{\cB_{\ell (L)} (\p (p))} |DN|^2 \geq \bar{c}_1 \,\bmo \,\ell(L)^{m+2-2\delta_2}\, . 
\end{equation}
Next, let $p:= (x,y)\in \hat{\pi}_H\times \hat{\pi}_H^\perp$, 
fix a $\bar{\eta}>0$, to be chosen later, and
note that \eqref{e:eccesso alto} allows us to
apply \cite[Theorem 1.7]{DS3}: there exists then $\bar{s}>0$ such that
\begin{equation}
\int_{B_{2\bar{s} \ell (L)} (x, \hat{\pi}_H)} \cG (f, Q\a{\etaa\circ f})^2 \leq \bar\eta\, \bar{s}^m \,\ell (L)^{m+2} E\, .
\end{equation}
Observe that, no matter how small $\bar\eta$ is chosen, such estimate holds
when $\bar{s}$ and $E$ are appropriately small:
the smallness of $E$ is then achieved choosing $\bar\ell$ as small as needed.
 
Now consider the graph $\gr (\etaa\circ f) \res \bC_{2\bar{s}\ell (L)} (x, \hat{\pi}_H)$ and project it down
onto $\cM$. Since $\cM$ is a graph over $\hat{\pi}_H$ of a function $\hat\phii$ with $\|D\hat\phii\|_{C^{2+\kappa}} \leq 
C \bmo^{\sfrac{1}{2}}$ and since the Lipschitz constant of $\etaa\circ f$ is controlled by $C \bmo^{\gamma_1}$,
provided $\eps_2$ is smaller than a geometric constant we have that $\Omega:= \p \left(\gr (\etaa\circ f) \res 
\bC_{2\bar{s}\ell (L)} (x, \hat{\pi}_H)\right)$ contains a ball $\cB_{\bar{s} \ell (L)} (\p (p))$. 

Consider now the map $F' (q) = \sum_i \a{q+N'_i (q)}$ such that $\bT_{F'} \res \p^{-1} (\Omega)= \bG_f \res \p^{-1} (\Omega)$
given by \cite[Theorem 5.1]{DS2}. 
Consider also the map
$\xi: B_{2\bar{s} \ell (L)} (x, \hat{\pi}_H) \ni z\mapsto \p ((z, \etaa\circ f (z)))\in \Omega$. This map is bilipschitz with controlled
constant, again assuming that $\eps_2$ is sufficiently small. Consider now $\hat{n}: \Omega \to \R^{m+n}$
with the property that $\hat{n} (q) \perp T_q \cM$ and $\xi (x)+\hat{n} (\xi (x)) = (x, \etaa \circ f (x))$.
Applying the estimate of \cite[Theorem 5.1 (5.5)]{DS2} we then get
\[
\cG (N' (\xi (x)), Q \a{\etaa\circ N' (\xi(x))}) \leq \cG (N' (\xi (x)), Q\a{\hat{n} (\xi (x))}) 
\leq 2 \sqrt{Q} \,\cG (f(x), Q\a{\etaa\circ f (x)})\, .
\]
Integrating the latter inequality, changing variable and using $\cB_{\bar{s} \ell (L)} (\p (p)) \subset \Omega$, we then
obtain
\begin{equation*}
\int_{\cB_{\bar{s} \ell (L)} (\p(p))} \cG (N', Q\a{\etaa\circ N'})^2 \leq C \,\bar\eta \,\bar{s}^m \,\ell (L)^{m+2} \,E
\leq C \,\bar\eta \,\bmo \,\bar{s}^m \ell (L)^{m+4-2\delta_2}\, .
\end{equation*}
Next, recalling the height bound and the fact that $N$ and $N'$ coincide outside a set of measure
$\bmo^{1+\gamma_1} \ell (L)^{m+2+\gamma_2}$, we infer
\begin{equation}\label{e:L2 da sopra}
\int_{\cB_{\bar{s} \ell (L)} (\p(p))} \cG (N, Q\a{\etaa\circ N})^2 \leq 
C_1\, \bar\eta \,\bmo\, \bar{s}^m \ell (L)^{m+4-2\delta_2} + C_2 \bmo^{1+\gamma_1} \ell (L)^{m+4+\gamma_2 +2\beta_2}\, .
\end{equation}
Since the constants $\bar{c}_1$, $C_1$ and $C_2$
in \eqref{e:Dir da sotto} and \eqref{e:L2 da sopra} are independent of $\ell (L)$
 and $\bar{\eta}$, we fix
$\bar\eta$ (and consequently $\bar{s}$) so small that
$C_1 \bar\eta \leq \bar{c}_1 \frac{\eta_2}{2}$. We therefore achieve from \eqref{e:L2 da sopra}
\begin{equation}
\mint_{\cB_{\bar{s} \ell (L)} (\p(p))} \cG (N, Q\a{\etaa\circ N})^2 \leq \frac{\bar{c}_1}{2}\, \eta_2\, 
\bmo\, \ell (L)^{4-2\delta_2} + C_2 \,\bmo^{1+\gamma_1} \bar{s}^{-m} \ell (L)^{4+\gamma_2 +2\beta_2}\, .
\end{equation}
Having now fixed $\bar{s}$ we choose $\bar{\ell}$ so small that $C_2 \bar{s}^{-m} \bar{\ell}^{2\delta_2+\gamma_2+2\beta_2}
\leq \bar{c}_1 \eta_2/2$. For these choices of the parameters, under the assumptions of the proposition we then infer
\begin{equation}
\mint_{\cB_{\bar{s} \ell (L)} (\p(p))} \cG (N, Q\a{\etaa\circ N})^2 \leq \eta_2 \,\bar{c}_1 \,\bmo \,\ell (L)^{4-2\delta_2}\, .
\end{equation}
The latter estimate combined with \eqref{e:Dir da sotto} gives the desired conclusion.

\section{Comparison between different center manifolds}

\begin{proof}[Proof of Proposition \ref{p:compara}] 
We first verify (i). Observe that
\[
\bE (T', \bB_{6\sqrt{m}}) = \bE (T, \bB_{6\sqrt{m} r})\leq \liminf_{\rho\downarrow r} \bE( T, \bB_{6\sqrt{m} \rho}) \leq \eps_2.
\]
Moreover, since $\Sigma'$ is a rescaling of $\Sigma$, 
$\mathbf{c} (\Sigma') \leq \mathbf{c} (\Sigma)\leq \bmo^{\sfrac{1}{2}}$. 
Therefore, \eqref{e:small ex} is fulfilled by $\Sigma'$ and $T'$
as well;
\eqref{e:pi0_ottimale} follows trivially upon substituting $\pi_0$ with an optimal $\pi$ 
for $T'$ in $\bB_{6\sqrt{m}}$ (which is an optimal plane for $T$ in $\bB_{6\sqrt{m} r}$ by a trivial scaling argument);
\eqref{e:basic} is scaling invariant; whereas $\partial T' \res \bB_{6\sqrt{m}}
= (\iota_{0, r})_\sharp (\partial T \res \bB_{6\sqrt{m} r}) = 0$.

We now come to (ii).
From now on we assume $N_0$ to be so large that $2^{-N_0}$ is much smaller than $c_s$. In this way we know that
$r$ must be much smaller than $1$. 
We have that $\ell (L) = c_s r$, otherwise condition
(a) would be violated.
Moreover, we can exclude that $L\in \sW_n$. Indeed, in this case there
must be a cube $J\in \sW$ with $\ell (J) = 2 \ell (L)$ and nonempty intersection with $L$. It then follows
that, for $\rho:= r + 2\sqrt{m} \,\ell (L) = (1+2\sqrt{m} \, c_s) r$, $B_\rho (0, \pi_0)$ intersects $J$. Again upon assuming $N_0$ sufficiently
large, such $\rho$ is necessarily smaller than $1$. On the other hand, since $2\sqrt{m} c_s <1$ we then have
$c_s \rho<2 \, c_s\, r\leq 2 \, \ell (L) = \ell (J)$.

Next observe that $\bE (T, \bB_{6\sqrt{m} \rho}) \leq C\bmo\rho^{2-2\delta_2}$ for some constant $C$
and for every $\rho\geq r$.
Indeed, if $\rho$ is smaller than a threshold $r_0$ but larger than $r$, then
$\bB_{6\sqrt{m}\rho}$ is contained in the ball $\bB_J$ for some ancestor $J$ of $L$ with $\ell (J)\leq C \rho$,
where the constant $C$ and the threshold $r_0$ depend upon the various parameters, but not upon $\eps_2$.
Then, 
$\bE (T, \bB_{6\sqrt{m} r}) \leq C \bE (T, \bB_J) \leq C\, \bmo\, \rho^{2 - 2\delta_2}$.
If instead $\rho\geq r_0$, we then use simply 
$\bE (T, \bB_{6\sqrt{m} \rho}) \leq C( r_0) \bE (T, \bB_{6\sqrt{m}}) \leq C (r_0) \,\bmo$. This estimate also has the consequence
that, if $\pi (\rho)$ is an optimal $m$-plane in $\bB_{6\sqrt{m} \rho}$,
then $|\hat{\pi}_L - \pi (\rho)| \leq C \bmo^{\sfrac{1}{2}} \rho^{1-\delta_2}$. 

We next consider the notation introduced in Section~\ref{ss:splitting}, the corresponding cubes
$L\subset H \subset J$ and the $\hat{\pi}_H$-approximation $f$ introduced there.
If $L\in \sW_e$, then by \eqref{e:stima dritta} we get
\begin{equation}\label{e:ancora_dal_basso}
\int_{B_{\ell/8} (x, \hat{\pi}_H)} \cG (f, Q \a{\etaa\circ f})^2 \geq c \,\bmo \,\ell^{m+4-2\delta_2}
\geq c\, \bmo\, r^{m+4-2\delta_2}\, ,
\end{equation}
where $x= \p_{\hat{\pi}_H} (x_H)$ and $c (\beta_2, \delta_2, M_0, N_0, C_e, C_h)>0$.
On the other hand, if $L\in \sW_h$, we can argue as in the proof of Proposition
\ref{p:separ} and use Theorem \ref{t:height_bound} to
conclude the existence of at least two stripes $\bS_1$ and $\bS_2$, at distance $\bar{c} \,\bmo^{\sfrac{1}{2m}} \ell^{1+\beta_2}$ 
with the property that any slice $\langle T, \p_{\hat{\pi}_H}, z\rangle$ with $z\in B_{\ell/8} (x, \hat{\pi}_H)$
must intersect both of them. Since for $x\in K$ such slice coincides with $f (x)$, we then have
\begin{align}
 \int_{B_{\ell/8} (x, \hat{\pi}_H)} \cG (f, Q \a{\etaa\circ f})^2 &\geq c \,
 \bmo^{\sfrac{1}{m}} \ell^{m+2+2\beta_2}
- C  \bmo^{\sfrac{1}{m}} \ell^{2+2\beta_2} |K|\nonumber\\
&\geq c\, \bmo^{\sfrac{1}{m}} \ell^{m+2+2\beta_2} - C \bmo^{1+\gamma_1+\sfrac{1}{m}}
\ell^{(m+2-2\delta_2)(1+\gamma_1) +2\beta_2}\notag\\
&\geq c\, \bmo r^{m+4-2\delta_2}\, .\label{e:ancora_dal_basso 2}
\end{align}

Rescale next through $\iota_{0,r}$ and consider $T':= (\iota_{0,r})_\sharp T$. We also rescale the graph of the corresponding
$\hat{\pi}_H$-approximation $f$ to the graph of a map $g$, which then has the following properties.
If $B\subset \hat{\pi}_H$ is the rescaling of the ball $B_{\ell/8} (x, \hat{\pi}_H)$, then $B\subset B_{3/2}$ and 
the radius of $B$ is larger than $c_s/8$. On $B$ we have the estimate
\begin{equation}\label{e:ancora_da_sotto_10}
\int_B \cG (g, Q \a{\etaa\circ g})^2 = r^{-m-2}  \int_{B_{\ell/8} (x, \hat{\pi}_H)} \cG (f, Q \a{\etaa\circ f})^2
\geq \bar{c} \, \bmo r^{2-2\delta_2}\, .
\end{equation}
The Lipschitz constant of $g$ is the same of that of $f$ and hence controlled by $C \bmo^{\gamma_1} r^{\gamma_1}$.
On the other hand, we have 
\begin{equation}\label{e:nuovo_m0}
\hat{{\bm m}}_0 := \max \big\{\bE (T', \bB_{6 \sqrt{m}}), \mathbf{c} (\Sigma')^2\big\}
 \leq \max \{ C \bmo r^{2-2\delta_2},
\mathbf{c} (\Sigma)^2 r^2\} \leq C \bmo r^{2-2\delta_2}\, .
\end{equation}
Moreover, denoting by $\hat{\bC}$ the rescaling of the cylinder $\bC_{8r_J} (p_J, \hat{\pi}_H)$, we have that
\begin{equation}
\|\bG_g - T'\| (\hat\bC) \leq C \bmo^{1+\gamma_1} r^{2 + \sfrac{\gamma_1}{2}}\, . 
\end{equation}
Finally, since $|\pi-\hat{\pi}_H|\leq C \bmo^{\sfrac{1}{2}} r^{1-\delta_2}$ and because $\cM'$ is
the graph of a function $\phii'$ over $\pi$ with $\|D\phii'\|_{C^{2,\kappa}} \leq C \hat{{\bm m}}_0^{\sfrac{1}{2}}$
and $\|\phii'\|_{C^0} \leq C \hat{{\bm m}}_0^{\sfrac{1}{2m}}$, by \eqref{e:nuovo_m0} we can actually conclude that
$\cM'$ is the graph over $\hat{\pi}_H$ of a map $\hat{\phii}: \hat{\pi}_H\to \hat{\pi}_H^\perp$ with 
$\|D\hat\phii\|_{C^{2,\alpha}} \leq C \bmo^{\sfrac{1}{2}} r^{1-\delta_2}$. 
Similarly, the $\cM'$-approximating map $x\mapsto F' (x) := \sum_i \a{x+N'_i (x)}$ 
coincides
with $T'$ over a subset $\cK'\subset \cM'$ with $|\cM'\setminus \cK'| \leq \hat{{\bm m}}_0^{1+\gamma_2}
\leq C \bmo^{1+\gamma_2} r^{(2-2\delta_2) (1+\gamma_2)}$.

Consider now the projection $\p'$ over $\cM'$ and hence define the sets
\begin{align}
\cH := &\p' (\gr (g))\\
\cJ := & \{q\in \cH : \langle \bT_{F'}, \p', q\rangle = \langle \bG_g, \p', q\rangle \}
\end{align}
Since $\cJ \subset \p' (\im (F')\setminus \supp (T)) \cup \p' (\gr (g)\setminus \supp (T))$, 
we have $|\cH \setminus \cJ| \leq \bmo^{1+\gamma_2} r^{(2-2\delta_2) (1+\gamma_2)}$.
On the other hand, by \cite[Theorem 5.1]{DS2} there is a Lipschitz map $G$ defined on a subset ${\rm Dom}\, (G)$ of $\cH\subset \cM'$ such that
$\im (G) \supset \gr (g|_B)$. We then have $G\equiv F'$ on any point of $\cJ\cap {\rm Dom}\, (G)$, which in turn is contained in $\bB_2\cap \cM'$
(at least provided $\bmo$ is small enough). Consider next
a point $p\in {\rm Dom}\, (G)$ with $p= (x, \hat\phii (x))$ and consider that for this point we have,
by \cite[Theorem 5.1 (5.3)]{DS2},
\[
\cG (g(x), Q\a{\etaa\circ g (x)}) \leq \cG (g (x), Q\a{\hat\phii (x)}) 
\leq C_0 \cG (G(p), Q\a{p}) \, . 
\]
Therefore, using \eqref{e:ancora_da_sotto_10} we can easily estimate 
\begin{align}
\int_{\bB_2\cap\cM'} |N'|^2 &\geq \int_{\cJ\cap {\rm Dom}\, (G)} \cG (G(p), Q\a{p})^2\geq \int_{{\rm Dom}\, (G)} \cG (G(p), Q \a{P})^2 - C \hat{{\bm m}}_0^{\sfrac{1}{m}} |\cH\setminus \cJ|\nonumber\\
&\geq c \, \bmo r^{2-2\delta_2} - C \bmo^{\sfrac{1}{m} +1+\gamma_2} r^{(2-2\delta_2) (1+\gamma_2)}
\geq c \, \bmo r^{2-2\delta_2}\, ,
\end{align}
where all the constants are independent of $\eps_2$ (but depend upon the other parameters) and as usual $\eps_2$ is assumed to be sufficiently small.
Thus finally, by \eqref{e:nuovo_m0} we conclude
\[
 \int_{\bB_2\cap \cM'} |N'|^2 \geq \bar{c}_s \hat{{\bm m}}_0 = \bar{c}_s \max \{\bE (T', \bB_{6\sqrt{m}}), \mathbf{c} (\Sigma')^2\}\, .
\qquad\qquad\qquad\qedhere
\]
\end{proof}

\appendix

\section{Height bound revisited}

In this section we prove a strengthened version of
the so-called ``height bound'' (see \cite[Lemma 5.3.4]{Fed}),
which appeared first in \cite{Alm3}. 
Our proof follows closely that of \cite{SS}.

\begin{theorem}\label{t:height_bound} Let $Q$, $m$, $\bar{n}$ and $n$ be positive integers. Then there are
$\eps (Q,m,\bar{n}, n)>0$ and $C_0 (Q, m, \bar{n}, n)$ with the following property. For $r>0$ and $\bC = \bC_r (x_0)$ assume:
\begin{itemize}
\item[(h1)] $\Sigma \subset \R^{m+n}$ is an $(m+\bar{n})$-dimensional $C^2$ submanifold with $\bA:= \|A_\Sigma\|_0 \leq \eps$;
\item[(h2)] $R$ is an integer rectifiable $m$-current with $\supp (R)\subset \Sigma$ and area-minimizing in
$\Sigma$;
\item[(h3)] $\partial R\res \bC = 0$, $(\p_{\pi_0})_\sharp R\res \bC = Q\a{B_r (\p_{\pi_0} (x_0), \pi_0)}$ 
and $E:= \bE (R, \bC) < \eps$.
\end{itemize}
Then there are $k\in \N\setminus\{0\}$, points $\{y_1, \ldots, y_k\}\subset \R^{n}$ and positive integers $Q_1, \ldots, Q_k$ such that:
\begin{itemize}
\item[(i)] having set $\sigma:= C_0 E^{\sfrac{1}{2m}}$, the open sets $\bS_i := \R^m \times (y_i +\, ]-r \sigma, r \sigma[^n)$
are pairwise disjoint and $\supp (R)\cap \bC_{r (1- \sigma |\log E|)} (x_0) \subset \cup_i \bS_i$;
\item[(ii)] $(\p_{\pi_0})_\sharp [R \res (\bC_{r(1-\sigma |\log E|)} (x_0) \cap \bS_i)] =Q_i  \a{B_{r (1-\sigma |\log E|)}(\p_{\pi_0} (x_0), \pi_0)}$ $\forall i\in \{1, \ldots , k\}$.
\item[(iii)] for every $p\in \supp (R)\cap \bC_{r(1-\sigma |\log E|)} (x_0)$ we have $\Theta (R, p) < \max_i  Q_i + \frac{1}{2}$.
\end{itemize}
\end{theorem}

\begin{remark} Obviously, $\sum_i Q_i = Q$ and hence $1\leq k\leq Q$. Most likely the bound
on the radius of the inner cylinder could be improved to $1-\sigma$. However this would not give us any advantage
in the rest of the paper and hence we do not pursue the issue here. 
\end{remark}

\begin{proof} We first observe that, without loss of generality,
we can assume $x_0=0$ and $r=1$. Moreover, (iii) follows from (i) and (ii) through the monotonicity formula.
Indeed,
let $p\in \supp (R)$ be such
that $\B_\rho (p):=\B_{E^{\sfrac{1}{2m}}} (p) \subset \bC_{1-\sigma |\log E|} (x_0)=: \bC'$. $p$ must be contained in one of the $\bS_i$, say $\bS_1$. Consider
the current $R_1 = R\res (\bS_1 \cap \bC')$. Observe that $R_1$ must be area-minimizing in $\Sigma$, $\Theta (R_1, p) = \Theta (R, p)$  and that $\bE (R_1, \bC') \leq E$. On the other hand,
if $\|A_\Sigma\|$ is smaller than a geometric constant, the monotonicity formula implies 
\[
\mass (R_1\res \bC_\rho (p)) \geq \mass (R_1\res \B_\rho (p))\geq \omega_m (\Theta (R, p)-\textstyle{\frac{1}{4}}) \rho^m = 
\omega_m (\Theta (R, p)-\textstyle{\frac{1}{4}}) E^{\sfrac{1}{2}}\, .
\]
On the other hand, $\mass (R_1\res \bC_\rho (p)) \leq \omega_m Q_1 \rho^m + E = \omega_m Q_1 E^{\sfrac{1}{2}} + E$. Thus, if $E$ is smaller than a geometric constant, we ensure $\Theta (R, p) \leq Q_1 +\frac{1}{2}$. 
This means that, having proved (i) and (ii) for $\sigma = C_0 E^{\sfrac{1}{2m}}$, (iii) would hold if we redefine $\sigma$ as $(C_0 +1)E^{\sfrac{1}{2m}}$.

\medskip

The proof of (i) and (ii) is by induction on $Q$. The starting step $Q=1$ 
is Federer's classical statement (cf. with \cite[Lemma 5.3.4]{Fed} and \cite[Lemma 2]{SS}) and though its proof 
can be easily concluded from what we describe next, our only concern will be to prove the inductive step. Hence,
from now on we assume that the theorem holds for all multiplicities up to $Q-1\geq 1$ and we prove it for $Q$.
Indeed, we will show a slightly weaker assertion, i.e. the existence of numbers $a_1, \ldots, a_k\in \R$ such that the conclusions (i)
and (ii) apply when we replace $\bS_i$ with $\bSig_i = \R^{m+n-1}\times\,  ]a_i-\sigma, a_i+\sigma[$. The general statement
is obviously a simple corollary. To simplify the notation we use $\bar\p$ in place of $\p_{\pi_0}$.

\medskip

{\bf Step 1.} Let $r\geq \frac{1}{2}$ and $a - b > 2\eta = 2 C^\flat E^{\sfrac{1}{2m}}$, where $C^\flat$ is a constant depending only
on $m$ and $n$, which will be determined later. We denote by $W_r (a,b)$ the open set $B_r\times \R^{n-1}\times ]a,b[$.
In this step we show
\begin{equation}\label{e.not_much}
\|R\| (W_r (a,b)) \leq \textstyle{\frac{2Q-1}{2Q}} \omega_m r^m\quad \Longrightarrow\quad \supp (R) \cap W_{r-\eta} (a+\eta, b -\eta) = \emptyset \, .
\end{equation}
Without loss of generality,
we assume $a=0$. For each $\tau\in ]0, \frac{b}{2}[$ consider the currents
$R_\tau := R \res W_r (\tau, b-\tau)$ and $S_\tau := \bar\p_\sharp R_\tau$. It follows from the slicing theory that $S_\tau$ is 
a locally integral current for a.e. $\tau$. There are then functions 
$f_\tau\in BV_{loc} (B_r)$ which take integer values and such that $S_\tau = f_\tau \a{B_r}$.
Since $\|f_\tau\|_1 = \mass (S_\tau)\leq \|R\| (W_r (0,b))\leq \frac{2Q-1}{2Q} \omega_m r^m$, $f_\tau$ must vanish on a set of measure
at least $\frac{\omega_m}{2Q} r^m$. By the relative Poincar\'e inequality,
\begin{equation}\nonumber
\mass (S_\tau)^{1-\sfrac{1}{m}} = \|f_\tau\|_{L^1}^{1-\sfrac{1}{m}}\leq C \|Df_\tau\| (B_r) = C \|\partial (\bar\p_\sharp R_\tau)\| (B_r)
\leq C\|\partial R_\tau\| (\bC_r)\, .
\end{equation}
We introduce the slice $\langle R, \tau\rangle$ relative to the map $x_{m+n}: \R^{m+n}\to \R$ which is the projection on
the last coordinate. Then the usual slicing theory gives that
\begin{equation}\label{e.poinc2}
(\mass (S_\tau))^{1- \sfrac{1}{m}} \leq C\|\partial R_\tau\| (\bC_r)
= C \mass \big(\langle R, \tau\rangle - \langle R, b - \tau\rangle\big)\quad \mbox{for a.e. $\tau$}\, .
\end{equation}  
Let now $\bar{\tau}$ be the supremum of $\tau$'s such that $\mass (S_t) \geq \sqrt{E}$ $\forall t<\tau$. If $\mass (S_0) < \sqrt{E}$, we then set 
$\bar{\tau}:= 0$. If $\bar{\tau}> 0$, observe that, for a.e. $\tau\in [0, \bar{\tau}[$ we have
\begin{align}\label{e:da_integrare}
E^{\frac{m-1}{2m}} &\leq C (\mass (S_\tau))^{1-\sfrac{1}{m}}
\leq
C \left(\mass \big(\langle R, \tau\rangle - \langle R, b - \tau\rangle\big)\right)\, .
\end{align}
Integrate \eqref{e:da_integrare} between $0$ and $\bar{\tau}$ to conclude
\begin{align}
&\bar{\tau} E^{\frac{m-1}{2m}} 
\leq C \int_0^{\bar{\tau}}
 \mass \big(\langle R, \tau\rangle - \langle R, b - \tau\rangle\big)\, d\tau
= C \int_{W_r (0, \bar{\tau})\cup W_r (b-\bar\tau, b)} |\vec{R} \res d x_{m+n}|\, d\|R\|\, .
\end{align}
We then apply Cauchy-Schwartz and recall
\[
\int_{\bC_1} |\vec{R}\res dx_{m+n}|^2\, d\|R\| \leq \bE (R, \bC_1) = E\, 
\]
We then conclude $\bar{\tau} E^{\frac{m-1}{2m}} \leq \bar{C} \sqrt{E}$ for some constant $\bar{C}$ depending only on $m$ and $n$, i.e.~$\bar{\tau} \leq \bar{C} E^{\sfrac{1}{2m}}$,
Set $C^\flat := \bar{C}+2$ and recall that $\eta = C^\flat E^{\sfrac{1}{2m}}$.
Observe also that there must be a sequence of $\tau_k\downarrow \bar{\tau}$ with $\mass (S_{\tau_k}) < \sqrt{E}$. Therefore,
\begin{equation}\label{e:poca_massa}
\|R\| (W_r (\bar{\tau}, b-\bar{\tau}))\leq \liminf_{k\to\infty} \|R\| (W_r (\tau_k, b - \tau_k)) \leq \liminf_{k\to\infty} \mass (S_{\tau_k}) + E
\leq 2 \sqrt{E}\, .
\end{equation}
Assume now the existence of $p\in \supp (T) \cap W_{r-\eta} (\eta, b-\eta)$. By the properties of area-minimizing currents, 
$\Theta (T, p)\geq 1$. Set $\rho:= 2 E^{\sfrac{1}{2m}}$ and $B':= \bB_\rho(p)  \subset W_r (\bar{\tau}, \ell - \bar{\tau})$. By the monotonicity formula, 
$\|R\| (B') \geq  c  \,2^m \omega_m \sqrt{E}$, where $c$ depends only on $\bA$ (recall that $\rho\leq 1$) and approaches $1$ as $\bA$ approaches
$0$. Thus, for $\eps_2$ sufficiently small, this would contradict \eqref{e:poca_massa}. We have therefore proved \eqref{e.not_much}.

\medskip

{\bf Step 2.} We are now ready to conclude the proof of (i) and (ii). Assume
\begin{equation}\label{e:spezza}
\max \big\{\|R\| (W_1 (0, \infty)), \|R\| (W_1 (-\infty, 0))\big\}\leq \textstyle{\frac{1}{2}} \mass (R)\, . 
\end{equation}
Divide the interval
$[0,1[$ into $Q+1$ intervals $[a_i, a_{i+1}[$ and let $W^i :=  W_1 (a_i, a_{i+1})$. For each $i$ consider
$S^i := \bar\p_\sharp (T\res W^i)$. Observe that there must be one $i$ for which $\mass (S^i)\leq (1-\frac{1}{2Q}) \omega_m$. Otherwise
we would have
\[
\omega_m Q + E\geq \mass (R) \geq 2 \sum_i \mass (S^i) \geq 2 (Q+1) \omega_m \textstyle{\frac{2Q-1}{2Q}}\, ,
\]
which is obviously a contradiction if $E$ is sufficiently small. 

It follows from Step 1 that there must be an $i$ so that $\supp (T)$ does not intersect $W_{1-\eta} (a_i+\eta, a_{i+1}-\eta)$.
Consider $R_1 := R \res  W_{1-\eta} (-\infty , a_i+\eta)$ and $R_2 := R\res W_{1-\eta}\times (a_{i+1} -\eta, \infty)$.
 By the constancy theorem $\bar\p_\sharp R_i = k_i \a{B_{1-\eta}}$, where both $k_i$'s
are integers. Indeed, having assumed that $E$ is sufficiently small, each $k_i$ must be nonnegative and their sum
is $Q$. There are now two possibilities.
\begin{itemize}
\item[(a)] Both $k_i$'s are positive. In this case $R_1$ and $R_2$ satisfy again the assumptions of the Theorem with $\bC_{1-\eta} (0)$
in place of $\bC_1 (0)$. After a suitable rescaling we can apply the inductive hypothesis to both currents and hence get the
desired conclusion.
\item[(b)] One $k_i$ is zero. In this case $\mass (R_i)\leq E$ and it cannot be $R_1$, since $\mass (R_1)\geq \frac{1}{2} \mass (R)$ by \eqref{e:spezza}.
Thus it is $R_2$ and, arguing as at the end of Step 1, we
conclude $R\res W_{1-2\eta} (]a_{i+1}+2\eta, \infty[) = 0$.
\end{itemize}

In case (b) we repeat the argument splitting $]-1, 0]$ into $Q+1$ intervals. Once again, either we can ``separate''
the current into two pieces and apply the inductive hypothesis, or we conclude that $\supp (R \res \bC_{1-4\eta})
\subset W_{1-4\eta} (-1-\eta,1+\eta) =: W_{1-4\eta} (a_0,b_0)$. If this is the case, we apply once again the argument above
and either we ``separate'' $R^1 := R \res \bC_{1-6\eta} \times \R^n$ into two pieces, or we conclude that
$\supp (R^1)\subset W_{1-6\eta} (a_1, b_1)$, where  
\[
b_1-a_1 \leq (b_0-a_0) \left(1-\textstyle{\frac{1}{Q+1}} +
\eta\right) \leq 2 \left(1-\textstyle{\frac{1}{Q+2}}\right)
\] 
(provided $\eps_2$ is smaller than a geometric constant). We now iterate this argument a finite number of times, stopping if at any step we ``separate'' the current and can apply the inductive hypothesis, or if the resulting current is contained
in $W_{1- (4+2k) \eta} (a_k, b_k)$ for some $a_k, b_k$ with $b_k-a_k\leq c_1 E^{\sfrac{1}{2m}}$. The constant $c_1$ is chosen larger than
1 and in such a way
that, if $\ell > c_1 E^{\sfrac{1}{2m}}$, then $\ell \frac{Q}{Q+1} + \eta \leq \frac{Q+1}{Q+2} \ell$. Observe that, since $\eta= C^\flat E^{\sfrac{1}{2m}}$, $c_1$ depends only upon $Q$, $m$ and $n$. We now want to estimate from above the maximal number of times $k$ that the procedure above gets iterated. Observe that we must have
\[
c_1 E^{\sfrac{1}{2m}} \leq b_{k-1} - a_{k-1} \leq (b_0 - a_0) \left(\frac{Q+1}{Q+2}\right)^{k-1}\, .
\]
Since $b_0 - a_0 = 2+2\eta$, we get the estimate
\[
- (k-1) \log \left(\frac{Q+1}{Q+2}\right) \leq - \log (2+2\eta) - \log c_1  - \frac{1}{2m} \log E\, .
\]
Since $E$ is assumed to be small we get the bound $k\leq - C \log E$. This completes the proof.
\end{proof}

\section{Changing coordinates for classical functions}

\begin{lemma}\label{l:rotazioni_semplici}
For any $m, n \in \mathbb N \setminus \{0\}$ there are constants $c_0,C_0>0$ with the following properties.
Assume that
\begin{itemize}
\item[(i)] $\varkappa, \varkappa_0\subset \mathbb R^{m+n}$ are $m$-dim. planes with $|\varkappa-\varkappa_0|\leq c_0$ and $0<r\leq 1$;
\item[(ii)] $p=(q,u)\in\varkappa\times \varkappa^\perp$ and $f,g: B^m_{7r} (q, \varkappa)\to \varkappa^\perp$ are Lipschitz functions such that
\begin{equation*}
\Lip (f), \Lip (g) \leq c_0\quad\text{and}\quad |f(q)-u|+|g(q)-u|\leq c_0\, r.
\end{equation*}
\end{itemize}
Then there are two maps $f', g': B_{5r} (p, \varkappa_0)\to \varkappa_0^\perp$ such that
\begin{itemize}
\item[(a)] $\bG_{f'} = \bG_f \res \bC_{5r} (p, \varkappa_0)$ and $\bG_{g'} = \bG_g \res \bC_{5r} (p, \varkappa_0)$;
\item[(b)] $\|f'-g'\|_{L^1 (B_{5r} (p, \varkappa_0))}\leq C_0\,\|f-g\|_{L^1 (B_{7r} (p, \varkappa))}$;
\item[(c)] { if $f\in C^{3, \kappa} (B_{7r} (p, \varkappa))$
then $f'\in C^{3, \kappa} (B_{5r} (p, \varkappa_0))$ with the estimates
\begin{align}
&\|f'- u'\|_{C^0}\leq C \|f-u\|_{C^0} + C |\varkappa-\varkappa_0| r\label{e:ruota_uno}\\
&\|Df'\|_{C^0} \leq C \|Df\|_{C^0} + C |\varkappa-\varkappa_0|\label{e:uota_due}\\
&\|D^2 f'\|_{C^{1, \kappa}} \leq \Phi (|\varkappa-\varkappa_0|, \|D^2 f\|_{C^{1, \kappa}}\label{e:ruota_tre}
\end{align}
where $(q',u')\in \varkappa_0 \times \varkappa_0^\perp$ coincides with the point $(q,u)\in \varkappa\times \varkappa^\perp$ and $\Phi$ is a smooth functions with $\Phi (\cdot, 0) \equiv 0$;}
\item[(d)] $\|f'-g'\|_{W^{1,2} (B_{5r} (p, \varkappa_0))} 
\leq C_0 (1+ \|D^2 f\|_{C^0})\|f-g\|_{W^{1,2} (B_{7r} (p, \varkappa))}$.
\end{itemize}
All the conclusions of the Lemma still hold if we replace the exterior radius $7r$ and interior radius $5r$ with
$\rho$ and $s$: the corresponding constants $c_0$ and $C_0$ (and functions $\Phi$, $\Lambda$ and $\Lambda_\kappa$) will then depend also on the ratio
$\frac{\rho}{s}$.
\end{lemma}
\begin{proof} The case of two general radii $s$ and $\rho$ follows easily from that of $\rho=7r$ and $s=5r$ and a simple
covering argument. 
In what follows, given a pair of points $x\in \varkappa, y\in \varkappa^\perp$ we use the notation $(x,y)$ for the vector $x+y$. 
By translation invariance we can assume that $(q,u)= (0,0)$ (and hence $(q', u') = (0,0)$).
Consider then the maps $F, G: B_{7r} (0, \varkappa)\to \varkappa_0^\perp$ and $I, J: B_{7r} (0, \varkappa)\to \varkappa_0$
given by
\[
F (x) = \p_{\varkappa_0^\perp} ((x,f(x)))\quad\text{and}\quad G(x) = \p_{\varkappa_0^\perp} ((x, g(x))),
\]
\[
I(x)= \p_{\varkappa_0} ((x, f(x)))\quad\text{and}\quad J (x) = \p_{\varkappa_0} ((x, g(x))).
\]
With $c_0\leq 1$ we can easily estimate
\[
|I (x) - I(y)| \geq |x-y| (1- C_0 |\varkappa - \varkappa_0| - C_0 \Lip (f))\, ,
\]
for some geometric constant $C_0$. Thus, if $c_0$ is small enough $I$ and (for the same reason) $J$ are injective Lipschitz maps. Therefore, the graphs $\gr_{\varkappa_0} (f)$ and $\gr_{\varkappa_0} (g)$ of $f$ and $g$ in the ``original'' coordinates system $\varkappa_0\times \varkappa_0$ coincide, in the new coordinate system $\varkappa\times \varkappa$, with the graphs $\gr_\varkappa (f')$ and $\gr_{\varkappa} (g')$ of the functions $f' = F \circ I^{-1}$ and $g'= G \circ J^{-1}$
defined respectively in $D:= I (B_{7r} (0, \varkappa))$ and $\bar{D}:= J (B_{7r} (0, \varkappa))$.
If $c_0$ is chosen sufficiently small, then we also conclude
\begin{equation}\label{e:Lip_bound}
\Lip (I), \; \Lip (J),\; \Lip (I^{-1}),\; \Lip (J^{-1}) \leq 1+C\,c_0,
\end{equation}
and
\begin{equation}\label{e:Linfty_bound}
|I (0)|, |J (0)|\leq C\,c_0\, r,
\end{equation}
where the constant $C$ is only geometric.
Clearly, \eqref{e:Lip_bound} and \eqref{e:Linfty_bound} easily imply that $B_{5r} (0, \varkappa_0) \subset D \cap \bar{D}$ when $c_0$ is smaller than
a geometric constant, thereby implying (a) if we restrict the domain of definition of $f'$ and $g'$  to $B_{5r} (0, \varkappa_0)$.
We claim next that, for a sufficiently small $c_0$,
\begin{equation}\label{e:claim}
|f'(x')-g'(x')|\leq 2 \,|f (I^{-1} (x')) - g (I^{-1} (x'))|
\quad\forall \;x'\in B_{5r} (0, \varkappa_0),
\end{equation}
from which,
using the change of variables formula for bilipschitz homeomorphisms
and \eqref{e:Lip_bound}, (b) follows.

In order to prove \eqref{e:claim}, consider 
any $x'\in B_r (q')$, set $x:= I^{-1} (x')$ and
\[
p_1 := (x, f(x))\in \varkappa\times \varkappa^\perp,\quad
p_2 := (x, g(x))\in \varkappa\times \varkappa^\perp\quad
\text{and}\quad
p_3 := (x', g'(x'))\in \varkappa_0\times \varkappa_0^\perp.
\]
Obviously $|f'(x')- g'(x')|= |p_1-p_3|$ and $|f(x)- g(x)|=|p_1-p_2|$.
Note that, $g(x)= f (x)$ if and only if $g'(x')= f' (x')$, and in this case \eqref{e:claim} follows trivially.
If this is not the case, the triangle with vertices $p_1$, $p_2$ and $p_3$ is non-degenerate.
Let $\theta_i$ be the angle at $p_i$.
Note that, $\Lip (g)\leq c_0$ implies $|\frac{\pi}{2}-\theta_2|\leq C c_0$ and $|\varkappa-\varkappa_0|\leq c_0$ implies
$|\theta_1|\leq C c_0$, for some dimensional constant $C$.
Since $\theta_3 = \pi - \theta_1 - \theta_2$, we conclude
as well $|\frac{\pi}{2} - \theta_3|\leq C c_0$.
Therefore, if $c_0$ is small enough, we have $1 \leq 2\sin \theta_3$, so that,
by the Sinus Theorem,
\begin{equation*}
|f'(x')-g'(x')|= |p_1-p_3| = \frac{\sin \theta_2}{\sin \theta_3}\, |p_1-p_2|
\leq 2 \, |p_1-p_2| = 2 \,|f(x)-g(x)|,
\end{equation*}
thus concluding the claim.

As for (c), observe that $I, F$ and $J, G$ are obviously as regular as $f$ and $g$. So, when the latter are $C^1$, $I$ and $J$ are also $C^1$. In the latter case, if we put suitable coordinates on both $\varkappa$ and $\varkappa_0$ (identifying them with $\mathbb R^m$) we can easily estimate $|d I - {\rm Id}| \leq C_0 (\|Df\|_0 + |\varkappa- \varkappa_0|)$, where $C_0$ is a geometric constant, $d I$ the differential of $I$ and ${\rm Id}$ the identity. Thus for $c_0$ sufficiently small we can apply the inverse function theorem: so $I^{-1}$ is as regular as $I$ and hence as $f$. Since $f' = F\circ I^{-1}$, also $f'$ is as regular as $f$. Recall next that we are assuming $q=0$ and $u=0$. Define the map $\tilde{I} (x) = I^{-1} (x) -x$. Since $f' = F\circ I^{-1}$, the bounds claimed in (c) follows easily if we can prove the very same bounds for the map $\tilde{I} (x)$. If we set $\bar{I} (x) = I (x)-x$, the inverse function theorem gives $\|\tilde{I}\|_{C^1} \leq 2\|\bar{I}\|_{C^1}$ provided $c_0$ is sufficiently small. The bounds on the higher derivatives can then be easily concluded differentiating the identity $d I^{-1} (x) = [d I]^{-1} (I^{-1} (x))$.

We finally come to (d). The estimate $\|f'-g'\|_{L^2}\leq C \|f-g\|_{L^2}$ is an obvious consequence of \eqref{e:claim}. Given next a point $p$ in the graph of $f$, 
resp. in the graph of $g$, we denote by $\sigma (p)$, resp. $\tau (p)$, the oriented tangent plane to the corresponding graphs. Observe that the points
are described by the pairs $(x', f(x'))$ and $(x', g (x'))$, in the coordinates $\kappa\times \kappa^\perp$, and by $(I^{-1} (x'), f (I^{-1} (x)'))$ and
$(J^{-1} (x'), g (J^{-1} (x')))$, in the coordinates $\varkappa_0\times \varkappa_0^\perp$. Thus 
\begin{align}
&|\nabla f' (x') - \nabla g' (x')| \leq C |\sigma (p) - \tau (q)|
\leq  C |\nabla f (I^{-1} (x')) - \nabla g (J^{-1} (x')|\nonumber\\
\leq& C |\nabla f (I^{-1} (x'))-\nabla f  (J^{-1} (x'))| + C |\nabla f (J^{-1} (x')) - \nabla g (J^{-1} (x'))|\nonumber\\
\leq &C \|D^2 f\|_{C^0} |I^{-1} (x') - J^{-1} (x')| + C |\nabla f (J^{-1} (x')) - \nabla g (J^{-1} (x'))|\nonumber\\
\leq & C \|D^2 f\|_{C^0} |f' (x')- g'(x')| + C |\nabla f (J^{-1} (x')) - \nabla g (J^{-1} (x'))|\, .
\end{align}
Integrating this last inequality in $x'$ and changing variables we then conclude 
\[
\|\nabla f' - \nabla g'\|^2 \leq C \|\nabla f- \nabla g\|_{L^2} + C \|D^2f\|_{C^0} \|f'-g'\|_{L^2}\, ,
\] 
which, together with the $L^2$ estimate, gives (d).
\end{proof}

\section{Two interpolation inequalities}
\begin{lemma}\label{l.interpolation}
Let $A>0$ and $\psi\in C^2(B_\rho,\R^n)$ satisfy $\|\psi\|_{L^1}\leq A\,\rho^{m+1}$ and 
$\|\Delta \psi\|_{L^{\infty}}\leq \rho^{-1}\,A$.
Then, for every $r<\rho$ there is a constant $C>0$ (depending only on $m$ and $\frac{\rho}{r}$) such that
\begin{equation}\label{e.interpolation}
\rho^{-1}\,\|\psi\|_{L^\infty(B_r)}+\|D\psi\|_{L^\infty(B_r)}\leq C\, A.
\end{equation}
\end{lemma}
\begin{proof} By a simple covering argument we can, w.l.o.g., assume $\rho=3r$. Moreover, if we apply the scaling
$\psi_r (x) := r^{-1} \psi (rx)$ we see that $\|\psi_r\|_{L^1 (B_3)} = (\rho/3)^{-m-1} \|\psi\|_{L^1 (B_\rho)}$,
$\|\psi_r\|_\infty = (\rho/3)^{-1} \|\psi\|_\infty$, $\|D \psi_r\|_\infty = \|D \psi\|_\infty$ and $\|\Delta \psi_r\|_\infty
= (\rho/3) \|\Delta \psi\|_\infty$. We can therefore assume $r=1$.
Consider the harmonic function $\zeta:B_{2}\to\R$ with boundary data $\psi\vert_{\partial B_{2}}$,
\[
\begin{cases}
\Delta\zeta=0 & \text{in }\;B_{2},\\
\zeta=\psi &\text{on }\;\de B_{2}.
\end{cases}
\]
Set $u:=\psi-\zeta$ and note that $u=0$ on $\de B_{2}$, $\|\Delta u\|_{C^0(B_{2})}\leq A$.
Hence, using the Poincar\'e inequality, we can estimate the $L^1$-norm of $u$ in the following way:
\begin{align*}
\|u\|_{L^1}&\leq \|u\|_{L^2}\leq C\,\|Du\|_{L^2}\leq 
C\,\left(\int_{B_{2}}|\Delta u\,u|\right)^{\sfrac{1}{2}}
 \leq C\,\|\Delta u\|_{C^0}^{\sfrac{1}{2}}\,\|u\|_{L^1}^{\sfrac{1}{2}} \leq C A\, .
\end{align*}
Choose now $a\in ]0,1[$ and $s\in ]1, \infty[$ such that
$\frac{1}{m}+a\left(\frac{1}{s}-\frac{2}{m}\right)+1-a< 0$ (which exist because for $s\to\infty$ and $a\to 1$ 
the expression converges to $-\frac{1}{m}$). By a
classical interpolation inequality, (see \cite{Nir})
\[
\|D u\|_{L^\infty}\leq C\,\|D^2 u\|^a_{L^s}\|u\|_{L^1}^{1-a}+C\|u\|_{L^1}\, .
\]
Using the $L^s$-estimate for the Laplacian, we deduce
\begin{align}\label{e.Du}
\|D u\|_{L^\infty}&\leq C\,\|\Delta u\|^a_{L^s}\|u\|_{L^1}^{1-a}+C\|u\|_{L^1}\leq C\, \|\Delta u\|_\infty^a \|u\|_{L^1}^{1-a}
+ \|u\|_{L^1} \leq C\, A\, .
\end{align}
From \eqref{e.Du} and $u|_{\partial B_2}=0$ it follows trivially $\|u\|_{L^\infty} \leq A$. 
To infer \eqref{e.interpolation}, we observe that, by 
$\|\zeta\|_{L^{1}(B_{2})}\leq \|u\|_{L^{1}(B_{2})}+\|\psi\|_{L^{1}(B_{2})}\leq C\,A$ and the
harmonicity of $\zeta$,
\[
\|\zeta\|_{L^\infty(B_{1})} + \|D\zeta\|_{L^\infty (B_1)} \leq C \|\zeta\|_{L^{1}(B_{2})}\leq C\,A\, . \qedhere
\]
\end{proof}

\begin{lemma}\label{l:interpolation_bis}
For every $m$, $r<s$ and $\kappa$ there is a positive constant $C$ (depending on $m$, $\kappa$ and $\frac{s}{r}$)
with the following property. 
Let $f$ be a $C^{3, \kappa}$ function in the ball $B_s\subset \R^m$. Then
\begin{equation}\label{e:interpolation_bis}
\|D^j f\|_{C^0 (B_r)} \leq C r^{-m-j} \|f\|_{L^1 (B_s)} + C r^{3+\kappa-j} [D^3 f]_{\kappa, B_s}\qquad \forall j\in \{0,1,2,3\}\, .
\end{equation}
\end{lemma}
\begin{proof} A simple covering argument reduces the lemma to the case $s=2r$. Moreover,
define $f_r (x):= f (rx)$ to see that we can assume $r=1$. 
So our goal is to show
\begin{equation}\label{e:ridotta}
\sum_{j=0}^3 |D^j f (y)| \leq C \|f-g\|_{L^1} + C [D^3 f]_\kappa\, \qquad \forall y \in B_1, \forall
f\in C^{3,\kappa} (B_2)\, .
\end{equation}
By translating it suffices then to prove the estimate
\begin{equation}\label{e:ridotta_2}
\sum_{j=0}^3 |D^j f (0)| \leq C \|f\|_{L^1 (B_1)} + C [D^3 f]_{\kappa, B_1} \qquad \forall f\in C^{3,\kappa} (B_1)\, .
\end{equation}
Consider now the space of polynomials $R$ in $m$ variables of degree at most 3, which we write as
$R = \sum_{j=0}^3 A_j x^j$. This is a finite dimensional vector space, on which we can define the norms
$|R| := \sum_{j=0}^3 |A_j|$ and $\|R\| := \int_{B_1} |R(x)|\, dx$.
These two norms must then be equivalent, so there is a constant $C$ (depending only on $m$), such 
that $|R|\leq C\|R\|$ for any such polynomial. In particular,
if $P$ is the Taylor polynomial of third order for $f$ at the point $0$, we conclude
\begin{align*}
\sum_{j=0}^3 |D^j f (0)| &= |P| \leq C \|P\| = C \int_{B_1} |P (x)|\, dx
\leq C\|f\|_{L^1 (B_1)} + C\|f-P\|_{L^1 (B_1)}\\
&\leq C \|f\|_{L^1} + C [D^3 f]_\kappa\, .\qquad\qquad \qedhere
\end{align*}
\end{proof}

\section{Proof of Lemma \ref{l:cambio_tre_piani}}\label{a:cambio_tre_piani}

\subsection{Reduction to special triples of planes} We first observe that, by a simple scaling, we can assume $r=1$.
The rescaling which we apply to any map $\varphi$ is the usual $x\mapsto r^{-1} \varphi (rx) =: \varphi_r$. It is easy to see that \eqref{e:che_fatica} is then scaling invariant.

We next fix the following terminology: we say that $R\in SO (m+\bar{n}+l)$
is a {\em 2d-rotation} if
there are two orthonormal vectors $e_1, e_2$ and an angle $\theta$
such that $R (e_1) = \cos \theta\, e_1 + \sin \theta\, e_2$, $R (e_2) = \cos \theta\, e_2 - \sin \theta\, e_1$ and $R (v)=v$
for every $v\perp {\rm span}\, (e_1, e_2)$. Given a triple $(\bar{\pi}, \bar{\varkappa}, \bar{\varpi})$ we then say that:
\begin{itemize}
\item $R$ is of type A with respect to $(\bar\pi, \bar\varkappa, \bar\varpi)$ if $e_1\in \bar\varkappa$ and $e_2\in \bar\varpi$;
\item $R$ is of type B with respect to $(\bar\pi, \bar\varkappa, \bar\varpi)$ if $e_1\in \bar\pi$ and $e_2\in \bar\varkappa$;
\item $R$ is of type C with respect to $(\bar\pi, \bar\varkappa, \bar\varpi)$ if $e_1\in \bar\pi$ and $e_2\in \bar\varpi$.
\end{itemize}
The following lemma will then allow us to reduce the general case of Lemma \ref{l:cambio_tre_piani} to the particular ones in which $(\bar{\pi}, \bar{\varkappa}, \bar{\varpi})$ is obtained from $(\pi, \varkappa, \varpi)$ through a (small) rotation of type A, B or C.

\begin{lemma}\label{l:linear_algebra} There are constants $C_0 (m, \bar{n}, l)$ and $\bar{N} (m, \bar{n}, l)$ with the following property. If $c_0$ in Lemma \ref{l:cambio_tre_piani} is sufficiently small, then there are $N \leq \bar{N}$ triples $(\pi_j, \varkappa_j, \varpi_j)$ ``joining'' 
$(\pi, \varkappa, \varpi) = (\pi_N, \varkappa_N, \varpi_N)$ with $(\bar\pi, \bar\varkappa,
\bar\varpi) = (\pi_{0}, \varkappa_{0}, \varpi_{0})$ such 
that each $(\pi_j, \varkappa_j, \varpi_j)$ is the
image of $(\pi_{j-1}, \varkappa_{j-1}, \varpi_{j-1})$
under a $2d$-rotation
of type A, B or C and angle $\theta_j$ with $|\theta_j| \leq C_0 (|\pi- \bar{\pi}| + |\varkappa - \bar\varkappa|)$.
\end{lemma}
\begin{proof}
We first show that, if $\varpi = \bar{\varpi}$, or $\varkappa = \bar{\varkappa}$ or $\pi=\bar{\pi}$, then
the claim can be achieved with small $2d$-rotations all of the same type, namely of type B, C and A, respectively.
Assume for instance that $\varpi = \bar{\varpi}$. Let $\omega$ be
the intersection of $\pi$ and $\bar{\pi}$ and
$\omega'$ be
the intersection of $\varkappa$ and $\bar{\varkappa}$.
Pick a vector $e\in \pi$ which is not
contained in $\bar{\pi}$ and is orthogonal to $\omega$.
Let $\bar{e}:= \frac{\p_{\bar{\pi}} (e)}{|\p_{\bar{\pi}} (e)|}$.
Then, $\bar{e}$ is necessarily orthogonal
to $\omega$ and the angle between $\bar{e}$ and $e$
is controlled by $|\pi - \bar{\pi}|$. There is therefore a $2d$-rotation
$R$ such that $R (e) = \bar{e}$ and obviously its angle is controlled by $|\pi-\bar\pi|$.
It turns out that $R$ keeps $\varpi$ and $\omega$ fixed. So the new triple $(R (\pi), R (\varkappa), R (\varpi))$
has the property that $R (\varpi)=\varpi = \bar{\varpi}$ and the dimension of $R(\pi) \cap \bar{\pi}$
is larger than that of $\pi\cap \bar{\pi}$. This procedure can be repeated and after $N\leq m$ times it leads
to a triple of planes $(\pi_N, \varkappa_N, \varpi_N)$ with $\varpi_N= \bar{\varpi}$ and $\pi_N = \bar{\pi}$.
This however implies necessarily $\bar{\varkappa} = \varkappa_N$.

Assume therefore that $\varpi$ and $\bar\varpi$ do not coincide.
Let $\omega := (\varkappa \times \pi)\cap (\bar{\varkappa}\times \bar{\pi})$.
There is then a unit vector $\bar e\in \bar{\varkappa}$ or a unit vector $\bar e\in
\bar{\pi}$ which does not belong to $\pi\times \varkappa$ and which is 
orthogonal to $\omega$. Assume for 
the moment that we are in the first case, and consider
the vector $e:= \frac{\p_{\pi\times \varkappa} (\bar e)}{|\p_{\pi\times \varkappa} (\bar e)|}$.
The vector $e$ forms an angle with the plane $\varkappa$ bounded
by $C_0|\varkappa-\bar\varkappa|$. Therefore there is a rotation
$R$ with angle smaller than $C_0 |\varkappa -
\bar\varkappa|$ of the plane $\pi \times \varkappa$ with the property that $R (\varkappa)$ 
contains $e$ and fixes
$\omega$, which is orthogonal to $e$. 
By the previous step, $R$ can be written
as composition $R_{N'}\circ \ldots \circ R_1$ of 
small $2d$-rotations of type B keeping $\varpi$ fixed. 
Since $e \perp R (\pi)$, we can then find a small $2d$-rotation $S$ of type A
with respect to $(R(\pi), R(\varkappa), \varpi)$
acting on the plane spanned by $e$ and $\bar{e}$
and such that $S (R(\varkappa))\ni \bar{e}$. $S$ keeps then $\omega$ fixed. An 
analogous argument works if the vector $e\in \bar\pi$. 
We therefore conclude that, after applying a finite number of rotations $R_1, \ldots, R_{N'}, 
R_{N'+1}$ of the
three types above, the dimension of 
$R_{N'+1} \circ R_{N'} \circ \ldots \circ R_1 (\pi\times \varkappa)\cap \bar{\pi}\times \bar
\varkappa$ is larger 
than that $\pi\times \varkappa  \cap \bar{\pi}\times \bar\varkappa$ (where the number $N'$ is 
smaller than a geometric constant depending only on $m$ and $\bar{n}$).
Obviously, after at most $m+\bar{n}$ iterations of this argument,
we are reduced to the situation $\pi \times \varkappa = \bar\pi\times
\bar\varkappa$.
\end{proof}

Assume now to have proved Lemma \ref{l:cambio_tre_piani} for some constants $c_0$ and $C_0$ and
for all 2d rotations which are of type A, B or C with respect to one of the two triples of planes. We next
claim that, at the price of possibly enlarging the constants, the Lemma holds for any pair of triples.
To this purpose we now fix two triples as in the statement of the lemma and choose a chain $(\pi_j, \varkappa_j, \varpi_j)$ as in Lemma \ref{l:linear_algebra}. As already observed it suffices to prove the statement when $r=1$, but we
assume of having proved it for {\em any} radius in the case of small $2d$ rotations of type A, B or C.
Lemma \ref{l:linear_algebra} implies that $|\pi_j - \pi_i| + |\varkappa_j - \varkappa_i| \leq \bar C_0 {\rm An}$
for some geometric constant $\bar C_0$. For each $i$ we therefore have Lipschitz maps $\Psi^i : \pi_i \times \varkappa_i \to \varpi_i$ and Lipschitz maps $f^i : B_{4} (0, \pi) \to \Iq (\varkappa\times \varpi)$ whose graph coincides with the ones of $\bar\Psi$
and $f$ (the latter restricted to $\bC_{4} (0, \pi)$): their existence is ensured by \cite[Proposition 5.2]{DS2} which also implies
\begin{align}
\|D\Psi_i\|_{C^0} &\leq \bar{C}_0 \big(\|D\Psi\|_{C^0} + {\rm An} \big)\label{e:ruoto_1}\\
|\Psi_i (0)|& \leq \bar{C}_0 \big(|\Psi (0)| + \|D\Psi\|_{C^0} + {\rm An}\big)\label{e:ruoto_2}\\
\Lip (f^i) &\leq \bar C_0 \big(\Lip (f) + {\rm An} \big)\label{e:ruoto_3}\\
\|f^i\|_{C^0} &\leq \bar{C}_0 \big(\|f\|_{C^0}+ {\rm An}\big)\label{e:ruoto_4}
\end{align}
Set now $r_i := 2^{2-i}$.
By assuming the constant $c_0$ sufficiently small we can therefore assume that the Lemma can be applied to the pairs
$(\pi_{i-1}, \varkappa_{i-1}, \varpi_{i-1})$ and $(\pi_i, \varkappa_i, \varpi_i)$, to the maps $\Psi_{i-1}, \Psi_i, f^{i-1}, f^i$
and to the radius $r_i/4$. In order to streamline the argument, for $j>i$ we use the notation $f^j = {\rm R}_{ij} f^i$ to underline that the graph of $f^j$ coincides, in the cylinder $\bC_{r_i} (0, \pi_i)$, with the graph of $f^i$. Likewise, if $u^i$ is the multivalued map into $\Iq (\varkappa_i)$ such that $f^i (x) = \sum_l \a{(u^i_l (x), \Psi_i (u^i_l (x)))}$, we then denote by ${\rm Av}\, (f^i)$ the map $(\etaa \circ u^i, \Psi_i (\etaa \circ u^i))$. With this notation we observe that $\hat\bef = {\rm R}_{N0} ({\rm Av}\, (f^0)) = {\rm R}_{N0} ({\rm Av}\, (f))$ and $\beg = {\rm Av}\, ({\rm R}_{N0}\, (f^0)) = {\rm Av}\, ({\rm R}_{N0}\, (f))$. We can then estimate 
\begin{align*}
& \|\hat\bef - \beg\|_{L^1 (B_{r_N} (\pi, 0))} \leq 
\underbrace{\|{\rm R}_{N(N-1)} ({\rm R}_{(N-1)0} ({\rm Av}\, (f))) - {\rm R}_{N(N-1)} ({\rm Av}\, ({\rm R}_{(N-1)0} (f)))\|_{L^1 (B_{r_N}(\pi_N, 0))}}_{(I)}\nonumber\\
&\qquad + \underbrace{\|{\rm R}_{N(N-1)} ({\rm Av}\, ({\rm R}_{(N-1)0} (f))) -
{\rm Av}\, ({\rm Rot}_{N(N-1)} ({\rm R}_{(N-1)0} (f)))\|_{L^1 (B_{r_N} (\pi_N, 0))}}_{(II)}\, .
\end{align*}
Now, to the first summand we apply Lemma \ref{l:rotazioni_semplici}(b) and we bound it with
\[
(I) \leq \bar{C}_0 \|{\rm R}_{(N-1)0} ({\rm Av}\, (f)) - {\rm Av}\, ({\rm R}_{(N-1)0} (f))\|_{L^1 (B_{r_{N-1}} (0, \pi_{N-1}))}\, .
\]
As for the second summand, observing that ${\rm R}_{(N-1)0} (f) = f^{N-1}$, we can apply Lemma \ref{l:cambio_tre_piani} for the special case of a 2d rotation of type A,B or C and conclude
\begin{align*}
(II) &\leq \bar{C}_0 \big(\|f^{N-1}\|_{C^0} + |\pi_N- \pi_{N-1}| + |\varkappa_N-\varkappa_{N-1}\big) \big(\D (f^{N-1})\\
&\qquad\qquad\qquad + \|D\Psi_{N-1}\|^2_{C^0} + (|\pi_N- \pi_{N-1}| + |\varkappa_N-\varkappa_{N-1}| )^2\big)\\
&\leq \bar{C}_0 \big(\|f^{N-1}\|_{C^0} + {\rm An}\big) \big(\D (f^{N-1}) + \|D\bar \Psi\|^2_{C^0} + {\rm An}^2\big)\, .
\end{align*}
On the other hand, by the Taylor expansion of the mass in \cite[Corollary 3.3]{DS2},
\begin{align}
\D (f^{N-1}) &\leq 4 \bE (\bG_{f^{N-1}}, \bC_{r_{N-1}} (0, \pi_{N-1}))\nonumber\\
&\leq 4 \bE (\bG_{f^{N-1}}, \bC_{r_{N-1}} (0, \pi_{N-1}), \pi_N) + \bar C_0 |\pi_{N-1} - \pi_N|^2\nonumber\\
&\leq 4 \bE (\bG_f, \bC_{8} (0, \bar\pi)) + \bar C_0 {\rm An}^2 \leq 8 \D (f) + \bar C_0{\rm An}^2\, .
\end{align}
Putting all these estimates together we then conclude
\begin{align*}
&\|\hat\bef - \beg\|_{L^1 (B_{r_N} (\pi, 0))} \leq \bar{C}_0 \|{\rm R}_{(N-1)0} ({\rm Av}\, (f)) - {\rm Av}\, ({\rm R}_{(N-1)0} (f))\|_{L^1 (B_{r_{N-1}} (0, \pi_{N-1}))}\nonumber\\
&\qquad + \bar{C}_0 \big(\|f\|_{C^0} + {\rm An}\big) \big(\D (f) + \|D\Psi\|^2_{C^0} + {\rm An}^2\big)\, .
\end{align*}
We can now iterate $N-1$ more times this argument to finally achieve
\[
\|\hat\bef - \beg\|_{L^1 (B_{r_N} (\pi, 0))} \leq
\bar{C}_0 \big(\|f\|_{C^0} + {\rm An}\big) \big(\D (f) + \|D\Psi\|^2_{C^0} + {\rm An}^2\big)\, .
\]
Of course this is not yet the estimate claimed in Lemma \ref{l:cambio_tre_piani} since the inner radius $r_N$ equals $2^{2-N}$ rather than $4$. However a simple covering argument allows to conclude the proof. 
In the remaining sections we focus our attention on 2D rotations of coordinates of type A, B and C. 

\subsection{Type A} As already observed it suffices to show the lemma in the case $r=1$.
We use the notation $(z,w)\in \varkappa\times \varpi$ and $(\bar z, \bar w)\in \bar\varkappa\times \bar\varpi$
for the same point. In what follows we will drop the $\cdot$ when writing the usual products between
matrices. We then have
$\bar z = U z + V w$ and $\bar w = W z + Z w$,
where the orthogonal matrix
\[
L := \left(\begin{array}{ll}
            U & V\\
W & Z
           \end{array}\right)\, 
\]
has the property that $|L-{\rm Id}|\leq C_0 {\rm An}$.
Clearly, $\Psi$ and $\bar \Psi$ are related by the
identity 
\begin{equation}\label{e:rotazione_A}
W z + Z \Psi (x,z) = \bar\Psi (x,U z + V \Psi (x,z))\, .
\end{equation}
Fix $x$ and $\hat f (x) = \sum_i \a{(\hat u_i (x),
\Psi (x,\hat u_i (x)))} =: \sum_i \a{(z_i, \Psi (x, z_i))}$. We then have
\[
\beg (x) = (a,b) := 
\Big(\frac{1}{Q} \sum z_i, \Psi \Big(x, \frac{1}{Q} \sum z_i\Big)\Big) \quad \text{in }\;\varkappa\times \varpi,
\] 
and 
\[
\hat\bef (x) = L^{-1} \left(U \textstyle{\frac{1}{Q}} \sum z_i + V \textstyle{\frac{1}{Q}} \sum \Psi (x, z_i),
\bar\Psi \left(x, U \textstyle{\frac{1}{Q}} \sum z_i + V \textstyle{\frac{1}{Q}} \sum \Psi (x, z_i)\right)\right)
=: L^{-1} (c,d)\, .
\]
Since $L$ is orthogonal, we have 
\begin{align*}
&|\hat\bef (x) - \beg (x)| = |L (a,b)-(c,d)|\nonumber\\
= &\left| \left(V \left(\Psi \left( x, \textstyle{\frac{1}{Q}} \sum z_i\right)
- \textstyle{\frac{1}{Q}} \sum_i \Psi (x, z_i)\right), W \textstyle{\frac{1}{Q}} \sum_i z_i + Z
\Psi \left(x, \textstyle{\frac{1}{Q}} \sum_i z_i\right)\right.\right.\nonumber\\
&\qquad\qquad\qquad\qquad\qquad\qquad\qquad\qquad\quad - \left.\left.\bar\Psi \left( x, U \textstyle{\frac{1}{Q}} \sum z_i + 
V \textstyle{\frac{1}{Q}} \sum \Psi (x, z_i)\right)\right)\right|\nonumber\\
\stackrel{\eqref{e:rotazione_A}}{=}& \left| \left(V \left(\Psi \left( x, \textstyle{\frac{1}{Q}} \sum z_i\right)
- \textstyle{\frac{1}{Q}} \sum_i \Psi (x,z_i)\right),\right.\right.\nonumber\\
&\qquad\qquad \left.\left. \bar\Psi \left(x, U \textstyle{\frac{1}{Q}} \sum z_i + 
V \Psi \left(\textstyle{\frac{1}{Q}} \sum z_i\right)\right) - \bar \Psi \left(x, 
U \textstyle{\frac{1}{Q}} \sum z_i + V \textstyle{\frac{1}{Q}}\sum \Psi (x, z_i)\right)\right)\right|\, .
\end{align*}
Thus,
\begin{align*}
|\hat\bef (x) - \beg (x)| \leq & \left(1+ \Lip (\bar\Psi)\right) |V| 
\left| \textstyle{\frac{1}{Q}} \sum \Psi (x,z_i) -\Psi \left(x, \textstyle{\frac{1}{Q}} \sum z_i\right)\right|.
\end{align*}
Observe that $|V|\leq |L-{\rm Id}| \leq C |\varkappa-\bar{\varkappa}|$. Moreover, with a simple Taylor
expansion around the point $(x,  \textstyle{\frac{1}{Q}} \sum z_i)$ we achieve
\begin{align*}
 & \left| \textstyle{\frac{1}{Q}} \sum \Psi (x,z_i) -\Psi \left(x, \textstyle{\frac{1}{Q}} \sum z_i\right)\right|
\leq C_0 \|D\Psi\|_0 \sum_i \left|z_i - \textstyle{\frac{1}{Q}} \sum z_i\right|
\leq C_0 \|D\Psi\|_0 \|\hat{u}\|_{C^0}\, .
\end{align*}
Since we have $\|D\Psi\|_0 \leq C_0 \|D\bar\Psi\|_0 + C_0 {\rm An}$ and 
$\|\hat{u}\|_{C^0}\leq \|\hat{f}\|_{C^0}
\leq C \|f\|_{C^0} + C_0 {\rm An}$ we conclude the pointwise estimate
\begin{align*}
|\hat\bef (x) - \beg (x)| \leq & C_0 {\rm An} (\|D\bar\Psi\|_0 + {\rm An}) (\|f\|_0 + {\rm An})\, ,
\end{align*}
which obviously implies \eqref{e:che_fatica}.

\subsection{Type B} In this case $\Psi = \bar\Psi$ and thus
\begin{equation}\label{e:Psi_via}
\|\hat\bef - \beg\|_{L^1} \leq C_0 (1+ \|D\bar\Psi\|_0) \|\etaa \circ \hat{u} - \p_\varkappa (\hat\bef)\|_{L^1}\, .
\end{equation}
Fix next an orthonormal base 
$e_1, \ldots, e_m, e_{m+1}, \ldots, e_{m+\bar{n}}$,
where the first $m$ vectors span $\pi$ and the remaining span $\varkappa$. We also assume that the rotation
$R$ acts on the plane spanned by $\{e_m, e_{m+1}\}$ and set $v= R(e_m) = a \,e_m + b\, e_{m+1}$
and $v_{m+1} = R (e_{m+1})$.  We then define two systems of coordinates: given $q\in\R^m\times R^{\bar{n}}$,
we write 
\begin{align*}
q =& \sum_{1\leq i \leq m-1} z_i (q) e_i + t(q) e_m + \tau (q) e_{m+1} + \sum_{2 \leq j\leq \bar{n}} y^j (q) e_{j+m}\\
=& \sum_i z_i (q) e_i + s(q) v_m + \sigma (q) v_{m+1} + \sum_j y^j (q) e_{j+m}\, .
\end{align*}
The first will be called $(t, \tau)$-coordinates and the second $(s, \sigma)$-coordinates.

We fix for the moment $x\in \R^{m-1}$ with $|x|\leq 4 $ and focus our attention on the interval 
$I_x = \{s: |(x, s)|\leq 6\}$.
We restrict the map $u$ to this interval and, by \cite[Proposition~1.2]{DS1}
we know that there
is a Lipschitz selection such that $u (x,s) = \sum_i \a{\theta_i (s)}$ in the $(s, \sigma)$-coordinates:
$\gr (\theta_i) = \{(x,s,\theta^1_i (s), \ldots, \theta^{\bar{n}}_i (s)):s\in I_x\}$. In the $(t,\tau)$ coordinates we can choose
functions $\vartheta_i$, also defined on an appropriate interval $J_x$, whose graphs coincide with the ones of the $\theta_i$.
We then obviously must have $\hat u (x,t) = \sum_i \a{\vartheta_i (t)}$ on the domain of definition of $\hat f$.
The coordinate functions $\theta^j_i$ and $\vartheta^j_i$ are linked by the following relations
\begin{equation}\label{e:relazione_cambio}
\begin{cases}
\Phi_i (t)=a\,t+b\,\vartheta^1_i(t), &\\
\theta^1_i(\Phi_i (t))=-b\,t+a\,\vartheta^1_i(t), &\\
\theta^l_i (\Phi_{i}(t))=\vartheta_i^l(t), & \text{for }\;l=2, \ldots, \bar n.
\end{cases}
\end{equation}
Observe that $\Lip(\Phi_i)\leq (1+C_0 |\pi - \bar\pi|) \leq 2$. Likewise we can assume that $\Lip(\Phi_i^{-1})\leq 2$.
Consider now $v (s) = \etaa\circ u (x,s) = \frac{1}{Q}
\sum_i \theta_i (s)$ and the corresponding $t\mapsto \hat{v} (t) = \p_\varkappa \circ \hat\bef (x,t)$, linked to 
$v = \etaa\circ u (x, \cdot)$ through a relation as in \eqref{e:relazione_cambio} with a 
corresponding map $\Phi$:
\begin{equation}\label{e:relazione_cambio_2}
\begin{cases}
\Phi (t)=a\,t+b\,\hat{v}^1(t), &\\
\frac{1}{Q} \sum_i \theta_i^1 (\Phi (t)) = v^1 (\Phi (t))=-b\,t+a\,\hat{v}^1(t),& \\
\frac{1}{Q} \sum_i \theta_i^l (\Phi (t)) = v^l (\Phi (t))=\hat{v}^l (t),& 
\text{for }\;l=2, \ldots, \bar n.
\end{cases}
\end{equation}
Moreover, write $\tilde{v} (t) = \frac{1}{Q} \sum_i \vartheta (t) = \etaa \circ \hat u (x,t) $.
We can then compute
 \begin{align}
&\etaa\circ \hat u(x,t) - \p_\varkappa (\hat\bef (x,t)) =  \tilde{v} (t)-\hat v (t)=
Q^{-1}\sum_i (\vartheta_i(t)- \hat v(t))\nonumber\\
= &Q^{-1} \sum_i
\Big(\underbrace{a^{-1}\theta_i^1(\Phi_i(t))-a^{-1}\theta_i^1(\Phi(t))
}_{1^{\text{st}}\;\text{component}},\ldots,\underbrace{
\theta_i^l(\Phi_i(t))-\theta_i^l(\Phi(t))}_{l^{\text{th}}\;\text{component}},
\ldots\Big).\label{e:componentwise}
\end{align}
This implies that
\begin{equation}\label{e.slice}
 |\etaa \circ \hat u (x, t)-\p_\varkappa (\hat\bef (x,t))| = |\tilde{v} (t) -\hat{v} (t)|
\leq C_0\sum_i \bigg|\int_{\Phi(t)}^{\Phi_i(t)} D\theta (\tau)\, d\tau\bigg|\, .
\end{equation}
Next we compute
\begin{equation}\label{e:Phi-Phi}
\Phi_i(t)-\Phi(t)=b\,\big(\vartheta_i^1(t)-\hat v^1 (t) \big)=
b\big(\vartheta_i^1 (t)-\tilde v^1(t)\big)+
b\big(\tilde v^1(t)-\hat{v}^1(t)\big).
\end{equation}
Since $|b|\leq C {\rm An}$, the terms in \eqref{e:Phi-Phi} can be estimated respectively as follows:
\begin{gather}
|b| |\vartheta_i^1 (t) - \tilde v^1 (t)| = |b| |\hat u_i^1(x, t)-(\etaa\circ \hat u)^1(t)|
\leq C_0 {\rm An}\,  \|\hat{u}\|_{C^0}, \nonumber\\
|\tilde v^1(t)-\hat{v}^1 (t)|\stackrel{\eqref{e.slice}}{\leq} \|D\theta\|_{L^\infty}\,\sum_{i=1}^Q |\Phi_i(t)-\Phi(t)|
 \leq C_0\,\Lip (u)\, \sum_{i=1}^Q |\Phi_i(t)-\Phi(t)|.\nonumber
\end{gather}
Recall that $\Lip (u)\leq \Lip (f) \leq C_0 \Lip (\hat{f}) + {\rm An}$.
Combining the last two inequalities with \eqref{e:Phi-Phi}, we therefore conclude, when $c_0$ is sufficiently small,
\begin{equation}\label{e.diff param2}
\sum_{i=1}^Q |\Phi_i(t)-\Phi(t)|\leq C_0 {\rm An} \|\hat{f}\|_{C^0} =: \rho\, .
\end{equation}
With this estimate at our disposal we can integrate \eqref{e.slice} in $t$ to conclude
\begin{equation*}\label{e.slice2}
\int_{J_x}| \tilde v (t) - \hat{v} (t)| \leq C_0 \int_{J_x}\int_{\Phi(t)-\rho}^{\Phi(t)+\rho}
|D\theta|(\tau)\,d\tau\,dt\leq
C_0 \int_{I_x} \int_{s-\rho}^{s+\rho}|D u| (x,\tau)\,d\tau\, ds\, ,
\end{equation*}
where in the latter inequality we have used the change of variables $s= \Phi (t)$ and the fact that
both the Lipschitz constants of $\Phi$ and its inverse are under control.
Integrating over $x$ and recalling that $\tilde v (t) - \hat{v} (t) = \etaa \circ \hat u (x,t) - \p_\varkappa (\hat\bef (x,t))$ we achieve
\begin{align}
&\int_{B_{4}} |\etaa\circ \hat u - \p_{\varkappa} \circ \hat\bef| \leq C_0 \int_{B_{4}} \int_{-\sqrt{36 -|x|^2}}^{\sqrt{36 
- |x|^2}} \int_{s-\rho}^{s+\rho} |Du | (x, \tau)\, d\tau\, ds\, dx\leq C_0 \rho \int_{B_{6 + 12\rho}} |D u| \nonumber\\
\leq & C_0 {\rm An} \|\hat f\|_{C^0} \left(\int_{B_{8}} |D u|^2\right)^{\sfrac{1}{2}}
\leq C_0 {\rm An} \big(\| f\|_{C^0} + {\rm An}\big) \D (f)^{\sfrac{1}{2}} \label{e:senza_Psi_OK}\, . 
\end{align}
Clearly \eqref{e:senza_Psi_OK} and \eqref{e:Psi_via} imply the desired estimate.

\subsection{Type C} Consider $\etaa\circ f$ and the $\xi: B_{4} (0, \pi)\to \pi^\perp$ such that
$\bG_\xi = \bG_{\etaa\circ f} \res \bC_{4} (0, \pi)$. We can then apply the argument
of the estimate for type B to conclude
\begin{equation}\label{e:tipo_C_1}
\|\etaa \circ\hat u - \p_\varkappa (\xi)\|_{L^1 (B_{4})} \leq
\|\etaa \circ \hat f - \xi\|_{L^1 (B_{4})} \leq C {\rm An} \big(\|f\|_{C^0} + {\rm An}\big) \D (f)^{\sfrac{1}{2}}\, .
\end{equation}
We need only to estimate
$\|\p_\varkappa (\xi) - \p_\varkappa (\hat\bef)\|_{L^1}$: since
$\beg (x) = (\etaa \circ \hat u (x), \Psi (x, \etaa\circ \hat u (x)))$ and 
$\hat\bef (x) = (\p_\varkappa (\hat \bef (x)), \Psi (x, \p_\varkappa (\hat \bef (x)))$, we can then estimate
\begin{align}
\|\hat\bef - \beg\|_{L^1} &\leq C_0 (1 + \|D\Psi\|_0) \left(\|\p_\varkappa (\hat\bef) - \p_\varkappa (\xi)\|_{L^1} +
\|\p_\varkappa (\xi) - \etaa \circ \hat{u}\|_{L^1}\right)\, . \label{e:tipo_C_2}
\end{align}
Define the maps $v,w$ and $w'$ as follows:
\begin{align*}
\bef (\bar x) &= 
\left(\etaa \circ u (\bar x) , \bar \Psi \left( \etaa \circ u (\bar x)\right)\right) =: (v (\bar x), w(\bar x)),\nonumber\\
\etaa \circ f (\bar x) &= 
\big(\etaa \circ u (\bar x), \textstyle{\frac{1}{Q}} \sum_i \bar\Psi (\bar x, u_i (\bar x))\big) =: (v (\bar x), w'
(\bar x))\, .
\end{align*}
Using the Lipschitz bound for $\bar{\Psi}$ we conclude 
\begin{equation}\label{e:taylorino}
\|\bef - \etaa \circ f\|_{C^0} = \|w-w'\|_0 \leq 
C \|D \bar \Psi\|_{C^0} \sum_i \left|u_i - \etaa\circ u\right|
\leq C \|D\bar\Psi\|_{C^0} \|f\|_{C^0}\, .
\end{equation} 
Consider
an orthogonal transformation 
\[
L = \left(
\begin{array}{ll}
U & V\\
W & Z
\end{array}
\right)\, 
\]
with the properties that $(\bar x,\bar z)\in \bar\pi \times \bar\varpi$ corresponds to $(U\bar x + V
\bar z, W\bar x + Z\bar z)\in \pi\times
\varpi$ and $|L-{\rm Id}|\leq C_0 {\rm An}$. We then have the following relations: 
$\p_\varkappa (\hat\bef (x)) = v (\Phi^{-1} (x))$ and
$\p_{\varkappa} (\xi(x)) = v ((\Phi')^{-1} (x))$, where $\Phi^{-1}$ and $(\Phi')^{-1}$ are the inverse,
respectively, of the maps $\Phi (\bar x) = U \bar x + V w (\bar x)$ and 
$\Phi' (\bar x) = U \bar x + V w' (\bar x)$.
Recalling that $|V|\leq |L-{\rm Id}|\leq C_0 {\rm An}$, we conclude that 
\[
|\Phi' (\bar x) - \Phi (\bar x)| \leq |V| \, |w (\bar x) - w' (\bar x)|\leq C_0 \|D\bar\Psi\|_{C^0} \|f\|_{C^0} {\rm An}
\quad\mbox{for every $\bar x$.}
\]
On the other hand we also know that $\Phi^{-1}$ has Lipschitz constant at most $2$ and so we achieve 
$|\Phi^{-1} (\Phi' (\bar x)) - \bar x| \leq C_0\|D\bar\Psi\|_{C^0} \|f\|_{C^0} {\rm An} $.
Being valid for any $\bar x$ we can apply
it to $\bar x = (\Phi')^{-1} (x)$ to conclude $|\Phi^{-1} (x) - (\Phi')^{-1} (x)| \leq C_0
\|D\bar\Psi\|_{C^0} \|f\|_{C^0} {\rm An}$.
Using then $\Lip (v) \leq \Lip (u) \leq c_0$,
we conclude the pointwise bound 
\[
|\p_\varkappa (\hat\bef (x)) - \p_\varkappa (\xi (x))| = |v (\Phi^{-1} (x))- v ((\Phi')^{-1} (x))|
\leq C_0\|D\bar\Psi\|_{C^0} \|f\|_{C^0} {\rm An} \, .
\]
After integrating in $x$, the latter bound combined with \eqref{e:tipo_C_1} and \eqref{e:tipo_C_2} gives the desired
estimate.

\bibliographystyle{plain}

\bibliography{references-CM}


\end{document}